\documentclass[11pt]{article}
\usepackage[english]{babel}
\usepackage{fullpage}
\usepackage[utf8]{inputenc}
\usepackage{amsmath,amssymb,amsthm}
\usepackage{mathtools}
\usepackage{mathrsfs}
\usepackage{graphicx}
\usepackage{xcolor}
\usepackage{cite}
\usepackage{float} 
\usepackage{subcaption}
\usepackage{wrapfig}
\usepackage{hyperref}
\usepackage{cancel}
\usepackage{caption}
\usepackage{subcaption}
\usepackage{tikz}
\usepackage{tikz-cd}
\usepackage[T1]{fontenc}

\pgfdeclarelayer{background}
\pgfsetlayers{background,main}

\makeatletter
\def\mathleft{\@fleqntrue\@mathmargin3em}
\def\mathcenter{\@fleqnfalse}
\makeatother

\let\amsamp=&
\let\comma=,
\let\lbracket=[
\let\rbracket=]

\usetikzlibrary{graphs,decorations.pathmorphing,decorations.markings,intersections,arrows.meta}

\newtheorem{theorem}{Theorem}[section]

\newtheorem{proposition}[theorem]{Proposition}
\newtheorem{definition}[theorem]{Definition}

\newtheorem{example}[theorem]{Example}
\newtheorem{remark}[theorem]{Remark}

\def\R{\mathbb{R}}
\def\C{\mathbb{C}}
\def\Z{\mathbb{Z}}

\def\M{\mathcal{M}}
\def\Cx{\mathbb{C}^{\times}}

\def\bim{\beta_{i^-}}
\def\d{\mathrm{d}}

\DeclareMathOperator{\Hom}{Hom}
\DeclareMathOperator{\Fun}{Fun}
\newcommand{\comp}{~\raisebox{1pt}{\tikz \draw[line width=0.6pt] circle(1.3pt);}~}
\definecolor{JK}{RGB}{0,178,229}
\definecolor{JY}{RGB}{154,205,50}
\definecolor{G}{RGB}{255,149,0}
\definecolor{M}{RGB}{246,42,51}
\definecolor{SA}{RGB}{238,134,167}
\definecolor{C}{RGB}{0,187,133}
\definecolor{Z}{RGB}{143,118,214}
\definecolor{JO}{RGB}{0,103,192}
\definecolor{JI+}{HTML}{faaf28}
\definecolor{JI}{RGB}{255,207,85}
\definecolor{I}{RGB}{0,106,184}
\definecolor{JT}{RGB}{246,139,30}
\definecolor{JA}{RGB}{46,139,87}
\definecolor{MO}{RGB}{0,54,134}

\hypersetup{
    colorlinks,
    linkcolor={MO},
    citecolor={MO},
    urlcolor={MO}
}

\title{Natural transformations between braiding functors\\ in the Fukaya category}
\author{Yujin Tong}
\date{}

\begin{document}
\maketitle

\begin{abstract}
We study the space of $A_\infty$-natural transformations between braiding functors acting on the Fukaya category associated to the Coulomb branch $\mathcal{M}(\bullet,1)$ of the $\mathfrak{sl}_2$ quiver gauge theory.
We compute all cohomologically distinct $A_\infty$-natural transformations $\mathrm{Nat}(\mathrm{id}, \mathrm{id})$ and $\mathrm{Nat}(\mathrm{id}, \beta_i^-)$, where $\beta_i^-$ denotes the negative braiding functor. 
Our computation is carried out in a diagrammatic framework compatible with the established embedding of the KLRW category into this Fukaya category.
We then compute the Hochschild cohomology of the Fukaya category using an explicit projective resolution of the diagonal bimodule obtained via the Chouhy--Solotar reduction system, and use this to classify all cohomologically distinct natural transformations.
These results determine the higher $A_\infty$-data encoded in the braiding functors and their natural transformations, and provide the first step toward a categorical formulation of braid cobordism actions on Fukaya categories.
\end{abstract}

\section{Introduction}

Braid group actions on triangulated categories are nontrivial and reveal a subtle interplay between topology and algebra. 
One interesting example is the action on Fukaya categories.

Conceptually, one considers a symplectic manifold $\mathcal{Y}_{\mathbf{a}}$ fibering over $\C$, determined by the data of $|\mathbf{a}|$ unordered points $\mathbf{a} = \{a_1, \dots, a_{|\mathbf{a}|}\}$ on $\C$, referred to as \emph{punctures}, whose fibers exhibit certain singular behaviors.

Fix a base point $\mathbf{x}$ in the configuration space of punctures $\mathrm{Conf}_{|\mathbf{x}|}(\C)/S_{|\mathbf{x}|}$. 
For a based loop in this configuration space, we obtain an $S^1$-family of symplectic manifolds equipped with fiberwise symplectic structures. 
The induced symplectic connection yields a monodromy symplectomorphism on the reference fiber $\mathcal{Y}_{\mathbf{x}}$. 
Viewing the braid group as the fundamental group of the configuration space of $|\mathbf{x}|$ unordered points in $\C$, we obtain a natural homomorphism
\begin{equation}
    \mathrm{Br}_{|\mathbf{x}|} = \pi_1(\mathrm{Conf}_{|\mathbf{x}|}(\C)/S_{|\mathbf{x}|})
    \longrightarrow 
    \pi_0(\mathrm{Aut}(\mathcal{Y}_{\mathbf{x}}, \omega)).
\end{equation}

Moreover, each symplectic automorphism of $\mathcal{Y}_{\mathbf{x}}$ induces corresponding endofunctors of its Fukaya category by transporting Lagrangian submanifolds and their Floer complexes. 
This gives rise to a representation
\begin{equation}
    \rho \colon
    \mathrm{Br}_{|\mathbf{x}|} = \pi_1(\mathrm{Conf}_{|\mathbf{x}|}(\C)/S_{|\mathbf{x}|})
    \longrightarrow 
    \pi_0(\mathrm{Aut}(\mathcal{Y}_{\mathbf{x}}))
    \longrightarrow 
    \mathrm{Aut}(\mathrm{Fuk}(\mathcal{Y}_{\mathbf{x}})),
\end{equation}
whose image consists of the \emph{braiding functors} on $\mathrm{Fuk}(\mathcal{Y}_{\mathbf{x}})$. 
For any two braiding functors $\beta_1$ and $\beta_2$, we can consider the space of $A_\infty$-natural transformations $\mathrm{Nat}(\beta_1, \beta_2)$ in the Fukaya category.
\begin{definition}
For the standard generators $\sigma_i$ of $\mathrm{Br}_{|\mathbf{x}|}$, we denote
\[
    \beta_{i^+} \coloneqq  \rho(\sigma_i), 
    \qquad
    \beta_{i^-} \coloneqq  \rho(\sigma_i^{-1}),
\]
and refer to them respectively as the \emph{positive} and \emph{negative braiding functors}.
\end{definition}

~

The punctures may arise, for instance, from singular fibers or from the critical values of a superpotential. 
Instances of the former appear in \cite{KT}, where fibrations with singular fibers are analyzed. 
In the present paper, we focus on the latter situation, in which the punctures specify a superpotential. 
More specifically, we consider the symplectic fibration
\[
\pi\colon \C_x\times \Cx_y \to \C,
\qquad \pi(x,y)=x.
\]
For a collection of punctures $\mathbf{x}=\{x_1,\dots,x_{|\mathbf{x}|}\}\in\mathrm{Conf}_{|\mathbf{x}|}(\C_x)/S_{|\mathbf{x}|}$, the associated Landau--Ginzburg superpotential $\mathcal{W}_{\mathbf{x}}$ is given by
\begin{equation}\label{eq:Wx}
  \mathcal{W}_{\mathbf{x}} = y \prod_{i=1}^{n}(x - x_i),
\end{equation}
which determines the fiberwise stops at $\mathcal{W}_{\mathbf{x}}\to +\infty$. 
We will consider the manifold with two stops in the base.

This space has a natural origin as the Coulomb branch $\M(\bullet,1)$ of the $\mathfrak{sl}_2$ quiver gauge theory with dimension vector $\vec{d}=1$, 
which plays a central role in Aganagic's categorification of Khovanov cohomology~\cite{Mina}. We will consider a full subcategory $\mathrm{Fuk}_{|||}(\M (\bullet,1), \mathcal{W}_{\mathbf{x}})\subset \mathrm{Fuk}(\M (\bullet,1), \mathcal{W}_{\mathbf{x}})$ generated by objects associated to curves without closed components \cite{Mina, ADLSZ, Elise}, which contains all the interested objects.

~

The braid group itself extends naturally to the \emph{braid cobordism category} $\mathsf{BrCob}_{|\mathbf{x}|}$, whose objects are braid group elements (tangles from $|\mathbf{x}|$ points to themselves) and whose morphisms are braid cobordisms between them. 
We expect the correspondence $\rho$ to lift to a functor
\begin{equation}
    \rho^{\sharp} \colon 
    \mathsf{BrCob}_{|\mathbf{x}|} 
    \longrightarrow 
    \mathrm{End}(\mathrm{Fuk}(\mathcal{Y}_{\mathbf{x}})),
\end{equation}
sending each braid cobordism to a natural transformation between the corresponding braiding functors.

In this paper, however, we will not determine the functor $\rho^{\sharp}$. 
Instead, we compute all (cohomologically distinct) natural transformations 
$\mathrm{Nat}(\mathrm{id},\mathrm{id})$ and $\mathrm{Nat}(\mathrm{id},\bim)$. 
This problem is interesting in its own right and reveals several elegant structures along the way. 
Once the natural transformations from the identity functor to the generating braiding functors are understood, 
their compositions yield much of the information about natural transformations between arbitrary braiding functors.

\begin{remark}
There are several reasons to focus on negative rather than positive braiding functors. 
While ordinary natural transformations from $\mathrm{id}$ to $\beta_{i^+}$ can be constructed easily, 
no such transformation exists from $\mathrm{id}$ to $\bim$.

For the Lagrangian $I_{i-1}$ introduced in Section~\ref{sect:Ibrane}, after applying the positive braiding, there is a degree 0 intersection point in $\Hom(I_{i-1},\beta_{i^+}I_{i-1})$, and its corresponding natural transformation is easy to determine.
However, for the negative braiding, one finds a degree 2 morphism 
$a \in \Hom^2(I_{i-1}, \bim I_{i-1})$.
If this morphism is to arise from a natural transformation, then there should exist a nontrivial $\eta:\mathrm{id}\Rightarrow \bim$ of degree $|\eta|=2$. 
\begin{figure}[H]
  \centering
  \begin{tikzpicture}[scale=1.3]
  \foreach \x in {0,1,2,3}{
    \node[text=M,scale=1.7] at (\x,0) {$\ast$};
  }
  \draw[JI,line width=1pt](0,0)--node[below,midway]{$I_{i-1}$} (1,0);
  \draw[JK,line width=1pt](0,0) to[out=45,in=90,looseness=2]node[above,midway]{$\bim I_{i-1}$}(2,0);
  \draw[->,SA!50,line width=1pt] (1.1,0.1) arc[start angle=180,end angle=0,x radius=0.4cm,y radius=0.4cm];
  \node[SA] at (1.5,0) {$\bim$};
  \draw[->,SA!50,line width=1pt] (1.9,-0.1) arc[start angle=0,end angle=-180,x radius=0.4cm,y radius=0.4cm];
  \fill (0,0) circle (1pt);
  \node at (-0.3,-0.3) {\small$a$};
  \end{tikzpicture}
  \caption{A degree-2 element in $\Hom(I_{i-1}, \bim I_{i-1})$}
\end{figure}
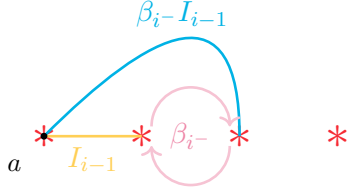
These indicates that higher structures must be taken into account, 
and that one should work with $A_\infty$-natural transformations. 
We will be able to show that this is indeed the case.
\end{remark}

\paragraph{Acknowledgements.}
The author would like to thank Mina Aganagic and Peng Zhou for invaluable guidance and discussions, 
and Marco David for many helpful conversations and for first obtaining the computation in Theorem~\ref{thm:Marco}.

\tableofcontents

\section{Structure of the Fukaya category}
\subsection{Generating objects}
Our interested subcategory $\mathrm{Fuk}_{|||}(\M (\bullet,1), \mathcal{W}_{\mathbf{x}})$, proved in \cite{Elise}, is generated by the $T_i$ objects, whose shapes in the base are the straight vertical lines between punctures drawn in Fig.~\ref{fig:Ti_base}, and whose shapes in the fibers are the straight lines from $y=0$ to $y=\infty$ drawn in Fig.~\ref{fig:Ti_fiber}. There are $|\mathbf{x}|+1$ generators, we label them as $T_0$, $T_1$, \dots, $T_{|\mathbf{x}|}$ according to Fig.~\ref{fig:Ti}.

\begin{figure}[H]
    \centering
    \begin{subfigure}[b]{0.49\textwidth}
    \centering
    \begin{tikzpicture}[scale=1.3, every node/.style={scale=1}, line width=0.8pt]
      \draw[black,thick] (0,0) ellipse [x radius=2, y radius=1];

      \foreach \x in {-1.2, -0.4, 0.4, 1.2} {
        \node[text=M,scale=1.5] at (\x,0) {$*$};
      }

      \begin{scope}
        \clip (0,0) ellipse [x radius=2, y radius=1];
        \fill[M] (-1.97,0) ellipse [x radius=0.08, y radius=0.2];
        \fill[M] ( 1.97,0) ellipse [x radius=0.08, y radius=0.2];
      \end{scope}

      \foreach [count=\i] \x in {-1.6, -0.8, 0.0, 0.8, 1.6} {
        \pgfmathtruncatemacro{\n}{\i - 1}
        \pgfmathsetmacro{\ymax}{sqrt(1 - (\x*\x)/4)}
        \draw[JK,thick] (\x,-\ymax) -- (\x,\ymax);
        \node[JK,above=1pt] at (\x,\ymax) {$T_{\n}$};
      }

    \end{tikzpicture}
    \caption{Base}
    \label{fig:Ti_base}
    \end{subfigure}
    \begin{subfigure}[b]{0.49\textwidth}
    \centering
	  \begin{tikzpicture}[scale=0.6, every node/.style={scale=1}, line width=0.8pt]
      \draw (-8.5,2.5) -- (0,2.5);
      \draw (-8.5,0) -- (0,0);
      \draw (-8.5,1.25) ellipse (0.25 cm and 1.25 cm);
      \draw (0,0) arc[
          start angle=-90,
          end angle=90,
          x radius=0.25cm,
          y radius=1.25cm
          ];
        \draw[dashed] (0,2.5) arc[
          start angle=90,
          end angle=270,
          x radius=0.25cm,
          y radius=1.25cm
          ];
      \draw[thick, JK, name path=T1] (.25,1) node[right] {$T_i$} -- (-8.25,1);
      \fill[M] (0,2.5) circle (4pt);
      \node[above] at (0,3) {$y=\infty$};
      \node[above] at (-8.5,3) {$y=0$};
    \end{tikzpicture}
    \caption{Fiber}
    \label{fig:Ti_fiber}
    \end{subfigure}
    \caption{The generators $T_i$, the red dots and stars denote the stops and punctures}
    \label{fig:Ti}
\end{figure}
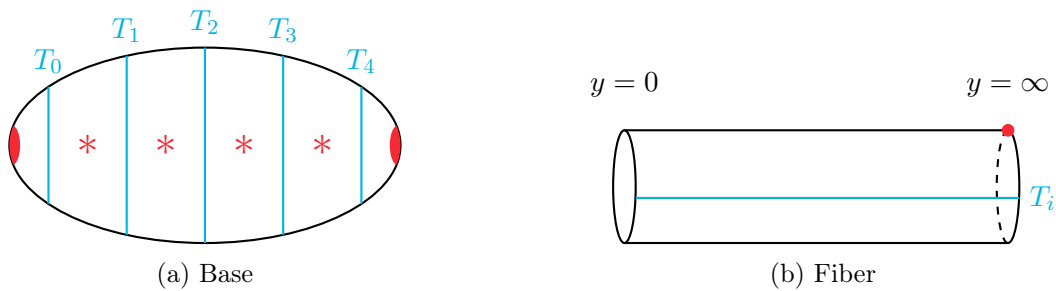

\begin{remark}
  Lagrangians with such shape in the fibers will be called $T$-branes, in the future we will simply draw the base and label the fiber type.
\end{remark}

\subsection{Morphisms between generators}
\label{sect:ajisk}

For generating objects $T_i$, $T_j$, the morphism space is given by their Floer complex
\begin{equation}
  \Hom ^{\bullet}(T_i,T_j)\coloneqq CF^{\bullet}(T_i,T_j)
\end{equation}

To compute this, we need to positively (counterclockwise) wrap the first input. The result for $\Hom (T_3,T_1)$ is shown in Fig.~\ref{fig:Mor}, where there is a unique intersection in the base, and infinitely many intersection points, indexed by $\Z_{\ge 0}$, above it in the fiber.

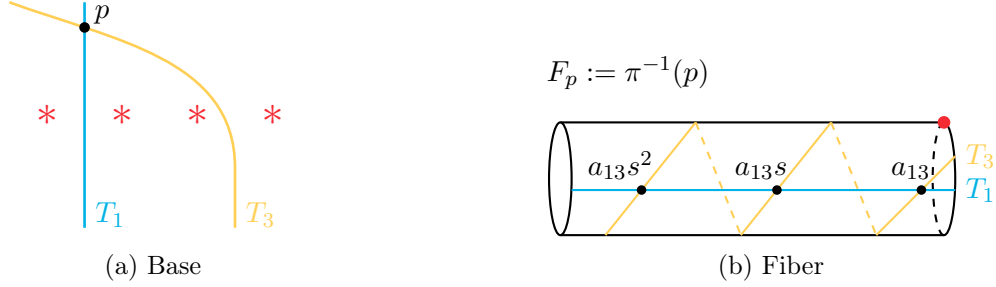
\begin{figure}[H]
    \centering
	\begin{subfigure}[b]{0.49\textwidth}
    \centering
    \begin{tikzpicture}[scale=1]
      \foreach \x in {0,1,2,3}{
    \node[text=M,scale=1.5] at (\x,0) {$\ast$};
  }
  \begin{scope}[xshift=4.5cm]
  \draw[thick, JI, name path=T3, line width=1pt] (-2,-1.5)node[right,yshift=5pt] {$T_3$}--(-2,-0.7) to[out=90,in=-20](-5,1.5);
  \draw[thick, JK, name path=T1, line width=1pt] (-4,-1.5)node[right,yshift=5pt] {$T_1$}--(-4,1.5);
  \fill[name intersections={of=T1 and T3}] (intersection-1) circle (2pt) node[right, yshift=5pt] {$p$};
  \end{scope}
  \end{tikzpicture}
  \caption{Base}
  \end{subfigure}
  \begin{subfigure}[b]{0.49\textwidth}
  \centering
	\begin{tikzpicture}[scale=.6,line width=0.8pt]
		\draw (-8.5,2.5) -- (0,2.5);
		\draw (-8.5,0) -- (0,0);
		\draw (-8.5,1.25) ellipse (0.25 cm and 1.25 cm);
		\draw (0,0) arc[
				start angle=-90,
				end angle=90,
				x radius=0.25cm,
				y radius=1.25cm
				];
	    \draw[dashed] (0,2.5) arc[
				start angle=90,
				end angle=270,
				x radius=0.25cm,
				y radius=1.25cm
				];
		\draw[thick, JK, name path=T1] (.25,1) node[right] {$T_1$} -- (-8.25,1);
        \draw[thick, JI, name path=T3a] (0.25,1.75) node[right] {$T_3$} -- (-1.5,0);
        \draw[thick, JI, dashed] (-2.5,2.5) -- (-1.5,0);
        \draw[thick, JI, name path=T3b] (-2.5,2.5) -- (-4.5,0);
        \draw[thick, JI, dashed] (-5.5,2.5) -- (-4.5,0);
        \draw[thick, JI, name path=T3c] (-5.5,2.5) -- (-7.5,0);
        \fill[M] (0,2.5) circle (4pt);
        \node[above] at (-7, 3) {$F_p\coloneqq \pi^{-1}(p)$};
        \path [name intersections={of=T1 and T3a, by={i1}}];
        \path [name intersections={of=T1 and T3b, by={i2}}];
        \path [name intersections={of=T1 and T3c, by={i3}}];

        \fill (i1) circle (3pt) node[above, xshift=-4pt] {$a_{13}$};
        \fill (i2) circle (3pt) node[above, xshift=-6pt] {$a_{13}s$};
        \fill (i3) circle (3pt) node[above, xshift=-8pt] {$a_{13}s^2$};
    \end{tikzpicture}
        \caption{Fiber}
        \label{fig:Mor_fiber}
    \end{subfigure}
    \caption{Morphisms between two generating $T_i$ objects.}
    \label{fig:Mor}
\end{figure}

The same pattern holds for arbitrary $T_i$ and $T_j$, and all these intersection points have degree 0. 
\begin{definition}
We label the $(\alpha+1)$st intersection point, counted from the $y\to\infty$ side of the cylindrical fiber, by $a_{ji}s^\alpha$.
\end{definition}
As we will see in Section~\ref{sect:KLRW}, this corresponds to a KLRW strand diagram from $j$ to $i$ with $\alpha$ dots.

\begin{remark}
  We write $a_{ji}$ for morphisms in $\Hom(T_i, T_j)$, following the common convention that morphisms are composed from right to left. 
  With this notation, an element of $\Hom(T_j, T_k) \otimes \Hom(T_i, T_j)$ appears as $a_{kj}\otimes a_{ji}$, which is notationally more convenient.
\end{remark}

\begin{theorem}
For any $i,j$, the morphism space is given by
\begin{equation}
  \Hom(T_i, T_j) \cong \Bbbk ^{\Z_{\ge 0}}=\bigoplus_{\alpha=0}^\infty \Bbbk\, a_{ji} s^\alpha.
\end{equation}
\end{theorem}
\begin{definition}\label{def:spq}
  We will use the following shorthand notations:
  \begin{itemize}
    \item $a_{ji}\coloneqq a_{ji}s^0$, for $i\ne j$
    \item $s_i^\alpha\coloneqq a_{ii}s^\alpha$, $s_i\coloneqq a_{ii}s^1$, $e_i\coloneqq a_{ii}s^0$,
    \item $p_i\coloneqq a_{i,i-1}$, $q_i\coloneqq a_{i,i+1}$.
  \end{itemize}
\end{definition}

\subsection{Structure maps}
\subsubsection{Composition of morphisms}
In the Fukaya category, the structure maps
\begin{equation}
  \mu^d\colon \Hom (X_{d-1},X_d)\otimes\cdots\otimes\Hom (X_0,X_1)\to \Hom (X_0,X_d)[2-d]
\end{equation}
are defined by counting pseudoholomorphic polygons bounded by the objects $X_0,\dots,X_d$.
They satisfy the $A_\infty$-relations
\begin{equation}
    \sum_{m,n}
    (-1)^{\maltese_n}
    \mu^{d-m+1}
    (a_d, \dots, a_{n+m+1}, 
     \mu^m(a_{n+m}, \dots, a_{n+1}), 
     a_n, \dots, a_1)
    = 0
\end{equation}
for all composable morphisms $a_i \in \Hom(X_{i-1}, X_i)$, where $\maltese_n=|a_1|+\cdots+|a_n|-n$ is the Koszul sign convention.

It is known that when restricted to the generating objects $\{T_i\}$, the compositions $\mu^d$ satisfy relations that coincide with those of the Khovanov--Lauda--Rouquier--Webster (KLRW) algebra \cite{KL1, KL2, Rouquier, KLRW}.
Namely,
\begin{theorem}
\cite{ADLSZ}
For the generators $T_i$, one has
\[
  \mu^d = 0 \quad \text{for } d \ne 2,
\]
while the binary product $\mu^2$ is given by
\begin{equation}
  \mu ^2 (  a_{kj} s^{\beta}, a_{ji}s^{\alpha} ) = a_{ki} s^{\beta+\alpha +\delta (i,j,k)}
  \label{eq:mu2}
\end{equation}
where
\begin{align}
\delta(i,j,k)
&=
\frac{1}{2}\left( |i-j|+|j-k|-|i-k| \right)\notag\\
&=
\begin{cases}
0, & (i-j)(j-k)\ge 0,\\
\min\{|i-j|,|j-k|\}, & (i-j)(j-k)<0.
\end{cases}
  \label{eq:deltaijk}
\end{align}
\end{theorem}

\begin{proof}
These can be proved by a direct count of holomorphic discs. For example,

\begin{figure}[H]
    \centering
	\begin{subfigure}[b]{0.49\textwidth}
    \centering
    \begin{tikzpicture}[scale=1.3]
      \foreach \x in {-0.5,0.5,1.5}{
    \node[text=M,scale=1.5] at (\x,0) {$\ast$};
  }
  \draw[thick, JI, name path=T3, line width=1pt] (2,-1)--(-0.5,1.5)node[above,yshift=5pt] {$T_2$};
  \draw[thick, JY, name path=T1, line width=1pt] (0,-1.5) --(0,1.5)node[above,yshift=5pt]{$T_1$};
  \draw[thick, JK, name path=T1', line width=1pt] (-0.5,-1.5)--(1,1.5)node[above,yshift=5pt] {$T_1$};
        \path [name intersections={of=T1 and T3, by={i1}}];
        \path [name intersections={of=T1' and T3, by={i2}}];
        \path [name intersections={of=T1' and T1, by={i3}}];
          \begin{pgfonlayer}{background}
            \fill[JO!30, fill opacity=0.35] (i1) -- (i2) -- (i3) -- cycle;
        \end{pgfonlayer}
  \fill[name intersections={of=T1 and T3}] (intersection-1) circle (1.5pt);
  \fill[name intersections={of=T1 and T1'}] (intersection-1) circle (1.5pt)node[right]{\small$s$};
  \fill[name intersections={of=T1' and T3}] (intersection-1) circle (1.5pt)node[right]{\small$s$};
  \end{tikzpicture}
  \caption{Base}
  \end{subfigure}
  \begin{subfigure}[b]{0.49\textwidth}
  \centering
	\begin{tikzpicture}[scale=.7]
		\draw (-8.5,2.5) -- (0,2.5);
		\draw (-8.5,0) -- (0,0);
		\draw (-8.5,1.25) ellipse (0.25 cm and 1.25 cm);
		\draw (0,0) arc[
				start angle=-90,
				end angle=90,
				x radius=0.25cm,
				y radius=1.25cm
				];
	    \draw[dashed] (0,2.5) arc[
				start angle=90,
				end angle=270,
				x radius=0.25cm,
				y radius=1.25cm
				];
        
        \draw[thick, JK!50, dashed] (-1.75,0) -- (-1.25,2.5);
        \draw[thick, JK!50, dashed] (-3.75,2.5) -- (-4.25,0);
        \draw[thick, JK!50, dashed] (-6.25,2.5) -- (-6.75,0);

        \draw[thick, JI, name path=ji] (0.098,2.4) node[right] {$T_2$} -- (-2.5,0);
        \draw[thick, JI!50, dashed] (-3,2.5) -- (-2.5,0);
        \draw[thick, JI] (-3,2.5) -- (-5,0);
        \draw[thick, JI!50, dashed] (-5.5,2.5) -- (-5,0);
        \draw[thick, JI, name path=T3c] (-5.5,2.5) -- (-7.5,0);

        \draw[thick, JK] (0.24,0.9) node[right,yshift=-3pt] {$T_1$}  -- (-1.25,2.5);
        \draw[thick, JK, name path=jk] (-1.75,0) -- (-3.75,2.5);
        \draw[thick, JK] (-4.25,0) -- (-6.25,2.5);
        \draw[thick, JK] (-6.75,0) -- (-8.2895,1.9244);
        \fill[M] (0,2.5) circle (3pt);
	    	\draw[thick, JY, name path=jy] (.2498,1.3) node[right] {$T_1$} -- (-8.2502,1.3);

        \path [name intersections={of=ji and jy, by={i1}}];
        \path [name intersections={of=jy and jk, by={i2}}];
        \path [name intersections={of=jk and ji, by={i3}}];
        \begin{pgfonlayer}{background}
            \fill[JO!33, fill opacity=0.35] (i1) -- (i2) -- (i3) -- cycle;
        \end{pgfonlayer}
        \fill (i1) circle (2.5pt);
        \fill (i2) circle (2.5pt) node[above, xshift=3pt] {\small$s$};
        \fill (i3) circle (2.5pt) node[left] {\small$s$};
    \end{tikzpicture}
        \caption{Fiber}
    \end{subfigure}
    \caption{Proof of $\mu^2(s_1,a_{12})=a_{12}s$.}
\end{figure}
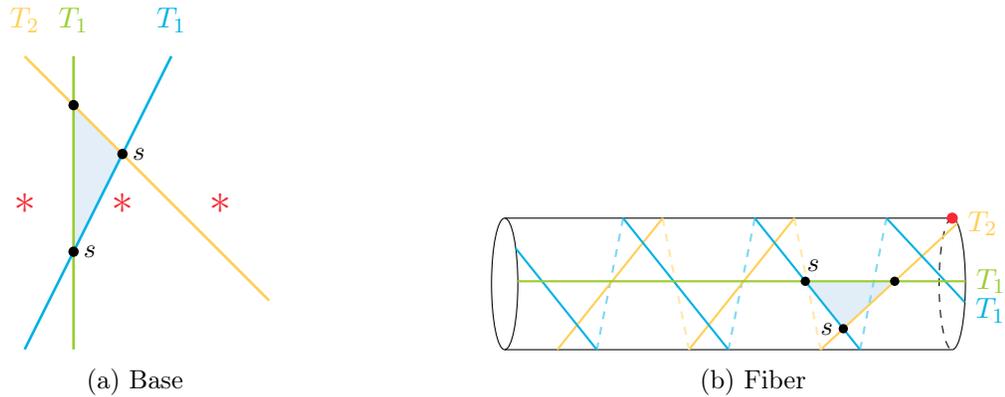

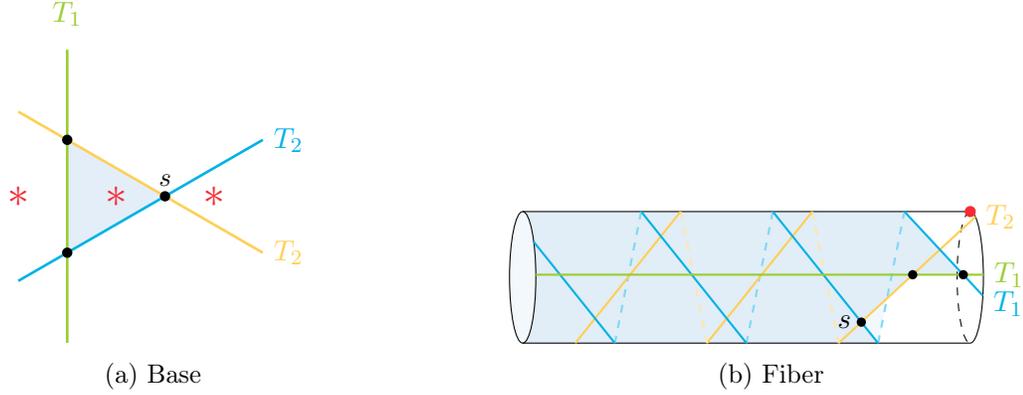
\begin{figure}[H]
    \centering
	\begin{subfigure}[b]{0.49\textwidth}
    \centering
    \begin{tikzpicture}[scale=1.3]
      \foreach \x in {-0.5,0.5,1.5}{
    \node[text=M,scale=1.5] at (\x,0) {$\ast$};
  }
  \draw[thick, JI, name path=T3, line width=1pt] (2,-0.57735)node[right] {$T_2$}--(-0.5,0.86603);
  \draw[thick, JY, name path=T1, line width=1pt] (0,-1.5) --(0,1.5)node[above,yshift=5pt]{$T_1$};
  \draw[thick, JK, name path=T1', line width=1pt] (-0.5,-0.86603)--(2,0.57735)node[right] {$T_2$};
        \path [name intersections={of=T1 and T3, by={i1}}];
        \path [name intersections={of=T1' and T3, by={i2}}];
        \path [name intersections={of=T1' and T1, by={i3}}];
          \begin{pgfonlayer}{background}
            \fill[JO!33, fill opacity=0.33] (i1) -- (i2) -- (i3) -- cycle;
        \end{pgfonlayer}
  \fill[name intersections={of=T1 and T3}] (intersection-1) circle (1.5pt);
  \fill[name intersections={of=T1 and T1'}] (intersection-1) circle (1.5pt);
  \fill[name intersections={of=T1' and T3}] (intersection-1) circle (1.5pt)node[above]{\small$s$};
  \end{tikzpicture}
  \caption{Base}
  \end{subfigure}
  \begin{subfigure}[b]{0.49\textwidth}
  \centering
	\begin{tikzpicture}[scale=.7]
		\draw (-8.5,2.5) -- (0,2.5);
		\draw (-8.5,0) -- (0,0);
		\draw (-8.5,1.25) ellipse (0.25 cm and 1.25 cm);
		\draw (0,0) arc[
				start angle=-90,
				end angle=90,
				x radius=0.25cm,
				y radius=1.25cm
				];
	    \draw[dashed] (0,2.5) arc[
				start angle=90,
				end angle=270,
				x radius=0.25cm,
				y radius=1.25cm
				];

        \draw[thick, JK!50, dashed, name path=jk1] (-1.75,0) -- (-1.25,2.5);
        \draw[thick, JK!50, dashed] (-3.75,2.5) -- (-4.25,0);
        \draw[thick, JK!50, dashed] (-6.25,2.5) -- (-6.75,0);

        \draw[thick, JI, name path=ji] (0.098,2.4) node[right] {$T_2$} -- (-2.5,0);
        \draw[thick, JI!50, dashed] (-3,2.5) -- (-2.5,0);
        \draw[thick, JI] (-3,2.5) -- (-5,0);
        \draw[thick, JI!50, dashed] (-5.5,2.5) -- (-5,0);
        \draw[thick, JI, name path=T3c] (-5.5,2.5) -- (-7.5,0);
        \draw[thick, JK, name path=jk0] (0.24,0.9) node[right,yshift=-3pt] {$T_1$}  -- (-1.25,2.5);
        \draw[thick, JK, name path=jk] (-1.75,0) -- (-3.75,2.5);
        \draw[thick, JK] (-4.25,0) -- (-6.25,2.5);
        \draw[thick, JK] (-6.75,0) -- (-8.2895,1.9244);
        \fill[M] (0,2.5) circle (3pt);
	    	\draw[thick, JY, name path=jy] (.2498,1.3) node[right] {$T_1$} -- (-8.2502,1.3);

        \path [name intersections={of=ji and jy, by={i1}}];
        \path [name intersections={of=jy and jk0, by={i2}}];
        \path [name intersections={of=jk and ji, by={i3}}];
        \path [name intersections={of=jk1 and ji, by={i4}}];
        \begin{pgfonlayer}{background}
        \fill[JO!33, fill opacity=0.35] (-1.25,2.5) -- (i2) -- (i1) -- (i3) -- (-1.75,0) -- (-8.5,0) arc[
				start angle=-90,
				end angle=90,
				x radius=0.25cm,
				y radius=1.25cm
				] -- cycle;
        \fill[JO!13, fill opacity=0.35] (-8.5,1.25) ellipse (0.25 cm and 1.25 cm);
        \fill[JO!13, fill opacity=0.35] (i3) -- (i4) -- (-1.75,0) -- cycle;
        \end{pgfonlayer}
        \fill (i1) circle (2.5pt);
        \fill (i2) circle (2.5pt);
        \fill (i3) circle (2.5pt) node[left] {$s$};
    \end{tikzpicture}
        \caption{Fiber}
    \end{subfigure}
    \caption{Proof of $\mu^2(p_2,q_1)=s_2$.}
    \label{fig:pq}
\end{figure}

Note that in Fig.~\ref{fig:pq}, the boundary of the disc in the fiber must rotate by one full turn relative to the stop. This occurs because the projection of the boundary to the base winds once around a puncture; 
according to the superpotential \eqref{eq:Wx}, moving around a puncture induces an opposite rotation of the stop in the fiber. Hence, for the boundary to close up and actually bound a disc, it must acquire an additional rotation in the fiber that cancels the rotation of the stop.

\end{proof}
We will also write $a_2\cdot a_1\coloneqq \mu^2(a_2,a_1)$ whenever no confusion can arise.

\subsubsection{KLRW category $\mathcal{C}$ inside the Fukaya category}
\label{sect:KLRW}

The KLRW algebra admits a convenient diagrammatic presentation, 
where generators and their compositions are represented by strands, dots, and crossings.
In this paper, we restrict to the $\mathfrak{sl}_2$ case with $d=1$, 
for which the diagrams and relations is much simplified.

\begin{definition}
  Fix a collection $F$ of $n=|\mathbf{x}|$ red points on $\R$.
  The KLRW category $\mathcal{C}$ we consider is defined as follows:
  \begin{itemize}
    \item \textbf{Objects:} configurations of a single black point on $\R$, distinct from the points in $F$.
    \item \textbf{Morphisms:} generated by strand diagrams in the strip $\R \times [0,1]$, 
    with no horizontal or vertical tangencies and only generic intersections.
    Black strands may carry dots as decorations.
    \item \textbf{Composition:} given by vertical stacking of diagrams. We read and compose the diagrams from the bottom to the top.
    \item \textbf{Relations:} diagrams are considered up to isotopy and subject to the local relations shown in Fig.~\ref{fig:KLRW}.
  \end{itemize}
\end{definition}

\begin{figure}[H]
  \centering
  \begin{subfigure}[b]{0.45\textwidth}
  \centering
	\begin{tikzpicture}[scale=.7]
  \foreach \x in {0.5,1.5,3.5,4.5,6.5,7.5} {
    \draw[M,line width=1pt] (\x,-1.5) -- (\x,1.5);
  }

  \draw[black,line width=1pt]
    (1,-1.5) -- (1,-1.3) to[out=90,in=-90] (0,0)
    to[out=90,in=-90] (1,1.3) -- (1,1.5);
  \node at (2.5,0) {$=$};
  \draw[black,line width=1pt]
    (4,-1.5) -- (4,1.5);
  \fill[black] (4,0) circle (3pt);
  \node at (5.5,0) {$=$};
  \draw[black,line width=1pt]
    (7,-1.5) -- (7,-1.3) to[out=90,in=-90] (8,0) 
    to[out=90,in=-90] (7,1.3) -- (7,1.5);
  \end{tikzpicture}
  \end{subfigure}
  \begin{subfigure}[b]{0.26\textwidth}
  \centering
	\begin{tikzpicture}[scale=.7]
  \foreach \x in {0.5,3.5} {
    \draw[M,line width=1pt] (\x,-1.5) -- (\x,1.5);
  }
  \draw[black,line width=1pt]
    (0,-1.5) -- (0,-.65) to[out=90,in=-90] (1,0.65) -- (1,1.5);
  \fill[black] (1,1.1) circle (3pt);
  \node at (2,0) {$=$};
  \draw[black,line width=1pt]
    (3,-1.5) -- (3,-.65) to[out=90,in=-90] (4,0.65) -- (4,1.5);
  \fill[black] (3,-1.1) circle (3pt);
  \end{tikzpicture}
  \end{subfigure}
  \begin{subfigure}[b]{0.26\textwidth}
  \centering
	\begin{tikzpicture}[scale=.7]
  \foreach \x in {0.5,3.5} {
    \draw[M,line width=1pt] (\x,-1.5) -- (\x,1.5);
  }
  \draw[black,line width=1pt]
    (1,-1.5) -- (1,-.65) to[out=90,in=-90] (0,0.65) -- (0,1.5);
  \fill[black] (0,1.1) circle (3pt);
  \node at (2,0) {$=$};
  \draw[black,line width=1pt]
    (4,-1.5) -- (4,-.65) to[out=90,in=-90] (3,0.65) -- (3,1.5);
  \fill[black] (4,-1.1) circle (3pt);
  \end{tikzpicture}
  \end{subfigure}
  \caption{Nontrivial KLRW relations in $\mathcal{C}_{\bullet,1,F}$}
  \label{fig:KLRW}
\end{figure}
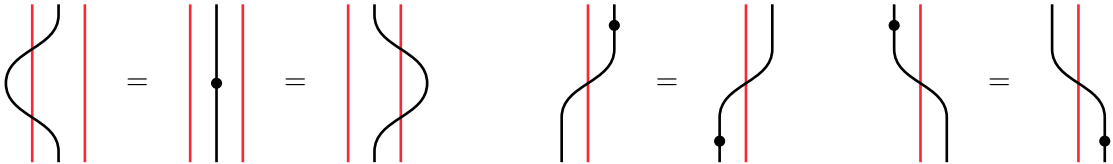

\begin{remark}
  More generally, one can define a KLRW category for any quiver $\Gamma$ with dimension vector $\vec{d}$. 
  The category $\mathcal{C}_{\Gamma,\vec{d},F}$ corresponds precisely to the Fukaya category of the multiplicative Coulomb branch $\M(\Gamma,\vec{d})$ for ADE type Lie algebras. 
  The simplified case considered above is $\mathcal{C}_{\bullet,1,F}$, which corresponds to Aganagic's $\mathfrak{sl}_2$ theory with $d=1$. 
  We will recall the general definition when necessary.
\end{remark}

The embedding 
$\mathcal{C}\hookrightarrow \mathrm{Fuk}_{|||}(\M (\bullet,1), \mathcal{W}_{\mathbf{x}})$ 
is defined as follows:
\begin{itemize}
  \item The red points $F$ correspond to the punctures $\mathbf{x}$ on the base of $\M (\bullet,1)$. For example, we can choose $F=\mathrm{Re}(\mathbf{x})$.  
  
  \item \textbf{Objects:} the position of the black point specifies the location of the brane $T_i$. For instance, the KLRW object shown at the top of Fig.~\ref{fig:EmbOb} corresponds to $T_1$ in the Fukaya category.
  
  \item \textbf{Morphisms:} a black strand connecting $T_i$ to $T_j$, which does not cross any red strand more than once and is decorated with $\alpha$ dots, corresponds to the morphism $a_{ji}s^\alpha\in\Hom(T_i,T_j)$ in the Fukaya category.
  
  \item \textbf{Composition:} one can directly verify that the composition in the KLRW category, under the correspondence above, reproduces the structure maps~(\ref{eq:mu2}). Some relations follow automatically from diagram isotopy, while others arise from the defining relations shown in Fig.~\ref{fig:KLRW}.
\end{itemize}

\begin{figure}[H]
    \centering
    \begin{subfigure}[b]{0.49\textwidth}
    \centering
    \begin{tikzpicture}[scale=1, every node/.style={scale=1}, line width=1pt]
		\draw[gray] (-2,3) -- (2,3);
    \foreach \x in {-1.2, -0.4, 0.4, 1.2} {
        \fill[M] (\x,3) circle (2pt);
      }
    \fill[black] (-0.8,3) circle (2pt);
    \draw[<->] (0,2.5) -- (0,1.5);
     \draw[black,thick] (0,0) ellipse [x radius=2, y radius=1];

      \foreach \x in {-1.2, -0.4, 0.4, 1.2} {
        \node[text=M,scale=1.5] at (\x,0) {$*$};
      }

      \begin{scope}
        \clip (0,0) ellipse [x radius=2, y radius=1];
        \fill[M] (-1.97,0) ellipse [x radius=0.08, y radius=0.2];
        \fill[M] ( 1.97,0) ellipse [x radius=0.08, y radius=0.2];
      \end{scope}

      \foreach [count=\i] \x in {-0.8} {
        \pgfmathsetmacro{\ymax}{sqrt(1 - (\x*\x)/4)}
        \draw[JK,thick] (\x,-\ymax) -- (\x,\ymax);
        \node[JK,above=1pt] at (\x,\ymax) {$T_{1}$};
      }
    \end{tikzpicture} 
    \caption{Objects}
    \label{fig:EmbOb}
    \end{subfigure}
    \begin{subfigure}[b]{0.49\textwidth}
    \centering
    \begin{tikzpicture}[scale=1, every node/.style={scale=1}, line width=1pt]
    \foreach \x in {-1.2, -0.4, 0.4, 1.2} {
      \draw[M,line width=1pt] (\x,-1) -- (\x,1);
    }

    \draw[black,line width=1pt]
      (0.8,-1) -- (0.8,-0.4)
      to[out=90,in=-90] (-0.8,0.8)
      -- (-0.8,1);
    \fill[black] (0.8,-0.5) circle (2pt);
    \fill[black] (0.8,-0.8) circle (2pt);
    \draw[<->] (0,-1.5) -- (0,-2.5);
    \node[black] at (0,-3) {$a_{13}s^2\in\Hom (T_3,T_1)$};
    \end{tikzpicture}
    \caption{Morphisms}
    \label{fig:EmbMor}
    \end{subfigure}
    \caption{Illustration of the embedding of the KLRW category $\mathcal{C}_{\bullet,1,F}$ in $\mathrm{Fuk}_{|||}(\M (\bullet,1), \mathcal{W}_{\mathbf{x}})$}
    \label{fig:Emb}
\end{figure}
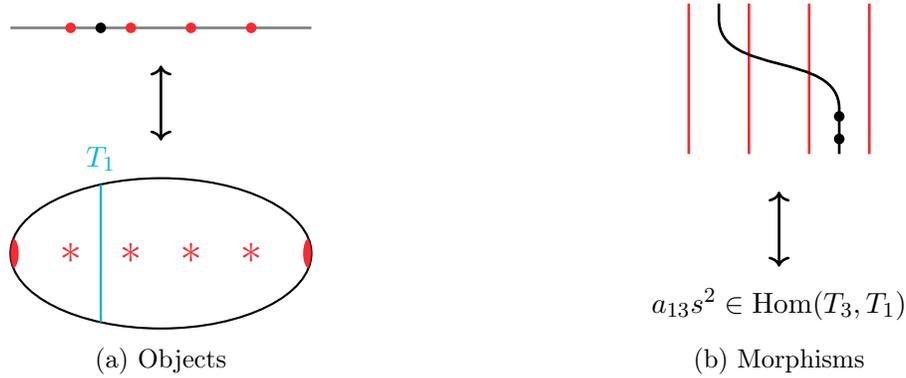

\subsection{General objects from generators}
\label{sect:Ibrane}

Besides the generating Lagrangians $T_i$, there exist other objects in the Fukaya category 
which appear as bounded twisted complexes built from the $T_i$. 
For such a twisted complex $(X^\bullet,\delta)$, the bounding cochain $\delta$ satisfies the Maurer--Cartan equation
\begin{equation}
  \mu^1_{\Sigma\mathcal{C}}(\delta)+\mu^2_{\Sigma\mathcal{C}}(\delta,\delta)+\mu^3_{\Sigma\mathcal{C}}(\delta,\delta,\delta)+\cdots = 0.
\end{equation}
where $\mu^d_{\Sigma\mathcal{C}}$ is the induced structure map on the additive enlargement defined in \cite[\S 3k]{Seidel-book}.
In our setting, since $\mu^d=0$ for $d\ne 2$, this reduces to the simple condition
\[
  \mu^2(\delta,\delta)=\delta^2=0.
\]
Hence $(X^\bullet,b)$ is nothing but an ordinary chain complex.

In \cite{Elise}, it is proved that all $T$-brane objects associated with curves without closed components can be generated from the basic objects $T_i$ by a standard procedure. More precisely, it means
\begin{equation}
  \mathrm{Tw}^b(\mathcal{C})\cong \mathrm{Fuk}_{|||}(\M (\bullet,1),\mathcal{W}_{\mathbf{x}})
\end{equation}
where in the latter we implicitly take all twisted complexes, which is essentially taking the pre-triangulated envelope since our category already has the idempotents $e_i$.

In the category $\mathrm{Tw}^b(\mathcal{C})$, the structure maps will be deformed by the bounding cochain by
\begin{equation}
\mu_{\delta}^d(a_d,\ldots,a_1)
  =
  \sum_{i_0,\ldots,i_d \ge 0}
  \mu^{d+i_0+\cdots+i_d}_{\Sigma \mathcal{C}}
  \big(
    \delta_d^{\otimes i_d}, a_d,
    \delta_{d-1}^{\otimes i_{d-1}}, a_{d-1},\dots,
    a_1, \delta_0^{\otimes i_0}
  \big).
\end{equation}
since only $\mu^2\ne 0$, the only nontrivial $\mu_\delta^d$ are

\begin{align}
&\mu _{\delta}^{1}\left( a \right) =(-1)^{|a|}\mu ^2\left( \delta,a \right) -\mu ^2\left( a,\delta \right) ,
\\
&\mu _{\delta}^{2}\left( a_2,a_1 \right) =(-1)^{|a_1|}\mu ^2\left( a_2,a_1 \right) .
\end{align}

Here $\mu^2$ on the RHS denotes the na\"ive composition in which the grading is ignored, that is, all morphisms are treated as of degree~0. We will always write $a_2\cdot a_1=\mu^2(a_2,a_1)$ for the na\"ive composition without adding any signs.

\begin{remark}
  The signs here follow strictly from Seidel's book \cite[\S 3k]{Seidel-book}, where the sign $(-1)^\triangleleft$ before each term is given by $\triangleleft=\sum_{p<q}|\phi^{i_p,i_{p-1}}_p|\cdot (|x^{i_q,i_{q-1}}_q|-1)$ for morphisms 
  \[a^{i_qi_{q-1}}_q=\phi^{i_qi_{q-1}}_q\otimes x^{i_qi_{q-1}}_q\in\mathrm{Hom}_{\Bbbk}(V^{i_{q-1}}_{q-1},V^{i_q}_{q})\otimes \Hom_{\mathcal{A}}(X^{i_{q-1}}_{q-1}\otimes X^{i_q}_{q}),\]
  where the underlying space for each twisted complex is given by $\bigoplus_i V_i\otimes X_i$.
  
  For our category $\mathcal{C}$, all morphisms are of degree 0, thus all $|x|=0$, and $|\phi|=|a|$, so we have the $\triangleleft$ before each $\mu^2(x_2,x_1)$ should be $|a_1|(0-1)=-|a_1|$.
\end{remark}

For the simplest examples, we can strech and resolve them into mapping cones, i.e. Lagrangian connected sums, of $T_i$. For example,

\begin{equation}\label{eq:stdmodT}
  \begin{tikzpicture}[scale=0.8]
  \foreach \x in {-2.5,-3.5,-4.5,-5.5}{
    \node[text=M,scale=1.5] at (\x,0) {$\ast$};
  }
  \draw[JK,line width=1pt]
    (-5,-1.5) -- (-5,0)
    arc[start angle=180,end angle=0,x radius=0.5cm,y radius=0.5cm]
    --  node[right,pos=0.75] {$T$} (-4,-1.5);

  \draw[<-] (-1.75,0) --node[above,midway] {$\mathrm{Cone}$} (-0.25,0);
  \foreach \x in {0.5,1.5,2.5,3.5}{
    \node[text=M,scale=1.5] at (\x,0) {$\ast$};
  }
  \draw[JK, name path=T1, line width=1pt]
    (1,-1.5) --node[left,pos=0.25] {$T_1$} (1,0) to[out=90,in=-135] (2,1.5);
    \draw[JY, name path=T2, line width=1pt]
    (2,-1.5) --node[right,pos=0.25] {$T_2$} (2,0) to[out=90,in=-45] (1,1.5);
    \fill[name intersections={of=T1 and T2}] (intersection-1) circle (2pt);

  \draw[<->] (4.5,0) -- (6,0);
  \node at (7,0) {$T_2$};
  \draw[->,shorten <=10pt, shorten >=10pt, line width=0.7pt] (7,0) --  (9,0);
  \node[above] at (8,0) {$\tikzset{every path/.append style={-}}
  \begin{tikzpicture}[scale=1.5, line width=0.7pt]
  \foreach \x in {0.05,0.15,0.25,0.35} {
      \draw[M] (\x,-0.1) -- (\x,0.1);
  }
  \draw[black]
      (0.2,-0.1) -- (0.2,-0.08)
      to[out=90,in=-90] (0.1,0.08)
      -- (0.1,0.1);
  \end{tikzpicture}$};
  \node at (9,0) {$T_1$};
  \end{tikzpicture}
\end{equation}

\begin{equation}\label{eq:beta-}
  \begin{tikzpicture}[scale=0.8]
  \foreach \x in {-2.5,-3.5,-4.5,-5.5}{
    \node[text=M,scale=1.5] at (\x,0) {$\ast$};
  }
  \draw[JK,line width=1pt]
    (-5,-1.5) -- (-5,0)
    arc[start angle=180,end angle=0,x radius=0.5cm,y radius=0.5cm]
    arc[start angle=180,end angle=360,x radius=0.5cm,y radius=0.5cm]
    --  node[left,pos=0.75] {$T$} (-3,1.5); 
  \draw[<-] (-1.75,0) --node[above,midway] {$\mathrm{Cone}$} (-0.25,0);
  \foreach \x in {0.5,1.5,2.5,3.5}{
    \node[text=M,scale=1.5] at (\x,0) {$\ast$};
  }
  \draw[JK, name path=T1, line width=1pt]
    (1,-1.5) --node[left,pos=0.25] {$T_1$} (1,0) to[out=90,in=-135] (2,1.5);
  \draw[JY, name path=T2, line width=1pt]
    (3,-1.5) to[out=135,in=-90]  (2,0) to[out=90,in=-45]node[right,pos=0.33] {$T_2$} (1,1.5);
  \draw[JK, name path=T3, line width=1pt]
    (2,-1.5) to[out=45,in=-90] (3,0)--node[right,pos=0.75] {$T_3$} (3,1.5);
  \fill[name intersections={of=T1 and T2}] (intersection-1) circle (2pt);
  \fill[name intersections={of=T3 and T2}] (intersection-1) circle (2pt);
  \draw[<->] (4.5,0) -- (6,0);
  \node at (7,0) {$T_2$};
  \draw[->,shorten <=10pt, shorten >=10pt, line width=0.7pt] (7,0) --  (9,1);
  \draw[->,shorten <=10pt, shorten >=10pt, line width=0.7pt] (7,0) --  (9,-1);
  \node[above, xshift=-3pt] at (8,0.5) {$\tikzset{every path/.append style={-}}
  \begin{tikzpicture}[scale=1.5, line width=0.7pt]
  \foreach \x in {0.05,0.15,0.25,0.35} {
      \draw[M] (\x,-0.1) -- (\x,0.1);
  }
  \draw[black]
      (0.2,-0.1) -- (0.2,-0.08)
      to[out=90,in=-90] (0.1,0.08)
      -- (0.1,0.1);
  \end{tikzpicture}$};
  \node[below, xshift=-9pt] at (8,-0.5) {$-$ $\tikzset{every path/.append style={-}}
  \begin{tikzpicture}[scale=1.5, line width=0.7pt]
  \foreach \x in {0.05,0.15,0.25,0.35} {
      \draw[M] (\x,-0.1) -- (\x,0.1);
  }
  \draw[black]
      (0.2,-0.1) -- (0.2,-0.08)
      to[out=90,in=-90] (0.3,0.08)
      -- (0.3,0.1);
  \end{tikzpicture}$};
  \node at (9,1) {$T_1$};
  \node at (9,0) {$\oplus$};
  \node at (9,-1) {$T_3$};
  \end{tikzpicture}
\end{equation}

\begin{equation}
  \begin{tikzpicture}[scale=0.8]
   \foreach \x in {-2.5,-3.5,-4.5,-5.5}{
    \node[text=M,scale=1.5] at (\x,0) {$\ast$};
  }
  \draw[JK,line width=1pt]
    (-3,-1.5) -- (-3,0)
    arc[start angle=0,end angle=180,x radius=0.5cm,y radius=0.5cm]
    arc[start angle=360,end angle=180,x radius=0.5cm,y radius=0.5cm]
    --  node[right,pos=0.75] {$T$} (-5,1.5);
  \draw[<-] (-1.75,0) --node[above,midway] {$\mathrm{Cone}$} (-0.25,0);
  \foreach \x in {0.5,1.5,2.5,3.5}{
    \node[text=M,scale=1.5] at (\x,0) {$\ast$};
  }
  \draw[JY, name path=T1, line width=1pt]
    (1,1.5) --node[left,pos=0.25] {$T_1$} (1,0) to[out=-90,in=135] (2,-1.5);
  \draw[JK, name path=T2, line width=1pt]
    (3,1.5) to[out=-135,in=90] node[left,pos=0.67] {$T_2$} (2,0) to[out=-90,in=45] (1,-1.5);
  \draw[JY, name path=T3, line width=1pt]
    (2,1.5) to[out=-45,in=90] (3,0)--node[right,pos=0.75] {$T_3$} (3,-1.5);
  \fill[name intersections={of=T1 and T2}] (intersection-1) circle (2pt);
  \fill[name intersections={of=T3 and T2}] (intersection-1) circle (2pt);

  \draw[<->] (4.5,0) -- (6,0);
  \node at (9,0) {$T_2$};
  \draw[->,shorten <=10pt, shorten >=10pt, line width=0.7pt] (7,1) --  (9,0);
  \draw[->,shorten <=10pt, shorten >=10pt, line width=0.7pt] (7,-1) --  (9,0);
  \node[above, xshift=2pt] at (8,0.5) {$-$ $\tikzset{every path/.append style={-}}
  \begin{tikzpicture}[scale=1.5, line width=0.7pt]
  \foreach \x in {0.05,0.15,0.25,0.35} {
      \draw[M] (\x,-0.1) -- (\x,0.1);
  }
  \draw[black]
      (0.1,-0.1) -- (0.1,-0.08)
      to[out=90,in=-90] (0.2,0.08)
      -- (0.2,0.1);
  \end{tikzpicture}$};
  \node[below, xshift=9pt] at (8,-0.5) {$\tikzset{every path/.append style={-}}
  \begin{tikzpicture}[scale=1.5, line width=0.7pt]
  \foreach \x in {0.05,0.15,0.25,0.35} {
      \draw[M] (\x,-0.1) -- (\x,0.1);
  }
  \draw[black]
      (0.3,-0.1) -- (0.3,-0.08)
      to[out=90,in=-90] (0.2,0.08)
      -- (0.2,0.1);
  \end{tikzpicture}$};
  \node at (7,1) {$T_1$};
  \node at (7,0) {$\oplus$};
  \node at (7,-1) {$T_3$};
  \end{tikzpicture}
\end{equation}

Besides $T$-branes, there exists another family of Lagrangian submanifolds of particular interest, 
whose projections to the fiber, illustrated in Fig.~\ref{fig:Ifiber}. 
We refer to these objects as \emph{$I$-branes}. Note that $I$-branes can end on punctures naturally, where the Lagrangian approaches $y\to \infty$, while $T$-branes cannot. Denote the $I$-brane in (\ref{eq:Ii}) connecting punctures $x_i$ and $x_{i+1}$ as $I_i$. These objects are \emph{dual} generators of the brane algebra \cite{Mina} to the $T$-branes in the sense that
\begin{equation}
  \Hom (T_i,I_j)=\delta_{ij}\Bbbk =\Hom(I_i,T_j).
\end{equation}

\begin{figure}[H]
    \centering
        \begin{subfigure}[b]{0.49\textwidth}
    \centering
	  \begin{tikzpicture}[xscale=0.5, yscale=0.5, every node/.style={scale=1}, line width=0.8pt]
      \draw (-8.5,2.5) -- (0,2.5);
      \draw (-8.5,0) -- (0,0);
      \draw (-8.5,1.25) ellipse (0.25 cm and 1.25 cm);
      \draw (0,0) arc[
          start angle=-90,
          end angle=90,
          x radius=0.25cm,
          y radius=1.25cm
          ];
        \draw[dashed] (0,2.5) arc[
          start angle=90,
          end angle=270,
          x radius=0.25cm,
          y radius=1.25cm
          ];
      \draw[thick, JK, name path=T1] (.25,1) -- (-8.25,1);
      \fill[M] (0,2.5) circle (4pt);
      \node[above] at (0,3) {$y=\infty$};
      \node[above] at (-8.5,3) {$y=0$};
    \end{tikzpicture}
    \caption{$T$-fiber}
    \end{subfigure}
    \begin{subfigure}[b]{0.49\textwidth}
    \centering
    	  \begin{tikzpicture}[xscale=0.5, yscale=0.5, every node/.style={scale=1}, line width=0.8pt]
      \draw (-8.5,2.5) -- (0,2.5);
      \draw (-8.5,0) -- (0,0);
      \draw (-8.5,1.25) ellipse (0.25 cm and 1.25 cm);
      \draw (0,0) arc[
          start angle=-90,
          end angle=90,
          x radius=0.25cm,
          y radius=1.25cm
          ];
        \draw[dashed] (0,2.5) arc[
          start angle=90,
          end angle=270,
          x radius=0.25cm,
          y radius=1.25cm
          ];
      \fill[M] (0,2.5) circle (4pt);
      \draw[JK, thick, in=-90, out=180] (.25,1.25) to (-3.247,2.5);
      \draw[JK, thick, in=90, out=180] (.25,1) to (-3.247,0);
      \draw[JK, thick, dashed] (-3.247,2.5) arc[
				start angle=90,
				end angle=270,
				x radius=0.25cm,
				y radius=1.25cm
				];
      \node[above] at (0,3) {$y=\infty$};
      \node[above] at (-8.5,3) {$y=0$};
    \end{tikzpicture}
    \caption{$I$-fiber}
    \label{fig:Ifiber}
    \end{subfigure}
    \caption{$T$-fiber and $I$-fiber}
\end{figure}
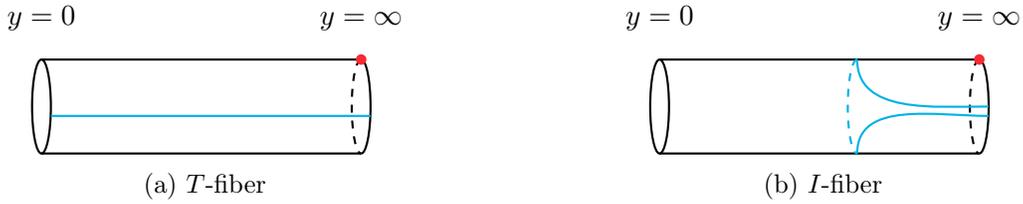

We recall some well-known results of $I$-branes as well.

\begin{equation}\label{eq:stdmodI}
  \begin{tikzpicture}[scale=0.8]
  \foreach \x in {0.5,1.5,2.5,3.5}{
    \node[text=M,scale=1.5] at (\x,0) {$\ast$};
  }
  \draw[JK,line width=1pt]
    (1.5,-1.5) --  node[right,midway] {$I$}  (1.5,0);
  \draw[<->] (4.5,0) -- (6,0);
  \node at (7,0) {$T_2$};
  \draw[->,shorten <=10pt, shorten >=10pt, line width=0.7pt] (7,0) --  (9,0);
  \node[above] at (8,0) {$\tikzset{every path/.append style={-}}
  \begin{tikzpicture}[scale=1.5, line width=0.7pt]
  \foreach \x in {0.05,0.15,0.25,0.35} {
      \draw[M] (\x,-0.1) -- (\x,0.1);
  }
  \draw[black]
      (0.2,-0.1) -- (0.2,-0.08)
      to[out=90,in=-90] (0.1,0.08)
      -- (0.1,0.1);
  \end{tikzpicture}$};
  \node at (9,0) {$T_1$};
  \end{tikzpicture}
\end{equation}

\begin{equation}
  \begin{tikzpicture}[scale=0.8]
  \foreach \x in {0.5,1.5,2.5,3.5}{
    \node[text=M,scale=1.5] at (\x,0) {$\ast$};
  }
  \draw[JK,line width=1pt]
    (1.5,0) --  node[above,midway] {$I$}  (2.5,0);
  \draw[<->] (4.5,0) -- (6,0);
  \node at (7,0) {$T_2$};
  \draw[->,shorten <=10pt, shorten >=10pt, line width=0.7pt] (7,0) --  (9,1);
  \draw[->,shorten <=10pt, shorten >=10pt, line width=0.7pt] (7,0) --  (9,-1);
  \node[above, xshift=-3pt] at (8,0.5) {$\tikzset{every path/.append style={-}}
  \begin{tikzpicture}[scale=1.5, line width=0.7pt]
  \foreach \x in {0.05,0.15,0.25,0.35} {
      \draw[M] (\x,-0.1) -- (\x,0.1);
  }
  \draw[black]
      (0.2,-0.1) -- (0.2,-0.08)
      to[out=90,in=-90] (0.1,0.08)
      -- (0.1,0.1);
  \end{tikzpicture}$};
  \node[below, xshift=-9pt] at (8,-0.5) {$-$ $\tikzset{every path/.append style={-}}
  \begin{tikzpicture}[scale=1.5, line width=0.7pt]
  \foreach \x in {0.05,0.15,0.25,0.35} {
      \draw[M] (\x,-0.1) -- (\x,0.1);
  }
  \draw[black]
      (0.2,-0.1) -- (0.2,-0.08)
      to[out=90,in=-90] (0.3,0.08)
      -- (0.3,0.1);
  \end{tikzpicture}$};
  \node at (9,1) {$T_1$};
  \node at (9,0) {$\oplus$};
  \node at (9,-1) {$T_3$};
  \draw[->,shorten <=10pt, shorten >=10pt, line width=0.7pt] (9,1) --  (11,0);
  \draw[->,shorten <=10pt, shorten >=10pt, line width=0.7pt] (9,-1) --  (11,0);
  \node[above, xshift=3pt] at (10,0.5) {$\tikzset{every path/.append style={-}}
  \begin{tikzpicture}[scale=1.5, line width=0.7pt]
  \foreach \x in {0.05,0.15,0.25,0.35} {
      \draw[M] (\x,-0.1) -- (\x,0.1);
  }
  \draw[black]
      (0.1,-0.1) -- (0.1,-0.08)
      to[out=90,in=-90] (0.2,0.08)
      -- (0.2,0.1);
  \end{tikzpicture}$};
  \node[below, xshift=3pt] at (10,-0.5) {$\tikzset{every path/.append style={-}}
  \begin{tikzpicture}[scale=1.5, line width=0.7pt]
  \foreach \x in {0.05,0.15,0.25,0.35} {
      \draw[M] (\x,-0.1) -- (\x,0.1);
  }
  \draw[black]
      (0.3,-0.1) -- (0.3,-0.08)
      to[out=90,in=-90] (0.2,0.08)
      -- (0.2,0.1);
  \end{tikzpicture}$};
  \node at (11,0) {$T_2$};
  \end{tikzpicture}
  \label{eq:Ii}
\end{equation}

\begin{remark}
  It is an interesting fact that (\ref{eq:stdmodT}) and (\ref{eq:stdmodI}) have the same resolution. The resulting complex is an example of \emph{standard modules} defined by Webster in \cite[\S 5]{Webster1}. This equivalence has both a disc-counting argument in \cite{Elise} and a geometric argument in \cite{op.213}.
\end{remark}

\section{Structure of the braiding functors}
Having established the structure of the category, we now proceed to discuss functors.

In general, an \emph{$A_\infty$-functor} between categories $\mathcal{A}$ and $\mathcal{B}$ consists of a map $F\colon\mathrm{Ob}\mathcal{A} \rightarrow \mathrm{Ob}\mathcal{B} $ and a series of maps $\{F^d\}$ for $d\ge 1$:
\begin{equation}
  F^d\colon \Hom_{\mathcal{A}}(X_{d-1},X_d)\otimes\cdots\otimes\Hom_{\mathcal{A}}(X_0,X_1)\to \Hom_{\mathcal{B}}(FX_0, FX_d)[1-d],
\end{equation}
where each $X_i\in\mathrm{Ob}\mathcal{A}$. The maps follow the constraints \cite[\S 1b]{Seidel-book}:

\begin{align}
&\sum_{r\ge1}
\sum_{\substack{i_1+\cdots+i_r=d}}
\mu_{\mathcal B}^r
\big(
  F^{i_r}(a_d,\dots,a_{d-i_r+1}),
  \dots,
  F^{i_1}(a_{i_1},\dots,a_1)
\big)\notag\\
&\qquad =
\sum_{m,n}
(-1)^{\maltese_n}
F^{d-m+1}
\big(
  a_d,\dots,a_{n+m+1},
  \mu_{\mathcal A}^m(a_{n+m},\dots,a_{n+1}),
  a_n,\dots,a_1
\big).\label{eq:functor}
\end{align}

Functors are composed by
\begin{align}
  &(G\comp F)^d(a_d,\dots,a_1)\notag\\
  &\qquad =\sum_{r\ge1}
\sum_{\substack{i_1+\cdots+i_r=d}}
G^r
\big(
  F^{i_r}(a_d,\dots,a_{d-i_r+1}),
  \dots,
  F^{i_1}(a_{i_1},\dots,a_1)
\big)\label{eq:funcomp}
\end{align}

Given an $A_\infty$-functor $F:\mathcal{A}\to\mathcal{B}$, one associates a 
$(\mathcal{B},\mathcal{A})$-bimodule by the Yoneda embedding of $F(X)$
\begin{equation}
  B_F(-,X)=\Hom_\mathcal{B}(-,F(X)).
\end{equation}
The structure maps $\mu_{B_F}^{0|1|s}$ coincide with the higher components $F^s$ of the functor, 
while the abstract realization of $F$ is given by the derived tensor functor
\[
F \simeq -\otimes^L_{\mathcal{A}} B_F.
\]

As mentioned in the introduction, given a set of punctures $\mathbf{x}$, the braid group $\mathrm{Br}_{|\mathbf{x}|}$ acts on $\M(\bullet, n)$ by braiding these punctures, which in turn induces an action on the corresponding Fukaya category.

\begin{equation}\label{eq:Br}
	\rho\colon \mathrm{Br}_{\left|\mathbf{x}\right|}\to \operatorname{Aut}(\mathrm{Fuk}_{|||}(\M (\Gamma,\vec{d}), \mathcal{W}_{\mathbf{x}}))
\end{equation}

The most essential part of this structure is the action of the generators of $\mathrm{Br}_{|\mathbf{x}|}$ and their inverses, and we can use \eqref{eq:funcomp} to get everything.
In what follows, we describe these functors in detail.

\subsection{Identity functor}

The simplest example of a braiding functor is the identity functor $\mathrm{id}=\rho(1)$
\begin{equation}
  \mathrm{id}\colon 
  \mathrm{Fuk}_{|||}(\M (\bullet,1),\mathcal{W}_{\mathbf{x}})
  \longrightarrow
  \mathrm{Fuk}_{|||}(\M (\bullet,1),\mathcal{W}_{\mathbf{x}}),
\end{equation}
which, by definition, acts trivially on the category. Explicitly,
on objects and morphisms we have
\begin{align}
  &\mathrm{id}(X)=X,\\
  &\mathrm{id}^1(a)=a,\\
  &\mathrm{id}^d=0,\qquad d\neq 1.
\end{align}
The $A_\infty$-functor equations~\eqref{eq:functor} are then automatically satisfied.

From the bimodule point of view, this functor corresponds to the 
\emph{diagonal bimodule} 
\begin{equation}
  \Delta(-,X)=\Hom (-,X),
\end{equation}
whose structure maps are the same as those of the category itself. 

\subsection{Negative braiding functor}

The \emph{negative braiding functor} $\beta_{i^-}=\rho(\sigma^{-1}_i)$ is the action of $\sigma^{-1}_i\in \mathrm{Br}_n$ under (\ref{eq:Br}), which clockwisely exchanges the punctures $x_i$ and $x_{i+1}$. We first specify its action on $\mathcal{C}$.

\subsubsection{Action on objects}
On the object level, as shown in Fig.~\ref{fig:beta-}, negative braiding $\beta_{i^-}$ twists $T_i$ into the object in (\ref{eq:beta-}), while leaving all other generators unchanged.

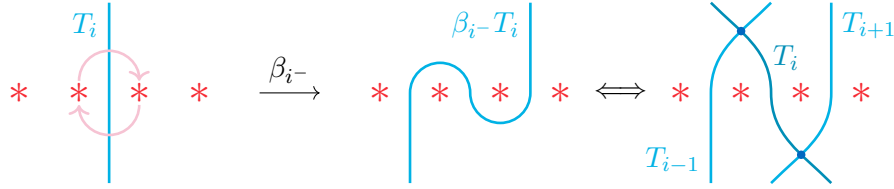
\begin{figure}[H]
  \centering
 \begin{tikzpicture}[scale=0.8]
  \foreach \x in {-5.5,-4.5,-3.5,-2.5,0.5,1.5,2.5,3.5}{
    \node[text=M,scale=1.5] at (\x,0) {$\ast$};
  }
  \draw[->] (-1.5,0) --node[above, midway]{$\beta_{i^-}$} (-0.5,0);
  \draw[JK,line width=1pt] (-4,-1.5)--node[left,pos=0.875] {$T_i$}(-4,1.5);
  \draw[JK,line width=1pt]
    (1,-1.5) -- (1,0)
    arc[start angle=180,end angle=0,x radius=0.5cm,y radius=0.5cm]
    arc[start angle=180,end angle=360,x radius=0.5cm,y radius=0.5cm] 
    --  node[left,pos=0.75] {$\beta_{i^-}T_i$} (3,1.5); 
  \draw[->,SA!50,line width=1pt] (-4.5,0.2) arc[start angle=180,end angle=0,x radius=0.5cm,y radius=0.5cm];
  \draw[->,SA!50,line width=1pt] (-3.5,-0.2) arc[start angle=0,end angle=-180,x radius=0.5cm,y radius=0.5cm];

  \node at (4.5,0) {$\iff$};
  \foreach \x in {5.5,6.5,7.5,8.5}{
    \node[text=M,scale=1.5] at (\x,0) {$\ast$};
  }
  \draw[JK, name path=T1, line width=1pt]
    (6,-1.5) --node[left,pos=0.25] {$T_{i-1}$} (6,0) to[out=90,in=-135] (7,1.5);
  \draw[JK!80!black, name path=T2, line width=1pt]
    (8,-1.5) to[out=135,in=-90]  (7,0) to[out=90,in=-45]node[right,pos=0.33] {$T_i$} (6,1.5);
  \draw[JK, name path=T3, line width=1pt]
    (7,-1.5) to[out=45,in=-90] (8,0)--node[right,pos=0.75] {$T_{i+1}$} (8,1.5);
  \fill[JO,name intersections={of=T1 and T2}] (intersection-1) circle (2pt);
  \fill[JO,name intersections={of=T3 and T2}] (intersection-1) circle (2pt);
 \end{tikzpicture}
 \caption{Negative braiding functor $\beta_{i^-}$ acting on $T_i$}
  \label{fig:beta-}
\end{figure}

\begin{theorem}\label{thm:obj}
  The negative braiding functor $\beta_{i^-}$ acts on generators $T_j$ by
  \begin{equation}
  \beta _{i^-}T_j=\begin{cases}
	T_j,&		i\ne j,\vspace{13pt}\\
	T_i\xrightarrow{\begin{bmatrix}
	\vcenter{\hbox{\tikzset{every path/.append style={-}}
  \begin{tikzpicture}[scale=1.5, line width=0.7pt]
  \foreach \x in {0.05,0.15,0.25,0.35} {
      \draw[red] (\x,-0.1) -- (\x,0.1);
  }
  \draw[black]
      (0.2,-0.1) -- (0.2,-0.08)
      to[out=90,in=-90] (0.1,0.08)
      -- (0.1,0.1);
  \end{tikzpicture}}}
	&-\vcenter{\hbox{\tikzset{every path/.append style={-}}
  \begin{tikzpicture}[scale=1.5, line width=0.7pt]
  \foreach \x in {0.05,0.15,0.25,0.35} {
      \draw[red] (\x,-0.1) -- (\x,0.1);
  }
  \draw[black]
      (0.2,-0.1) -- (0.2,-0.08)
      to[out=90,in=-90] (0.3,0.08)
      -- (0.3,0.1);
  \end{tikzpicture}}}
\end{bmatrix}}T_{i-1}\oplus T_{i+1}&		i=j.\\
\end{cases}
  \end{equation}
\end{theorem}

\subsubsection{Action on morphisms}
On the morphism level, $\beta_{i^-}^1$ acts trivially on $\Hom(T_j, T_k)$ unless $j=i$ or $k=i$,
\begin{itemize}
  \item For $a=a_{kj}s^\alpha\in\Hom(T_j,T_k)$ with $j,k\ne i$:
\end{itemize}
\begin{equation}\label{eq:beta1*}
  \beta_{i^-}^1(T_j\xrightarrow{s^\alpha}T_k)=T_j\xrightarrow{s^\alpha}T_k.
\end{equation}
where we write $s^\alpha$ on an arrow $T_i\to T_j$ to indicate the morphism contains $\alpha$ dots, i.e. $a_{ji}s^\alpha\in \Hom (T_i,T_j)$.

\begin{itemize}
  \item For $a=a_{ij}s^\alpha\in\Hom (T_j,T_i)$ with $j\ne i$:
\end{itemize}
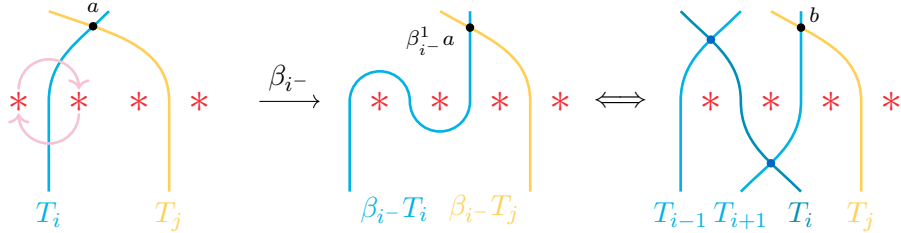
\begin{figure}[H]
  \centering
 \begin{tikzpicture}[scale=0.8]
  \foreach \x in {-4.5,-3.5,-2.5,-1.5,1.5,2.5,3.5,4.5}{
    \node[text=M,scale=1.5] at (\x,0) {$\ast$};
  }
  \draw[->] (-0.5,0) --node[above, midway]{$\beta_{i^-}$} (0.5,0);
  \draw[JI, name path=Tj, line width=1pt] (-2,-1.5)node[below] {$T_j$}--(-2,0) to[out=90,in=-20](-4,1.5);
  \draw[JK, name path=Ti, line width=1pt] (-4,-1.5)node[below] {$T_i$}--(-4,0) to[out=90,in=-135](-3,1.5);
  \fill[name intersections={of=Tj and Ti}] (intersection-1) circle (2pt)node[yshift=7pt]{\scriptsize $a$};
  \draw[->,SA!50,line width=1pt] (-4.5,0.2) arc[start angle=180,end angle=0,x radius=0.5cm,y radius=0.5cm];
  \draw[->,SA!50,line width=1pt] (-3.5,-0.2) arc[start angle=0,end angle=-180,x radius=0.5cm,y radius=0.5cm];

  \draw[JK, name path=BTi, line width=1pt]
    (1,-1.5) --node[right,pos=-0.2] {$\beta_{i^-}T_i$} (1,0)
    arc[start angle=180,end angle=0,x radius=0.5cm,y radius=0.5cm]
    arc[start angle=180,end angle=360,x radius=0.5cm,y radius=0.5cm] 
    --   (3,1.5); 

  \draw[JI, name path=BTj, line width=1pt]
    (4,-1.5)--node[left,pos=-0.2] {$\beta_{i^-}T_j$}(4,0)to[out=90,in=-30](2.5,1.5);
  \fill[name intersections={of=BTj and BTi}] (intersection-1) circle (2pt)node[xshift=-15pt,yshift=-5pt]{\scriptsize $\beta_{i^-}^1a$};
  
  \node at (5.5,0) {$\iff$};
\begin{scope}[xshift=0.5cm]
  \foreach \x in {6.5,7.5,8.5,9.5}{
    \node[text=M,scale=1.5] at (\x,0) {$\ast$};
  }
  \draw[JK, name path=T1, line width=1pt]
    (6,-1.5) node[below] {$T_{i-1}$} -- (6,0) to[out=90,in=-135] (7,1.5);
  \draw[JK!80!black, name path=T2, line width=1pt]
    (8,-1.5) node[below] {$T_i$} to[out=135,in=-90] (7,0) to[out=90,in=-45] (6,1.5);
  \draw[JK, name path=T3, line width=1pt]
    (7,-1.5) node[below] {$T_{i+1}$} to[out=45,in=-90] (8,0) -- (8,1.5);
  \draw[JI, name path=T4, line width=1pt]
    (9,-1.5) node[below] {$T_j$} -- (9,0) to[out=90,in=-30] (7.5,1.5);
  \fill[JO,name intersections={of=T1 and T2}] (intersection-1) circle (2pt);
  \fill[JO,name intersections={of=T3 and T2}] (intersection-1) circle (2pt);
  \fill[name intersections={of=T4 and T3}] (intersection-1) circle (2pt)
       node[xshift=5pt,yshift=5pt]{\scriptsize $b$};
\end{scope}
 \end{tikzpicture}
 \caption{$\beta_{i^-}^1$ acting on $\Hom (T_j,T_i)$ with $j>i$}
 \label{fig:beta1a}
\end{figure}
The morphism $b$ in Fig.~\ref{fig:beta1a} have the same number of dots with $a$. Thus, we have $\beta_{i^-}^1(a)=b$ for $j>i$, which is

\begin{equation}\label{eq:TjTi}
  \begin{tikzpicture}
  \node[yshift=3] at (-3,-1) {$\beta_{i^-}^1(T_j\xrightarrow{s^\alpha}T_i)=$};
  \node at (0,-1) {$T_j$};
  \draw[->,shorten <=13pt, shorten >=13pt, line width=0.7pt] (0,-1) --node[above,pos=0.382]{\small $b=s^\alpha$}   (6,-1);
  \begin{scope}[xshift=-3cm]
  \node at (7,0) {$T_i$};
  \draw[->,shorten <=13pt, shorten >=13pt, line width=0.7pt] (7,0) --  (9,1);
  \draw[->,shorten <=13pt, shorten >=13pt, line width=0.7pt] (7,0) --  (9,-1);
  \node[above, xshift=-3pt] at (8,0.5) {$\tikzset{every path/.append style={-}}
  \begin{tikzpicture}[scale=1.5, line width=0.7pt]
  \foreach \x in {0.15,0.25} {
      \draw[M] (\x,-0.1) -- (\x,0.1);
  }
  \draw[black]
      (0.2,-0.1) -- (0.2,-0.08)
      to[out=90,in=-90] (0.1,0.08)
      -- (0.1,0.1);
  \end{tikzpicture}$};
  \node[below, xshift=-9pt] at (8,-0.5) {$-$ $\tikzset{every path/.append style={-}}
  \begin{tikzpicture}[scale=1.5, line width=0.7pt]
  \foreach \x in {0.15,0.25} {
      \draw[M] (\x,-0.1) -- (\x,0.1);
  }
  \draw[black]
      (0.2,-0.1) -- (0.2,-0.08)
      to[out=90,in=-90] (0.3,0.08)
      -- (0.3,0.1);
  \end{tikzpicture}$};
  \node at (9,1) {$T_{i-1}$};
  \node at (9,0) {$\oplus$};
  \node at (9,-1) {$T_{i+1}$};
  \draw[gray] (6.7,1.5)--(6.5,1.5)--(6.5,-1.5)--(6.7,-1.5);
  \draw[gray] (9.5,1.5)--(9.7,1.5)--(9.7,-1.5)--(9.5,-1.5);
  \end{scope}
  \end{tikzpicture}
\end{equation}

Similarly, we get for $j<i$,

\begin{equation}
  \begin{tikzpicture}
  \node[yshift=3] at (-3,1) {$\beta_{i^-}^1(T_j\xrightarrow{s^\alpha}T_i)=$};
  \node at (0,1) {$T_j$};
  \draw[->,shorten <=13pt, shorten >=13pt, line width=0.7pt] (0,1) --node[above,pos=0.382]{\small $s^\alpha$}   (6,1);

  \begin{scope}[xshift=-3cm]
  \node at (7,0) {$T_i$};
  \draw[->,shorten <=13pt, shorten >=13pt, line width=0.7pt] (7,0) --  (9,1);
  \draw[->,shorten <=13pt, shorten >=13pt, line width=0.7pt] (7,0) --  (9,-1);
  \node[above, xshift=-3pt] at (8,0.5) {$\tikzset{every path/.append style={-}}
  \begin{tikzpicture}[scale=1.5, line width=0.7pt]
  \foreach \x in {0.15,0.25} {
      \draw[M] (\x,-0.1) -- (\x,0.1);
  }
  \draw[black]
      (0.2,-0.1) -- (0.2,-0.08)
      to[out=90,in=-90] (0.1,0.08)
      -- (0.1,0.1);
  \end{tikzpicture}$};
  \node[below, xshift=-9pt] at (8,-0.5) {$-$ $\tikzset{every path/.append style={-}}
  \begin{tikzpicture}[scale=1.5, line width=0.7pt]
  \foreach \x in {0.15,0.25} {
      \draw[M] (\x,-0.1) -- (\x,0.1);
  }
  \draw[black]
      (0.2,-0.1) -- (0.2,-0.08)
      to[out=90,in=-90] (0.3,0.08)
      -- (0.3,0.1);
  \end{tikzpicture}$};
  \node at (9,1) {$T_{i-1}$};
  \node at (9,0) {$\oplus$};
  \node at (9,-1) {$T_{i+1}$};
  \draw[gray] (6.7,1.5)--(6.5,1.5)--(6.5,-1.5)--(6.7,-1.5);
  \draw[gray] (9.5,1.5)--(9.7,1.5)--(9.7,-1.5)--(9.5,-1.5);
  \end{scope}
  \end{tikzpicture}
\end{equation}

\begin{itemize}
  \item For $a=a_{ki}s^\alpha\in\Hom (T_i,T_k)$ with $k\ne i$:
\end{itemize}

\begin{figure}[H]
  \centering
 \begin{tikzpicture}[scale=0.8]
  \foreach \x in {-4.5,-3.5,-2.5,-1.5,1.5,2.5,3.5,4.5}{
    \node[text=M,scale=1.5] at (\x,0) {$\ast$};
  }
  \draw[->] (-0.5,0) --node[above, midway]{$\beta_{i^-}$} (0.5,0);
  \draw[JI, name path=Ti, line width=1pt] (-4,1.5)node[above] {$T_i$}--(-4,0) to[out=-90,in=160](-2,-2.5);
  \draw[JK, name path=Tj, line width=1pt] (-2,1.5)node[above] {$T_k$}--(-2,0) to[out=-90,in=45](-3,-2.5);
  \fill[name intersections={of=Tj and Ti}] (intersection-1) circle (2pt)node[yshift=7pt]{\scriptsize $a$};
  \draw[->,SA!50,line width=1pt] (-4.5,0.2) arc[start angle=180,end angle=0,x radius=0.5cm,y radius=0.5cm];
  \draw[->,SA!50,line width=1pt] (-3.5,-0.2) arc[start angle=0,end angle=-180,x radius=0.5cm,y radius=0.5cm];

  \draw[JI, name path=BTi, line width=1pt]
    (2,-2.5) to[out=135, in=-90](1,-1.0) -- (1,0)
    arc[start angle=180,end angle=0,x radius=0.5cm,y radius=0.5cm]
    arc[start angle=180,end angle=360,x radius=0.5cm,y radius=0.5cm] 
    -- node[left,pos=1.3] {$\beta_{i^-}T_i$}  (3,1.5); 

  \draw[JK, name path=BTj, line width=1pt]
    (1,-2.5)to[out=20,in=-90](4,0)--node[right,pos=1.3] {$\beta_{i^-}T_k$}(4,1.5);
  \fill[name intersections={of=BTj and BTi}] (intersection-1) circle (2pt)node[xshift=3pt, yshift=13pt]{\scriptsize $\beta_{i^-}^1a$};
  
  \node at (5.5,0) {$\iff$};
\begin{scope}[xshift=0.5cm]
  \foreach \x in {6.5,7.5,8.5,9.5}{
    \node[text=M,scale=1.5] at (\x,0) {$\ast$};
  }
  \draw[JI, name path=T1, line width=1pt]
    (6,-2.5)  -- (6,0) to[out=90,in=-135] (7,1.5)node[above] {$T_{i-1}$};
  \draw[JI+, name path=T2, line width=1pt]
    (8,-2.5)  to[out=135,in=-90] (7,-1)--(7,0) to[out=90,in=-45] (6,1.5)node[above] {$T_i$};
  \draw[JI, name path=T3, line width=1pt]
    (7,-2.5) to[out=45,in=-90] (8,-1) -- (8,1.5)node[above] {$T_{i+1}$};
  \draw[JK, name path=T4, line width=1pt]
    (5,-2.5) --(5,-2) to[out=90,in=-90] (9,0.3) -- (9,1.5)node[above] {$T_k$};
  \fill[JT,name intersections={of=T1 and T2}] (intersection-1) circle (2pt)node[yshift=-8pt]{\scriptsize $\delta$};
  \fill[JT,name intersections={of=T3 and T2}] (intersection-1) circle (2pt)node[yshift=8pt]{\scriptsize $\delta$};
  \fill[name intersections={of=T4 and T3}] (intersection-1) circle (2pt)
       node[xshift=8pt,yshift=-5pt]{\scriptsize $b_+$};
  \fill[name intersections={of=T4 and T1}] (intersection-1) circle (2pt)
       node[xshift=8pt,yshift=-5pt]{\scriptsize $b_-$};
  \fill[name intersections={of=T4 and T2}] (intersection-1) circle (2pt)
       node[xshift=7pt,yshift=6pt]{\scriptsize $b_0$};
\end{scope}
 \end{tikzpicture}
 \caption{$\beta_{i^-}^1$ acting on $\Hom (T_i,T_k)$ with $k>i$}
\end{figure}
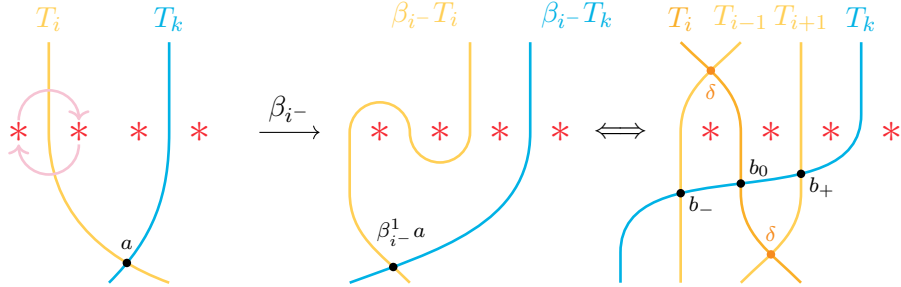
We have $\beta_{i^-}^1(a)=[b_-, b_+]$. Here, $b_-$ is directly inherited from $a$, thus $b_-=a_{k,i-1}s^\alpha$. $b_+$ and $b_0$ are given by the pseudoholomorphic discs $b_0=b_-\cdot \delta =b_+\cdot \delta$, thus $b_+=a_{k,i+1}s^{\alpha +1}$ since one of the discs encloses a puncture.

For $k>i$,
\begin{equation}
  \begin{tikzpicture}
  \node at (4,0) {$\beta_{i^-}^1(T_i\xrightarrow{s^\alpha} T_k)=$};
  \draw[gray] (6.7,1.5)--(6.5,1.5)--(6.5,-1.5)--(6.7,-1.5);
  \draw[gray] (9.5,1.5)--(9.7,1.5)--(9.7,-1.5)--(9.5,-1.5);
  \node at (7,0) {$T_i$};
  \draw[->,shorten <=13pt, shorten >=13pt, line width=0.7pt] (7,0) --  (9,1);
  \draw[->,shorten <=13pt, shorten >=13pt, line width=0.7pt] (7,0) --  (9,-1);
  \node[above, xshift=-3pt] at (8,0.5) {$\tikzset{every path/.append style={-}}
  \begin{tikzpicture}[scale=1.5, line width=0.7pt]
  \foreach \x in {0.15,0.25} {
      \draw[M] (\x,-0.1) -- (\x,0.1);
  }
  \draw[black]
      (0.2,-0.1) -- (0.2,-0.08)
      to[out=90,in=-90] (0.1,0.08)
      -- (0.1,0.1);
  \end{tikzpicture}$};
  \node[below, xshift=-9pt] at (8,-0.5) {$-$ $\tikzset{every path/.append style={-}}
  \begin{tikzpicture}[scale=1.5, line width=0.7pt]
  \foreach \x in {0.15,0.25} {
      \draw[M] (\x,-0.1) -- (\x,0.1);
  }
  \draw[black]
      (0.2,-0.1) -- (0.2,-0.08)
      to[out=90,in=-90] (0.3,0.08)
      -- (0.3,0.1);
  \end{tikzpicture}$};
  \node at (9,1) {$T_{i-1}$};
  \node at (9,0) {$\oplus$};
  \node at (9,-1) {$T_{i+1}$};

  \node at (13,-1) {$T_k$};
  \draw[->,shorten <=13pt, shorten >=13pt, line width=0.7pt] (9,1)--node[above,midway,yshift=5pt]{\small $s^\alpha$}  (13,-1);
  \draw[->,shorten <=13pt, shorten >=13pt, line width=0.7pt] (9,-1)--node[above,midway]{\small $s^{\alpha+1}$}  (13,-1);
  \end{tikzpicture}
\end{equation}

Correspondingly, for $k<i$
\begin{equation}
  \begin{tikzpicture}
  \node at (4,0) {$\beta_{i^-}^1(T_i\xrightarrow{s^\alpha} T_k)=$};
  \draw[gray] (6.7,1.5)--(6.5,1.5)--(6.5,-1.5)--(6.7,-1.5);
  \draw[gray] (9.5,1.5)--(9.7,1.5)--(9.7,-1.5)--(9.5,-1.5);
  \node at (7,0) {$T_i$};
  \draw[->,shorten <=13pt, shorten >=13pt, line width=0.7pt] (7,0) --  (9,1);
  \draw[->,shorten <=13pt, shorten >=13pt, line width=0.7pt] (7,0) --  (9,-1);
  \node[above, xshift=-3pt] at (8,0.5) {$\tikzset{every path/.append style={-}}
  \begin{tikzpicture}[scale=1.5, line width=0.7pt]
  \foreach \x in {0.15,0.25} {
      \draw[M] (\x,-0.1) -- (\x,0.1);
  }
  \draw[black]
      (0.2,-0.1) -- (0.2,-0.08)
      to[out=90,in=-90] (0.1,0.08)
      -- (0.1,0.1);
  \end{tikzpicture}$};
  \node[below, xshift=-9pt] at (8,-0.5) {$-$ $\tikzset{every path/.append style={-}}
  \begin{tikzpicture}[scale=1.5, line width=0.7pt]
  \foreach \x in {0.15,0.25} {
      \draw[M] (\x,-0.1) -- (\x,0.1);
  }
  \draw[black]
      (0.2,-0.1) -- (0.2,-0.08)
      to[out=90,in=-90] (0.3,0.08)
      -- (0.3,0.1);
  \end{tikzpicture}$};
  \node at (9,1) {$T_{i-1}$};
  \node at (9,0) {$\oplus$};
  \node at (9,-1) {$T_{i+1}$};

  \node at (13,1) {$T_k$};
  \draw[->,shorten <=13pt, shorten >=13pt, line width=0.7pt] (9,1)--node[above,midway,yshift=5pt]{\small $s^{\alpha+1}$}  (13,1);
  \draw[->,shorten <=13pt, shorten >=13pt, line width=0.7pt] (9,-1)--node[above,midway]{\small $s^{\alpha}$}  (13,1);
  \end{tikzpicture}
\end{equation}

The only morphisms left are the ones in $\Hom (T_i,T_i)$, which can be computed using the same procedure above. To be concise, we write out the results directly.

\begin{itemize}
  \item For $a=s^\alpha_i\in \Hom(T_i,T_i)$:
\end{itemize}

\begin{equation}
  \begin{tikzpicture}
  \node at (-3,0) {$\beta_{i^-}^1(T_i\xrightarrow{s^\alpha}T_i)=$};
  \begin{scope}[xshift=-7cm]
  \node at (7,0) {$T_i$};
  \draw[->,shorten <=13pt, shorten >=13pt, line width=0.7pt] (7,0) --   (11,0);
  \draw[->,shorten <=13pt, shorten >=13pt, line width=0.7pt] (9,1) --   (13,1);
  \draw[->,shorten <=13pt, shorten >=13pt, line width=0.7pt] (9,-1) --   (13,-1);
  \node[above] at (10,0){\small $s^\alpha$};
  \node[above] at (10,1){\small $s^\alpha$};
  \node[above] at (10,-1){\small $s^\alpha$};
  \draw[->,shorten <=13pt, shorten >=13pt, line width=0.7pt] (7,0) --  (9,1);
  \draw[->,shorten <=13pt, shorten >=13pt, line width=0.7pt] (7,0) --  (9,-1);
  \node[above, xshift=-3pt] at (8,0.5) {$\tikzset{every path/.append style={-}}
  \begin{tikzpicture}[scale=1.5, line width=0.7pt]
  \foreach \x in {0.15,0.25} {
      \draw[M] (\x,-0.1) -- (\x,0.1);
  }
  \draw[black]
      (0.2,-0.1) -- (0.2,-0.08)
      to[out=90,in=-90] (0.1,0.08)
      -- (0.1,0.1);
  \end{tikzpicture}$};
  \node[below, xshift=-9pt] at (8,-0.5) {$-$ $\tikzset{every path/.append style={-}}
  \begin{tikzpicture}[scale=1.5, line width=0.7pt]
  \foreach \x in {0.15,0.25} {
      \draw[M] (\x,-0.1) -- (\x,0.1);
  }
  \draw[black]
      (0.2,-0.1) -- (0.2,-0.08)
      to[out=90,in=-90] (0.3,0.08)
      -- (0.3,0.1);
  \end{tikzpicture}$};
  \node at (9,1) {$T_{i-1}$};
  \node at (9,-0.3) {$\oplus$};
  \node at (9,-1) {$T_{i+1}$};
  \draw[gray] (6.7,1.5)--(6.5,1.5)--(6.5,-1.5)--(6.7,-1.5);
  \draw[gray] (9.5,1.5)--(9.7,1.5)--(9.7,-1.5)--(9.5,-1.5);
  \end{scope}
  \begin{scope}[xshift=-3cm]
  \node at (7,0) {$T_i$};
  \draw[->,shorten <=13pt, shorten >=13pt, line width=0.7pt] (7,0) --  (9,1);
  \draw[->,shorten <=13pt, shorten >=13pt, line width=0.7pt] (7,0) --  (9,-1);
  \node[above, xshift=-3pt] at (8,0.5) {$\tikzset{every path/.append style={-}}
  \begin{tikzpicture}[scale=1.5, line width=0.7pt]
  \foreach \x in {0.15,0.25} {
      \draw[M] (\x,-0.1) -- (\x,0.1);
  }
  \draw[black]
      (0.2,-0.1) -- (0.2,-0.08)
      to[out=90,in=-90] (0.1,0.08)
      -- (0.1,0.1);
  \end{tikzpicture}$};
  \node[below, xshift=-9pt] at (8,-0.5) {$-$ $\tikzset{every path/.append style={-}}
  \begin{tikzpicture}[scale=1.5, line width=0.7pt]
  \foreach \x in {0.15,0.25} {
      \draw[M] (\x,-0.1) -- (\x,0.1);
  }
  \draw[black]
      (0.2,-0.1) -- (0.2,-0.08)
      to[out=90,in=-90] (0.3,0.08)
      -- (0.3,0.1);
  \end{tikzpicture}$};
  \node at (9,1) {$T_{i-1}$};
  \node at (9,0) {$\oplus$};
  \node at (9,-1) {$T_{i+1}$};
  \draw[gray] (6.7,1.5)--(6.5,1.5)--(6.5,-1.5)--(6.7,-1.5);
  \draw[gray] (9.5,1.5)--(9.7,1.5)--(9.7,-1.5)--(9.5,-1.5);
  \end{scope}
  \end{tikzpicture}\label{eq:beta1.}
\end{equation}

\begin{theorem}
  \eqref{eq:beta1*}--\eqref{eq:beta1.} gives the action of $\beta_{i^-}^1$ on all morphisms $\Hom (T_j,T_k)$.
\end{theorem}

\begin{remark}
We can already verify the condition for a functor (\ref{eq:functor}) for $d=1$, which is
\begin{equation}
  \mu^1_{\delta}\beta_{i^-}^1(a)=\beta_{i^-}^1\mu^1_{\delta}(a).
\end{equation}
Since for morphisms between generators we have $\mu^1_{\delta}(a)=0$ and $|\beta^1_{i-}(a)|=0$, this is
\begin{align}
  \mu^1_{\delta}\beta_{i^-}^1(a) =\delta \beta_{i^-}^1(a)-\beta_{i^-}^1(a) \delta =0,
\end{align}
which is followed by direct computation.
\end{remark}

\subsubsection{Higher components}
There remains one piece of the puzzle to be clarified. 
According to \eqref{eq:functor}, the components $\bim^d$ carry a cohomological degree shift of $[1-d]$. 
Theorem~\ref{thm:obj} shows that the bimodule $\bim T_j$ is a complex concentrated in degree~$0$ for $j\neq i$, 
while for $j=i$ it has nonzero terms in degrees $0$ and $-1$. 
Consequently, the only potentially nontrivial components $\bim^d$ on $\mathcal{C}$ occur for $d=1$ and $d=2$,
\begin{equation}
  \bim^d=0,\quad d\ge 3.
\end{equation}
In particular, since $\bim^2$ has degree~$-1$, it can be nonzero only on composable morphisms of the form
\[
T_j \longrightarrow T_k \longrightarrow T_i.
\]
where $\bim^2$ gets a morphism in $\Hom^{-1}(\bim T_j,\bim T_i)$.

The component $\bim^2$ can be solved directly from \eqref{eq:functor} with $d=2$, which is
\begin{align}
  &\mu _{\delta}^{2}\left( \beta _{i^-}^{1}\left( a_2 \right) ,\beta _{i^-}^{1}\left( a_1 \right) \right) +\mu _{\delta}^{1}\left( \beta _{i^-}^{2}\left( a_2,a_1 \right) \right) =\beta _{i^-}^{1}\left( \mu ^2\left( a_2,a_1 \right) \right) 
  \\
  \implies &\delta \beta _{i^-}^{2}\left( a_2,a_1 \right) +\beta _{i^-}^{2}\left( a_2,a_1 \right) \delta =\beta _{i^-}^{1}\left( a_2 \right) \cdot\beta _{i^-}^{1}\left( a_1 \right)  -\beta _{i^-}^{1}\left( a_2\cdot a_1  \right) \label{eq:beta2cond}
\end{align}

\begin{theorem}\label{thm:Marco}
  \mathleft
  This was first computed by Marco David. For $a_{lk}s^\beta\otimes a_{kj}a^\alpha\in \Hom(T_k,T_l)\otimes\Hom(T_j,T_k)$, $\bim^2(a_{lk}s^\beta, a_{kj}s^\alpha)$ is given by:
  \begin{itemize}
    \item if $l\ne i$,\begin{equation}
      \bim^2(a_{lk}s^\beta, a_{kj}s^\alpha)=0;
    \end{equation}
    \item if $l=i$ and $j\ne i$,
    \begin{itemize}
      \item if $(i-j)(i-k)\ge 0$,\begin{equation}
      \bim^2(a_{ik}s^\beta, a_{kj}s^\alpha)=0;
    \end{equation}
      \item if $(i-j)(i-k)<0$,\begin{equation}
        \bim^2(a_{ik}s^\beta, a_{kj}s^\alpha)=\mathrm{sgn}(i-k) a_{ij}s^{\beta+\alpha+|i-k|-1},
      \end{equation}
      which is a map from $\bim T_j =T_j$ to the degree $-1$ component of $\bim T_i$;
    \end{itemize}
    \item if $l=j=i$,
    \begin{equation}
      \bim^2(a_{ik}s^\beta, a_{ki}s^\alpha)=\mathrm{sgn}(i-k) a_{i,i+\mathrm{sgn}(i-k)}s^{\beta+\alpha+|i-k|-1},
    \end{equation}
    which is a degree $-1$ map from $(\bim T_i)^0=T_{i-1}\oplus T_{i+1}$ to $(\bim T_i)^{-1}=T_i$.
  \end{itemize}
\end{theorem}
\begin{proof}
  These can be checked directly to satisfy \eqref{eq:beta2cond}. For example, for $T_j\xrightarrow{s^\alpha}T_k\xrightarrow{s^\beta}T_i$ with $k<i<j$,
  \begin{equation}
        \begin{tikzpicture}
      \node at(0,0) {$T_j$};
      \node at(2,1) {$T_k$};
      \node at(6,0) {$T_i$};
      \node at(8,1) {$T_{i-1}$};
      \node at(8,-1){$T_{i+1}$};
      \draw[gray] (5.7,1.5)--(5.5,1.5)--(5.5,-1.5)--(5.7,-1.5);
      \draw[gray] (8.5,1.5)--(8.7,1.5)--(8.7,-1.5)--(8.5,-1.5);
      \draw[->,shorten <=9pt, shorten >=13pt, line width=0.7pt](6,0)--(8,1);
      \draw[->,shorten <=9pt, shorten >=13pt, line width=0.7pt](6,0)--node[below,midway]{$-$}(8,-1);
      \draw[->,JI,shorten <=13pt, shorten >=13pt, line width=0.7pt](0,0)--node[above,midway,JI+]{\small$s^\alpha$}(2,1);
      \draw[->,JI,shorten <=13pt, shorten >=13pt, line width=0.7pt](2,1)--(8,1);
      \node[above,JI+] at(4,1){\small$s^\beta$};
      \draw[->,JK,shorten <=13pt, shorten >=13pt, line width=0.7pt] (0,0) to[out=-45, in=180]  (4,-1)-- (8,-1);
      \node[above,JK] at(4,-1){\small$s^{\alpha+\beta+|i-k|}$};
      \draw[->,JY,shorten <=13pt, shorten >=9pt, line width=0.7pt] (0,0) -- (6,0);
      \node[above,JY] at(4,0){\small$s^{\alpha+\beta+|i-k|-1}$};
      \node[JI+] at(11,1){$\bim^1\bim^1$};
      \node[JY] at(11,0){$\bim^2$};
      \node[JK] at(11,-1){$\bim^1\mu^2$};
    \end{tikzpicture}
  \end{equation}
  for $T_i\xrightarrow{s^\alpha}T_k\xrightarrow{s^\beta}T_i$ with $k<i$,
  \begin{equation}
    \begin{tikzpicture}
      \begin{scope}[xshift=-7cm]
      \node at(6,0) {$T_i$};
      \node at(8,1) {$T_{i-1}$};
      \node at(8,-1){$T_{i+1}$};
      \draw[gray] (5.7,1.5)--(5.5,1.5)--(5.5,-1.5)--(5.7,-1.5);
      \draw[gray] (8.5,1.5)--(8.7,1.5)--(8.7,-1.5)--(8.5,-1.5);
      \draw[->,shorten <=9pt, shorten >=13pt, line width=0.7pt](6,0)--(8,1);
      \draw[->,shorten <=9pt, shorten >=13pt, line width=0.7pt](6,0)--node[below,midway]{$-$}(8,-1);
      \end{scope}
      \node at(5,2) {$T_k$};
      \node at(6,0) {$T_i$};
      \node at(8,1) {$T_{i-1}$};
      \node at(8,-1){$T_{i+1}$};
      \draw[gray] (5.7,1.5)--(5.5,1.5)--(5.5,-1.5)--(5.7,-1.5);
      \draw[gray] (8.5,1.5)--(8.7,1.5)--(8.7,-1.5)--(8.5,-1.5);
      \draw[->,shorten <=9pt, shorten >=13pt, line width=0.7pt](6,0)--(8,1);
      \draw[->,shorten <=9pt, shorten >=13pt, line width=0.7pt](6,0)--node[below,midway]{$-$}(8,-1);

      \draw[->,JI,shorten <=13pt, shorten >=13pt, line width=0.7pt](1,1)--node[JI+,above,pos=0.4]{\small$s^{\alpha+1}$}(5,2);
      \draw[->,JI,shorten <=13pt, shorten >=13pt, line width=0.7pt](1,-1)--node[JI+,above,pos=0.4]{\small $s^\alpha$}(5,2);
      \draw[->,JI,shorten <=13pt, shorten >=13pt, line width=0.7pt](5,2)--node[JI+,above,midway]{\small $s^\beta$}(8,1);
      \draw[->,JK,shorten <=13pt, shorten >=9pt, line width=0.7pt] (-1,0) -- (6,0);
      \draw[->,JK,shorten <=13pt, shorten >=13pt, line width=0.7pt] (1,-1) -- (8,-1);
      \draw[->,JK,shorten <=13pt, shorten >=13pt, line width=0.7pt] (1,1) -- (8,1);

      \node[above,JK] at(4,0){\small $s^{\alpha+\beta+|i-k|}$};
      \node[above,JK] at(4,1){\small $s^{\alpha+\beta+|i-k|}$};
      \node[above,JK] at(4,-1){\small $s^{\alpha+\beta+|i-k|}$};
      \draw[->,JY,shorten <=13pt, shorten >=9pt, line width=0.7pt] (1,-1) --node[above,pos=0.4] {\small $s^{\alpha+\beta+|i-k|-1}$} (6,0);
      
      \node[JI+] at(11,1){$\bim^1\bim^1$};
      \node[JY] at(11,0){$\bim^2$};
      \node[JK] at(11,-1){$\bim^1\mu^2$};
    \end{tikzpicture}
  \end{equation}
  
\end{proof}

Let us note that $\bim^2$ also has direct geometrical meaning. 
As an illustrative example, consider $a_2\otimes a_1\in\Hom(T_{i-1},T_i)\otimes \Hom(T_{i+1},T_{i-1})$ under braiding,

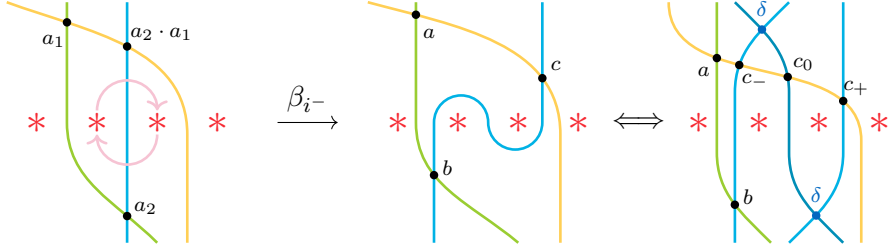
\begin{figure}[H]
  \centering
 \begin{tikzpicture}[scale=0.8]
  \foreach \x in {-5.5,-4.5,-3.5,-2.5}{
    \node[text=M,scale=1.5] at (\x,0) {$\ast$};
  }
  \draw[JI,line width=1pt, name path=T3] (-3,-2)--(-3,-0.2) to[out=90,in=-20] (-6,2);
  \draw[JY,line width=1pt, name path=T1] (-3.5,-2) to[out=135,in=-90] (-5,0)--(-5,2);
  \draw[JK,line width=1pt, name path=T2] (-4,-2)--(-4,2);
  \draw[->] (-1.5,0) --node[above, midway]{$\beta_{i^-}$} (-0.5,0);
  \fill[name intersections={of=T3 and T1}] (intersection-1) circle (2pt) node[xshift=-5pt,yshift=-7pt]{\scriptsize $a_1$};
  \fill[name intersections={of=T1 and T2}] (intersection-1) circle (2pt) node[xshift=7pt,yshift= 3pt]{\scriptsize $a_2$};
  \fill[name intersections={of=T2 and T3}] (intersection-1) circle (2pt) node[xshift=13pt,yshift= 5pt]{\scriptsize $a_2\cdot a_1$};

  \begin{scope}[xshift=6cm]
  \foreach \x in {-5.5,-4.5,-3.5,-2.5}{
    \node[text=M,scale=1.5] at (\x,0) {$\ast$};
  }
  \draw[JI,line width=1pt, name path=BT3] (-2.8,-2)--(-2.8,-0.2) to[out=90,in=-15] (-6,2);
  \draw[JY,line width=1pt, name path=BT1] (-3.5,-2) to[out=140,in=-90] (-5.2,0)--(-5.2,2);
  \draw[JK,line width=1pt, name path=BT2]
    (-4.9,-2) -- (-4.9,0)
    arc[start angle=180,end angle=0,x radius=0.45cm,y radius=0.45cm]
    arc[start angle=180,end angle=360,x radius=0.45cm,y radius=0.45cm] 
    -- (-3.1,2); 
  \fill[name intersections={of=BT3 and BT1}] (intersection-1) circle (2pt) node[xshift=5pt,yshift=-7pt]{\scriptsize $a$};
  \fill[name intersections={of=BT1 and BT2}] (intersection-1) circle (2pt) node[xshift=5pt,yshift= 3pt]{\scriptsize $b$};
  \fill[name intersections={of=BT2 and BT3}] (intersection-1) circle (2pt) node[xshift=5pt,yshift= 5pt]{\scriptsize $c$};
  \end{scope}
  \begin{scope}[xshift=11cm]
  \foreach \x in {-5.5,-4.5,-3.5,-2.5}{
    \node[text=M,scale=1.5] at (\x,0) {$\ast$};
  }
  \draw[JI, name path=T4, line width=1pt]
    (-2.8,-2) --(-2.8,-0.3) to[out=90,in=-90] (-6,2);
  \draw[JY,line width=1pt, name path=T0] (-4.3,-2) to[out=135,in=-90] (-5.2,-0.5)--(-5.2,2);
  \draw[JK,line width=1pt, name path=T1]
    (-4.9,-2) -- (-4.9,0.5) to[out=90,in=-135](-4,2);
  \draw[JK!80!black, name path=T2, line width=1pt]
    (-3.1,-2) to[out=135,in=-90]  (-4,-0.5)--(-4,0.5) to[out=90,in=-45] (-4.9,2);
  \draw[JK,line width=1pt, name path=T3]
    (-3.1,2) -- (-3.1,-0.5) to[out=-90,in=45](-4,-2);
  \fill[name intersections={of=T4 and T0}] (intersection-1) circle (2pt) node[xshift=-5pt,yshift=-5pt]{\scriptsize $a$};
  \fill[name intersections={of=T0 and T1}] (intersection-1) circle (2pt) node[xshift=5pt,yshift= 3pt]{\scriptsize $b$};
  \fill[JO,name intersections={of=T1 and T2}] (intersection-1) circle (2pt)node[yshift=8pt]{\scriptsize $\delta$};
  \fill[JO,name intersections={of=T3 and T2}] (intersection-1) circle (2pt)node[yshift=8pt]{\scriptsize $\delta$};

  \fill[name intersections={of=T4 and T3}] (intersection-1) circle (2pt)
       node[xshift=5pt,yshift=5pt]{\scriptsize $c_+$};
  \fill[name intersections={of=T4 and T1}] (intersection-1) circle (2pt)
       node[xshift=6pt,yshift=-6pt]{\scriptsize $c_-$};
  \fill[name intersections={of=T4 and T2}] (intersection-1) circle (2pt)
       node[xshift=6pt,yshift=5pt]{\scriptsize $c_0$};
  \end{scope}

  \draw[->,SA!50,line width=1pt] (-4.5,0.2) arc[start angle=180,end angle=0,x radius=0.5cm,y radius=0.5cm];
  \draw[->,SA!50,line width=1pt] (-3.5,-0.2) arc[start angle=0,end angle=-180,x radius=0.5cm,y radius=0.5cm];
  \node at (4.5,0) {$\iff$};
 \end{tikzpicture}
 \caption{Action of $\beta_{i^-}$ on $\Hom(T_{i-1},T_i)\otimes \Hom(T_{i+1},T_{i-1})$}
\end{figure}

Before resolving $\bim T_i$, we have a disc $c=b\cdot a$. After resolving $\bim T_i$, the disc $c=ba$ breaks into three smaller discs, we have $c_-=b\cdot a=\bim^1(a_2)\cdot \bim^1(a_1)$ and $c_+=\bim^1(a_2\cdot a_1)$, thus $\bim^1(a_2)\cdot \bim^1(a_1)\ne \bim^1(a_2\cdot a_1)$. Nevertheless, $c_-$ and $c_+$ are cohomologous since $\delta c_0=c_--c_+$ is a coboundary. Thus $c_0=\bim^2(a_2,a_1)$ gives the information of chain homotopy.

\subsubsection{Extension to the full category}

One can check that the $d>2$ equations of \eqref{eq:functor} are satisfied.
Thus, we have constructed a functor\begin{equation}
  \bim \colon \mathcal{C}\to\mathrm{Tw}^b(\mathcal{C}).
\end{equation}
where $\mathrm{Tw}^b(\mathcal{C})\cong\mathrm{Fuk}_{|||}(\M (\bullet,1),\mathcal{W}_{\mathbf{x}})$.
We are then able to extend it canonically to a functor
\begin{equation}
  \bim \colon \mathrm{Tw}^b(\mathcal{C})\to\mathrm{Tw}^b(\mathcal{C}).
\end{equation}
using the construction in \cite[\S 3m]{Seidel-book}.

\begin{remark}\label{rmk:beta}
  In fact, the $A_\infty$-functor $\bim$ that acts on the dg-category of $\mathrm{Perf}(\mathcal{C})$ is only uniquely defined up to contractible choices, there are other totally good choices. For instance, in \eqref{eq:TjTi}, whenever $\alpha\ge 1$, the value of $\bim^1$ on $a_{ij}s^\alpha$ may be chosen as
  \begin{equation}
    \bim^1(a_{ij}s^\alpha)=\left[\lambda a_{i-1,j}s^{\alpha-1},\;(1-\lambda) a_{i+1,j}s^{\alpha}\right]
  \end{equation}
  for any $\lambda\in\Bbbk$. Geometrically, this corresponds to stretching $T_j$ across $T_i$ and $T_{i-1}$ in Fig.~\ref{fig:beta1a}, thereby producing the corresponding discs. Different choices of $\lambda$ lead to different representatives of $\bim^1$, but these ambiguities are compensated by appropriate corrections in $\bim^2$, so that the resulting $A_\infty$-functor remains homotopy equivalent. Throughout, we use the simplest such representative.
\end{remark}

\subsubsection{Realization as a bimodule}

The $(\mathcal{C},\mathcal{C})$-bimodule $\mathfrak{B}_{i^-}$ corresponding to the functor $\beta_{i^-}$ admits a diagrammatic description as the crossing of the red strands associated with the punctures $x_i$ and $x_{i+1}$ \cite{Webster1, Webster2}.
An element in the bimodule is a strand diagram, considered up to isotopy and the relations in Figs.~\ref{fig:KLRW} and \ref{fig:Crossing}. The bimodule action is given by vertical stacking of KLRW diagrams on the top and bottom.

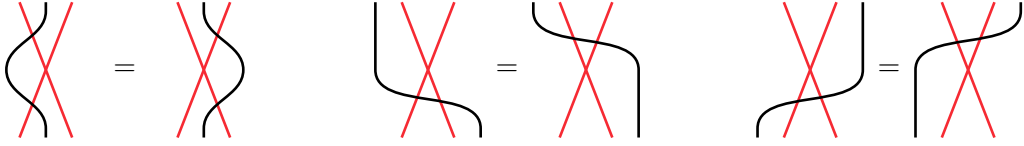
\begin{figure}[H]
  \centering
  \begin{subfigure}[b]{0.3\textwidth}
  \centering
	\begin{tikzpicture}[xscale=0.7, yscale=0.6]
  \draw[M,line width=1pt] (-0.5,-1.5)--(0.5,1.5);
  \draw[M,line width=1pt] (0.5,-1.5)--(-0.5,1.5);
  \draw[black,line width=1pt]
    (0,-1.5) -- (0,-1.3) to[out=90,in=-90] (-0.75,0)
    to[out=90,in=-90] (0,1.3) -- (0,1.5);
  \node at (1.5,0) {$=$};
  \draw[M,line width=1pt] (2.5,-1.5)--(3.5,1.5);
  \draw[M,line width=1pt] (3.5,-1.5)--(2.5,1.5);
  \draw[black,line width=1pt]
    (3,-1.5) -- (3,-1.3) to[out=90,in=-90] (3.75,0)
    to[out=90,in=-90] (3,1.3) -- (3,1.5);
  \end{tikzpicture}
  \end{subfigure}
  \begin{subfigure}[b]{0.3\textwidth}
  \centering
	\begin{tikzpicture}[xscale=0.7, yscale=0.6]
  \draw[M,line width=1pt] (-0.5,-1.5)--(0.5,1.5);
  \draw[M,line width=1pt] (0.5,-1.5)--(-0.5,1.5);
  \draw[black,line width=1pt]
    (1,-1.5) -- (1,-1.3) to[out=90,in=-90] (-1,0) -- (-1,1.5);
  \node at (1.5,0) {$=$};
  \draw[M,line width=1pt] (2.5,-1.5)--(3.5,1.5);
  \draw[M,line width=1pt] (3.5,-1.5)--(2.5,1.5);
  \draw[black,line width=1pt]
    (4,-1.5) -- (4,0) to[out=90,in=-90] (2,1.3)  -- (2,1.5);
  \end{tikzpicture}
  \end{subfigure}
  \begin{subfigure}[b]{0.3\textwidth}
  \centering
	\begin{tikzpicture}[xscale=0.7, yscale=0.6]
    \draw[M,line width=1pt] (-0.5,-1.5)--(0.5,1.5);
    \draw[M,line width=1pt] (0.5,-1.5)--(-0.5,1.5);
  \draw[black,line width=1pt]
    (-1,-1.5) -- (-1,-1.3) to[out=90,in=-90] (1,0) -- (1,1.5);
  \node at (1.5,0) {$=$};
  \draw[M,line width=1pt] (2.5,-1.5)--(3.5,1.5);
  \draw[M,line width=1pt] (3.5,-1.5)--(2.5,1.5);
  \draw[black,line width=1pt]
    (2,-1.5) -- (2,0) to[out=90,in=-90] (4,1.3) -- (4,1.5);
  \end{tikzpicture}
  \end{subfigure}
  \caption{Additional relations for braiding bimodules}
  \label{fig:Crossing}
\end{figure}

To see why this gives negative braiding, consider $\mathfrak{B}_{i^-}(-,T_i)$ as a one-side module, resolving this gives
\begin{equation}
  \begin{tikzpicture}[line width=0.7pt]
  \node at (0,0) {$\oplus$};
  \node at (-4,0) {\tikz[line width=0.7pt,scale=0.5]{
  \draw[gray] (0,1) rectangle (4,-1);
    \foreach \x in {0.5,1.5,2.5,3.5}{
    \draw[M] (\x,-1)--(\x,1);
  }
  \fill (2,1) circle (3pt);}
  };
  \node at (0,2) {\tikz[line width=0.7pt,scale=0.5]{
  \draw[gray] (0,1) rectangle (4,-1);
    \foreach \x in {0.5,1.5,2.5,3.5}{
    \draw[M] (\x,-1)--(\x,1);
  }
  \fill (1,1) circle (3pt);}
  };
  \node at (0,-2) {\tikz[line width=0.7pt,scale=0.5]{
  \draw[gray] (0,1) rectangle (4,-1);
    \foreach \x in {0.5,1.5,2.5,3.5}{
    \draw[M] (\x,-1)--(\x,1);
  }
  \fill (3,1) circle (3pt);}
  };
  \node at (6,0) {\tikz[line width=0.7pt,scale=0.5]{
  \draw[gray] (0,1) rectangle (4,-1);
    \foreach \x in {0.5,3.5}{
    \draw[M] (\x,-1)--(\x,1);
  }
    \draw[M] (1.5,-1)--(2.5,1);
    \draw[M] (2.5,-1)--(1.5,1);
  \fill (2,1) circle (3pt);}};
\draw[->,shorten <=50pt, shorten >=50pt, line width=0.7pt] (-5,0) --  (0,2.5);
\draw[->,shorten <=50pt, shorten >=50pt, line width=0.7pt] (-5,0) --  (0,-2.5);
\draw[->,shorten <=50pt, shorten >=65pt, line width=0.7pt] (0,2.5) --  (7,0);
\draw[->,shorten <=50pt, shorten >=65pt, line width=0.7pt] (0,-2.5) --  (7,0);

  \node at (-3,1.5) {\tikz[scale=2, line width=0.7pt]{
  \foreach \x in {0.05,0.15,0.25,0.35} {
      \draw[M] (\x,-0.1) -- (\x,0.1);
  }
  \draw[black]
      (0.2,-0.1) -- (0.2,-0.08)
      to[out=90,in=-90] (0.1,0.08)
      -- (0.1,0.1);}};
  \node[xshift=-5] at (-3,-1.5) {$-$\tikz[scale=2, line width=0.7pt]{
  \foreach \x in {0.05,0.15,0.25,0.35} {
      \draw[M] (\x,-0.1) -- (\x,0.1);
  }
  \draw[black]
      (0.2,-0.1) -- (0.2,-0.08)
      to[out=90,in=-90] (0.3,0.08)
      -- (0.3,0.1);}};
  \node at (4,1.5) {\tikz[scale=2, line width=0.7pt]{
  \foreach \x in {0.05,0.35} {
      \draw[M] (\x,-0.1) -- (\x,0.1);
  }
      \draw[M] (0.15,-0.1) -- (0.25,0.1);
      \draw[M] (0.25,-0.1) -- (0.15,0.1);
  \draw[black]
      (0.1,-0.1) -- (0.1,-0.06)
      to[out=90,in=-90] (0.2,0.1);}};
  \node at (4,-1.5) {\tikz[scale=2, line width=0.7pt]{
  \foreach \x in {0.05,0.35} {
      \draw[M] (\x,-0.1) -- (\x,0.1);
  }
      \draw[M] (0.15,-0.1) -- (0.25,0.1);
      \draw[M] (0.25,-0.1) -- (0.15,0.1);
  \draw[black]
      (0.3,-0.1) -- (0.3,-0.06)
      to[out=90,in=-90] (0.2,0.1);}};
  \draw[gray] (-5.5,3)--(-6,3)--(-6,-3)--(-5.5,-3);
  \draw[gray] (2,3)--(2.5,3)--(2.5,-3)--(2,-3);
  \end{tikzpicture}
\end{equation}
where the boxes mean all possible diagrams with the specified ends, and the maps mean appending the diagrams on the top. The resolution is precisely the Yoneda embedding $\Hom (-,\beta_{i^-}T_i)$. Therefore,
\begin{equation}
  \mathfrak{B}_{i^-}(-,T_i)=\Hom(-,\beta_{i^-}T_i)
\end{equation}
which is expected. Here, $\mathfrak{B}_{i^-}$ is treated as an object in the abelian category of $\mathcal{C}$-bimodules. The higher component of braiding functors can be realized by resolving the corresponding one-side modules. Here in the resolution, one also finds the same choices encountered in Remark~\ref{rmk:beta}.

There is a monomorphism of bimodules

\begin{equation}
  \begin{tikzcd}
    \iota_i:&\mathfrak{B}_{i^-}\ar[r,hookrightarrow]&\Delta\\
    &\tikz[scale=4, line width=0.7pt]{
  \foreach \x in {0.05,0.35} {
      \draw[M] (\x,-0.15) -- (\x,0.15);
  }
      \draw[M] (0.15,-0.15) -- (0.25,0.15);
      \draw[M] (0.25,-0.15) -- (0.15,0.15);
  \draw[black]
      (0.2,-0.15) -- (0.2,-0.13)
      to[out=90,in=-90] (0.275,0)
      to[out=90,in=-90] (0.2,0.13)--(0.2,0.15);}\ar[r,mapsto]
  &\tikz[scale=4, line width=0.7pt]{
  \foreach \x in {0.05,0.15,0.25,0.35} {
      \draw[M] (\x,-0.15) -- (\x,0.15);
  }
  \draw[black]
      (0.2,-0.15) --(0.2,0.15);
  \fill (0.2,0) circle (0.5pt);}
  \end{tikzcd}
\end{equation}
and it acts on other elements in the obvious way. The cokernel of this monomorphism
\begin{equation}
  \begin{tikzcd}
    0\ar[r]&\mathfrak{B}_{i^-}\ar[r,"\iota_i"]&\Delta\ar[r,"\pi_i"]&\mathfrak{S}_i\ar[r]&0
  \end{tikzcd}\label{eq:coker}
\end{equation}
is generated by only a single diagram $\textit{\k{e}}_i$ in Fig.~\ref{fig:ogonek}.

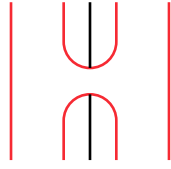
\begin{figure}[H]
  \centering
  \begin{tikzpicture}[scale=.7]
  \foreach \x in {0.5,3.5} {
    \draw[M,line width=1pt] (\x,-1.5) -- (\x,1.5);
  }
  \draw[M,line width=1pt](1.5,-1.5)--(1.5,-0.75)
  arc[start angle=180,end angle=0,x radius=0.5cm,y radius=0.5cm]
  --(2.5,-1.5);
  \draw[M,line width=1pt](1.5,1.5)--(1.5,0.75)
  arc[start angle=180,end angle=360,x radius=0.5cm,y radius=0.5cm]
  --(2.5,1.5);
  \draw[black,line width=1pt](2,-1.5)--(2,-0.25);
  \draw[black,line width=1pt](2,1.5)--(2,0.25);
  \end{tikzpicture}
  \caption{The only diagram $\textit{\k{e}}_i$ in $\mathfrak{S}_i=\operatorname{coker}\iota_i$}
  \label{fig:ogonek}
\end{figure}

Multiplication in $\mathfrak{S}_i$ is given by
\begin{equation}
  a_{ki}s^\beta\cdot \textit{\k{e}}_i\cdot a_{ij}s^{\alpha}=\begin{cases}
    \textit{\k{e}}_i,&j=k=i,\; \alpha =\beta =0,\\
    0,& \text{else}. 
  \end{cases}
\end{equation}
This can be summarized diagrammatically as Fig.~\ref{fig:coker} \cite{Webster2}.

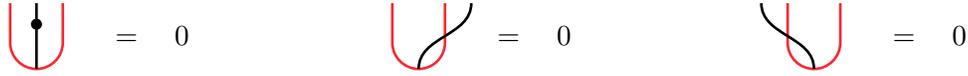
\begin{figure}[H]
  \centering
  \begin{subfigure}[b]{0.3\textwidth}
  \centering
	\begin{tikzpicture}[scale=0.7]
  \draw[M,line width=1pt](-0.5,1.5)--(-0.5,0.75)
  arc[start angle=180,end angle=360,x radius=0.5cm,y radius=0.5cm]
  --(0.5,1.5);
  \draw[black,line width=1pt]
    (0,1.5) -- (0,0.25);
  \fill[black] (0,1.1) circle (3pt);
  \node at (2.2,.875) {$=\quad 0$};
  \end{tikzpicture}
  \end{subfigure}
  \begin{subfigure}[b]{0.3\textwidth}
  \centering
	\begin{tikzpicture}[scale=0.7]
  \draw[M,line width=1pt](-0.5,1.5)--(-0.5,0.75)
  arc[start angle=180,end angle=360,x radius=0.5cm,y radius=0.5cm]
  --(0.5,1.5);
  \draw[black,line width=1pt]
    (1,1.5)  to[out=-90,in=90] (0,0.25);
  \node at (2.2,.875) {$=\quad 0$};
  \end{tikzpicture}
  \end{subfigure}
  \begin{subfigure}[b]{0.3\textwidth}
  \centering
	\begin{tikzpicture}[scale=0.7]
  \draw[M,line width=1pt](-0.5,1.5)--(-0.5,0.75)
  arc[start angle=180,end angle=360,x radius=0.5cm,y radius=0.5cm]
  --(0.5,1.5);
  \draw[black,line width=1pt]
    (-1,1.5) to[out=-90,in=90] (0,0.25);
  \node at (2.2,.875) {$=\quad 0$};
  \end{tikzpicture}
  \end{subfigure}
  \caption{Relations for $\mathfrak{S}_i$}
  \label{fig:coker}
\end{figure}

The epimorphism $\pi_i$ is given by
\begin{equation}
  \begin{tikzcd}
    \pi_i:&\Delta\ar[r, two heads]& \mathfrak{S}_i\\
    &\tikz[scale=4, line width=0.7pt]{
  \foreach \x in {0.05,0.15,0.25,0.35} {
      \draw[M] (\x,-0.15) -- (\x,0.15);
  }
  \draw[black]
      (0.2,-0.15) --(0.2,0.15);}\ar[r,mapsto]
  &\tikz[scale=4, line width=0.7pt]{\foreach \x in {0.05,0.35} {
    \draw[M] (\x,-0.15) -- (\x,0.15);
  }
  \draw[M](.15,-.15)--(.15,-0.075)
  arc[start angle=180,end angle=0,x radius=0.05cm,y radius=0.05cm]
  --(.25,-.15);
  \draw[M](.15,.15)--(.15,0.075)
  arc[start angle=180,end angle=360,x radius=0.05cm,y radius=0.05cm]
  --(.25,.15);
  \draw[black](.2,-.15)--(.2,-0.025);
  \draw[black](.2,.15)--(.2,0.025);}
  \end{tikzcd}
\end{equation}
and it maps everything else to $0$.

\begin{remark}
  This is not a coincidence. In general, the situation can be described as follows. 
  Suppose we have a Lefschetz fibration $f\colon Y \to \C$ with a singular value at $0$. 
  On the Fukaya category of a regular fiber $\mathrm{Fuk}\,(f^{-1}(1))$, there are two twisting exact triangles \cite{Seidel-Thomas}:
  \[
    S \otimes \Hom(S, -) \longrightarrow \mathrm{id} \longrightarrow T_{\circlearrowleft} \xrightarrow{[1]}
  \]
  and
  \[
    T_{\circlearrowright} \longrightarrow \mathrm{id} \longrightarrow S \otimes \Hom(-, S)^{\vee} \xrightarrow{[1]} .
  \]
  Here $T_{\circlearrowleft}$ denotes the monodromy obtained by going around $0$ counterclockwise, and $S \subset f^{-1}(1)$ is the vanishing cycle.
  The second triangle can be viewed as applying the clockwise monodromy $T_{\circlearrowright}$ to the first one, together with the usual identification 
  $\Hom(-, S)^{\vee} \simeq \Hom(S, -)[\pm1]$.

  In Aganagic's setup, one considers a one-parameter family of spaces where the puncture positions $\mathbf{x}=\{x_i\}$ vary with a parameter $t$. Consider
  \[
    (x - x_i)(x - x_{i+1}) = x^2 - t,
  \]
  while other punctures stay fixed.  
  The resulting family defines a generalized Lefschetz fibration 
  $\tilde{\pi} : \tilde{Y} \to \mathbb{C}_t$,
  applying the twisting exact triangles gives \eqref{eq:coker}, where the vanishing cycle is $I_i$ in \eqref{eq:Ii}.
\end{remark}

The construction of the positive braiding functors $\beta_{i^+}$ proceeds identically. For the moment, we omit it here.

\section{$A_\infty$-natural transformations and Hochschild cohomology}

In this section we recall the definition of $A_\infty$-natural transformations, describe their differential, explain their interpretation in terms of Hochschild cohomology and bimodules, and clarify our notation conventions.

For two $A_\infty$ categories $\mathcal{A}$ and $\mathcal{B}$, the functors from $\mathcal{A}$ to $\mathcal{B}$ form another $A_\infty$-category $\Fun(\mathcal{A},\mathcal{B})$. A morphism $\eta\in \Hom^g_{\Fun(\mathcal{A},\mathcal{B})}(F,G)$ is called a \emph{pre-natural transformation}, which consists of a series of multilinear maps $\{\eta^d\}$ for $d\ge 0$:
\begin{equation}
  \eta^d\colon \Hom_{\mathcal{A}}(X_{d-1},X_d)\otimes\cdots\otimes\Hom_{\mathcal{A}}(X_0,X_1)\to\Hom_{\mathcal{B}}(FX_0,GX_d)[g-d],
\end{equation}
where each $X_i\in\mathrm{Ob}\mathcal{A}$.

There are $A_\infty$-structure maps $\mu^k_{\Fun(\mathcal{A},\mathcal{B})}$. In particular, the differential $\mu^1_{\Fun(\mathcal{A},\mathcal{B})}$ acts on $\eta$ by
\begin{equation}\label{eq:mu1-natural}
\begin{aligned}
&\mu_{\Fun(\mathcal{A},\mathcal{B})}^1(\eta)^d(a_d,\dots,a_1)\\
&\quad=
\sum_{r,i}
\sum_{s_1+\cdots+s_r=d}
(-1)^{\dagger}
\mu_\mathcal{B}^r\left(\begin{aligned}
  &G^{s_r}(a_d,\dots,a_{d-s_r+1}),
  \dots,
  G^{s_{i+1}}(\dots,a_{s_i+\cdots+s_i+1}),
  \\
  &\quad\eta^{s_i}(a_{s_i+\cdots+s_i},\dots,a_{s_{i-1}+1}),\\
  &\qquad F^{s_{i-1}}(a_{s_1+\cdots+s_{i-1}},\dots),\dots,
  F^{s_1}(a_{s_1},\dots,a_1)
\end{aligned}
\right) \\
&\qquad -
\sum_{m,n}
(-1)^{\maltese_n+|\eta|-1}
\eta^{d-m+1}(
  a_d,\dots,a_{n+m+1},
  \mu_\mathcal{A}^m(a_{n+m},\dots,a_{n+1}),
  a_n,\dots,a_1
),
\end{aligned}
\end{equation}
where $\dagger=(|\eta|-1)(|a_1|+\cdots+|a_{s_1+\cdots+s_{i-1}}|-s_1-\cdots -s_{i-1})$.
\begin{definition}
  A degree $g$ $A_\infty$-natural transformation $\eta\colon F\Rightarrow G$ is a pre-natural transformation $\eta\in\Hom^g_{\Fun(\mathcal{A},\mathcal{B})}(F,G)$ satisfying the cocycle condition
  \begin{equation}
    \mu_{\Fun(\mathcal{A},\mathcal{B})}^1(\eta)=0.
  \end{equation}
\end{definition}

When $\mathcal{B}=\mathcal{A}$ and $F=\mathrm{id}_\mathcal{A}$, the cochain complex $(\Hom^\bullet_{\Fun(\mathcal{A},
\mathcal{A})}(\mathrm{id},G),\mu_{\Fun(\mathcal{A},\mathcal{A})}^1)$ is the Hochschild cochain $\mathrm{CC}^\bullet(G)$, and natural transformations are the Hochschild cocycles.

The Hochschild cohomology $\mathrm{HH}^\bullet(G)=H^\bullet\Hom (\mathrm{id},G)$ captures the essential structure of the space of natural transformations. In the following sections, we will first determine $\mathrm{HH}^\bullet(\mathrm{id})$ and $\mathrm{HH}^\bullet(\bim)$, and then write out the details of the natural transformations.

By Morita equivalence, we have
\begin{equation}
  \Hom_{\Fun(\mathrm{Tw}^b(\mathcal{C}),\mathrm{Tw}^b(\mathcal{C}))}(F,G)
  \simeq 
  \mathbf{RHom}_{\mathcal{C}\text{-mod-}\mathcal{C}}(B_F,B_G).  
\end{equation}

Hence the Hochschild cohomology may be computed as the derived Hom complex from the diagonal bimodule to the target bimodule.
To carry out this computation, one needs a projective resolution of the diagonal bimodule.
The Hochschild complex is defined canonically on the bar resolution:
\begin{definition}\label{def:Bar}
The \emph{bar resolution} of the diagonal bimodule 
$\Delta_{\mathcal{C}}$ is the $A_\infty$-bimodule complex
\begin{equation}
  \mathrm{Bar}_d(\mathcal{C})
  = (\mathcal{C}[1])^{\otimes d},
\end{equation}
whose elements are written as tuples $(a_d,\dots,a_1)$.
The differential
\begin{equation}
  \partial_{\mathrm{Bar}}\colon \mathrm{Bar}_d(\mathcal{C}) \to \mathrm{Bar}_{d-1}(\mathcal{C})
\end{equation}
is defined by inserting the $A_\infty$-structure maps $\mu_{\mathcal{C}}^k$ 
into all consecutive subsequences:
\begin{equation}\label{eq:bar-diff}
\begin{aligned}
  \partial_{\mathrm{Bar}}(a_d,\dots,a_1)
  &= 
  \sum_{m,n}
  (-1)^{\maltese_n}
  (a_d,\dots,a_{m+n+1},
  \mu_{\mathcal{C}}^m(a_{m+n},\dots,a_{n+1}),
  a_n,\dots,a_1),
\end{aligned}
\end{equation}
and
$\mathrm{Bar}_\bullet(\mathcal{C}) \xrightarrow{\ \epsilon\ }\Delta_{\mathcal{C}}$
is a projective resolution of the diagonal bimodule.
\end{definition}

Computing the Hochschild cohomology is known to be a hard problem because, although the bar resolution is always available in principle, its size $\mathrm{Bar}_n=\mathcal{C}\otimes(\mathcal{C}[1])^{\otimes n}\otimes \mathcal{C}$ grows exponentially with $n$, making it computationally infeasible in practice.  
The main drawback of the bar resolution is that it ignores the specific algebraic relations present in $\mathcal{C}$.
To obtain a more tractable model for actual calculations, one seeks a projective resolution of the diagonal bimodule that makes better use of the intrinsic structure of $\mathcal{C}$.

\section{The Chouhy--Solotar resolution}
In 1997, M.~Bardzell constructed a projective resolution for monomial algebras~\cite{B}. 
Later, in 2014, S.~Chouhy and A.~Solotar extended this construction to quiver algebras with relations~\cite{CS}.
As we will see momentarily, for our category $\mathcal{C}$, the main advantage of this resolution is that the size of each projective bimodule stabilizes as $n$ increases, making the computation of Hochschild cohomology feasible in all degrees.
We summarize the Chouhy--Solotar resolution below.

\subsection{Reduction system}
Let $Q$ be a finite quiver with $Q_0$ the set of vertices and $Q_1$ the set of arrows. Denote by $Q_n$ the set of paths of length $n$. To be consist with the notations in $A_\infty$-categories, our convention for a path $u=u_{k-1}\dots u_1u_0\in Q_k$ where $u_i\in Q_1$ is composing from the right to left, i.e. $\begin{tikzcd}[column sep=small]
	\bullet&\bullet\cdots\bullet\ar[l,"u_{k-1}"']&\bullet\ar[l,"u_1"']&\bullet\ar[l,"u_0"']
\end{tikzcd}$.
We write $vu$ to denote the concatenation of paths, that is, 
the path obtained by connecting the target of $u$ with the source of $v$. 
We write $v \subset u$ if $v$ is a subpath of $u$.

\begin{definition}
The \emph{path algebra} $\Bbbk Q$ of a quiver $Q$ over a field $\Bbbk$ 
is the tensor algebra of the vector space $\Bbbk Q_1$ of arrows over the semisimple algebra 
$\Bbbk Q_0$ spanned by the vertices:
\begin{equation}
  \Bbbk Q
  \coloneqq  T_{\Bbbk Q_0}(\Bbbk Q_1)
  = \Bbbk Q_0 \oplus
  \bigoplus_{n\ge 1}(\Bbbk Q_1)^{\otimes_{\Bbbk Q_0} n}
  =: \bigoplus_{n\ge 0}\Bbbk Q_n.
\end{equation}
Each $\Bbbk Q_n$ is the $\Bbbk$-span of paths of length $n$, and multiplication is given by concatenation of paths whenever it is defined.
\end{definition}

Let $I\subset \Bbbk Q$ be a two-sided ideal, referred to as the \emph{relations}.  
The quotient $A=\Bbbk Q/I$ is called the \emph{path algebra of a quiver with relations}.  
We denote by $\pi:\Bbbk Q\to A$ the canonical projection and by $i:A\hookrightarrow \Bbbk Q$ the inclusion of a chosen section.

\begin{definition}
A \emph{reduction system} $R$ for $\Bbbk Q$ consists of a collection of pairs
\begin{equation}
  R=\{(s,\varphi_s)\mid s\in S,\ \varphi_s\in \Bbbk Q\},
\end{equation}
where
\begin{itemize}
  \item $S$ is a subset of paths of length at least $2$, 
  and distinct elements of $S$ share no common subpaths.
  \item For each $s\in S$, the elements $s$ and $\varphi_s$ are \emph{parallel}, meaning they have the same source and target.
  \item Each $\varphi_s$ is \emph{irreducible}, that is, a $\Bbbk$-linear combination of irreducible paths.
\end{itemize}
\end{definition}

Here, a path is called \emph{irreducible} if it does not contain any element of $S$ as a subpath.
We denote by
\[
\mathrm{Irr}_S(Q)
= Q_\bullet \setminus Q_\bullet S Q_\bullet
\]
the set of all such irreducible paths.
Each element of $\mathrm{Irr}_S(Q)$ serves as a representative of a coset in the quotient algebra
$A=\Bbbk Q/I$.
Hence, as a $\Bbbk$-vector space we have the natural isomorphism
\begin{equation}\label{cong}
  \Bbbk\mathrm{Irr}_S(Q)\cong A=\Bbbk Q/I.
\end{equation}

Given a reduction system $R$, we associate to it the ideal of relations
\[
I=\langle s-\varphi_s \mid s\in S\rangle
\]
and the corresponding path algebra with relations
$A=\Bbbk Q/I$.

\begin{remark}
Following \cite{BW}, we write $\otimes$ for $\otimes_{\Bbbk Q_0}$ and
$\mathrm{Hom}$ for $\mathrm{Hom}_{\Bbbk Q_0^{\mathrm e}}$, 
where $\Bbbk Q_0^{\mathrm e}=\Bbbk Q_0\otimes_{\Bbbk}\Bbbk Q_0^{\mathrm{op}}$.
\end{remark}

\begin{definition}
Finally, we define a relation~$\preceq$ on the set
\[
  \Bbbk^\times Q_{\ge 0} \coloneqq  \{\lambda p \mid \lambda\in\Bbbk^\times,\; p\in Q_{\ge 0}\}\cup\{0\}.
\]
We define~$\preceq$ as the least reflexive and transitive relation satisfying that
$\lambda p \preceq \mu q$ whenever the $\mu q$ can be reduced, through a finite (including length 0) sequence of applications of the reduction system, that is, by successively replacing a subpath $s \subset q$ with $\varphi_s$ for some $s \in S$, to a linear combination $\lambda p + x$ in which $p$ does not appear among the terms of $x$.
In addition, we write $0 \preceq \lambda p$ for all $\lambda p\in \Bbbk^\times Q_{\ge 0}$.
\end{definition}

\subsection{Ambiguities}

\begin{definition}
A \emph{(right) $n$-ambiguity} is a path $u=u_{n+1}\cdots u_1u_0$ satisfying the following conditions:
\begin{itemize}
  \item For every $i$, the concatenation $u_{i+1}u_i$ belongs to $S$;
  \item For every factorization $u_{i+1}=rq$ with $r\ne 0$, one has $qu_i \notin S$.
\end{itemize}
\end{definition}

\begin{remark}
Left $n$-ambiguities are defined analogously.  
It can be shown that left and right ambiguities are equivalent notions.
\end{remark}

Such paths are called \emph{ambiguities} because they admit multiple possible reductions.  
A reduction system is said to be \emph{reduction unique} if all possible reduction sequences of a path yield the same irreducible result.  
In what follows, all quivers under consideration are assumed to be reduction unique.

For convenience, we set $S_0=Q_0$, $S_1=Q_1$, $S_2=S$, and for $n>2$, 
let $S_n$ denote the set of $(n-2)$-ambiguities.

\begin{definition}
For each $n\ge 0$, define a $\Bbbk$-linear map
\[
  \mathrm{split}_n:\Bbbk Q\longrightarrow A\otimes \Bbbk S_n\otimes A
\]
by the formula
\begin{equation}
  \mathrm{split}_n(w)
  \coloneqq  \sum_{\substack{w=urv\\r\in S_n}}
  \pi(u)\otimes r\otimes \pi(v),
\end{equation}
where the sum runs over all decompositions of $w$ through a subpath $r\in S_n$. In addition, define
\begin{align}
  &\mathrm{split}^\mathrm{R}(w)=\pi(u_{\mathrm{R}})\otimes r_{\mathrm{R}}\otimes \pi(v_{\mathrm{R}})\\
  &\mathrm{split}^\mathrm{L}(w)=\pi(u_{\mathrm{L}})\otimes r_{\mathrm{L}}\otimes \pi(v_{\mathrm{L}})
\end{align}
where $u_\mathrm{R}r_\mathrm{R}v_\mathrm{R}=u_{\mathrm{L}}r_\mathrm{L}v_{\mathrm{L}}=w$, and $r_\mathrm{R}$ (resp. $r_\mathrm{L}$) is the rightmost (resp. leftmost) subpath of $w$ which lies in $S$.

\end{definition}

\subsection{The projective resolution}\label{1.3}
\label{sect:CSresolution}

We now describe the Chouhy--Solotar projective resolution of $A=\Bbbk Q/I$.
It takes the form of a chain complex
\begin{equation}\label{eq:CScplx}
  \cdots \xrightarrow{\partial_{n+1}} 
  P_n \xrightarrow{\partial_n} 
  P_{n-1} \xrightarrow{\partial_{n-1}} \cdots
  \xrightarrow{\partial_2} 
  P_1 \xrightarrow{\partial_1} 
  P_0 \xrightarrow{\partial_0} 
  A \to 0,
\end{equation}
where $P_n=A\otimes \Bbbk S_n\otimes A$.  
In particular,
\[
  P_0
  =A\otimes \Bbbk S_0\otimes A
  =A\otimes_{\Bbbk Q_0}\Bbbk Q_0\otimes_{\Bbbk Q_0}A
  =A\otimes_{\Bbbk Q_0}A.
\]
Each $P_n$ is an $(A,A)$-bimodule, with left and right actions given by
\begin{equation}
  x\cdot(u\otimes w\otimes v)\cdot y
  \coloneqq  \pi(xu)\otimes w\otimes \pi(vy),
\end{equation}
for $x,y,u,v\in A$ and $w\in \Bbbk S_n$.

\begin{definition}
  \mathleft
  Define maps $\delta_n: P_n \longrightarrow P_{n-1}$ by 
\begin{gather}
  \delta_0(x\otimes y)
  =\pi(xy),\\
  \delta_1(x\otimes w\otimes y)
  =x\cdot (1\otimes w - w\otimes 1)\cdot y
  =x\otimes \pi(wy)-\pi(xw)\otimes y,\\
  \delta_n(x\otimes w\otimes y)
  =
  \begin{cases}
  x\cdot (\mathrm{split}_{n-1}^{\mathrm L}(w)
   -\mathrm{split}_{n-1}^{\mathrm R}(w))\cdot y,
  & n\ge 2\text{ odd},\\[4pt]
  x\cdot \mathrm{split}_{n-1}(w)\cdot y,
  & n\ge 2\text{ even}.
  \end{cases}\label{delta}
\end{gather}
\end{definition}
\begin{definition}
  For $\mu w\in \Bbbk^\times Q_{\ge 0}$, define subsets of $P_n$
  \begin{equation}
    \bar{\mathcal{L}}^{\prec}_n(\mu w)\coloneqq \{
      \lambda u\otimes r\otimes v\mid u,v\in A,\, \lambda urv\prec\mu w
    \}.
  \end{equation}
\end{definition}

One of the main theorem of \cite{CS} is,
\begin{theorem}\label{thm:4.1}
If $\partial_\bullet$ satisfies the following conditions
\begin{itemize}
  \item $\partial_{i-1}\comp\partial_i=0$, $\forall i\in \Z_{\ge 0}$;
  \item $(\partial_i-\delta_i)(1\otimes w\otimes 1)\in\langle\bar{\mathcal{L}}^{\prec}_{i-1}(w)\rangle_{\Bbbk}$, $\forall i\in \Z_{\ge 0}$ and $\forall w\in S_i$;
\end{itemize}
then the complex \eqref{eq:CScplx} is exact.
\end{theorem}

The paper \cite{CS} further proved for a reduction system satisfying some conditions, such $\partial_\bullet$ always exists and gave an explicit construction, where the differentials $\partial_n$ are constructed from the auxiliary maps $\delta_n$ and $\gamma_{n-1},\rho_{n-1}: P_{n-1}\to P_n$, where $\gamma_{n-1}$ is defined by
\mathleft
\begin{gather}
  \gamma_{-1}(x)=1\otimes x,\\
  \gamma_{n-1}(x\otimes w\otimes y)
  =(-1)^n\,  \mathrm{split}_n(xw)\cdot y.
\end{gather}

And the differentials $\partial_n$ and the homotopies $\rho_{n-1}$ are then defined recursively by
\begin{gather}
  \partial_{-1}=0,\qquad \rho_{-2}=0,\\
  \partial_n(x\otimes w\otimes y)
  =x\cdot
  \big((\mathrm{id}-\rho_{n-2}\partial_{n-1})
  \delta_n\big)(1\otimes w\otimes 1)
  \cdot y,\label{par}\\
  \rho_{n-1}
  =\gamma_{n-1}
  +\sum_{i\ge1}
  \gamma_{n-1}
  (\delta_n\gamma_{n-1}-\partial_n\gamma_{n-1})^i.\label{rho}
\end{gather}

\begin{proposition}
For $n=0$, equation~\eqref{par} yields
\begin{align*}
  \partial_0(x\otimes w\otimes y)
  &=x\cdot 
  \big((\mathrm{id}-\rho_{-2}\partial_{-1})
  \delta_0\big)
  (1\otimes w\otimes 1)\cdot y
  \\
  &=x\cdot \delta_0(1\otimes w\otimes 1)\cdot y
  =\pi(xy),
\end{align*}
where $w\in Q_0$ is a vertex of the quiver (the target of $x$ and the source of $y$).  
Thus we have
\begin{equation}\label{par0}
  \partial_0(x\otimes y)
  =\delta_0(x\otimes y)
  =\pi(xy).
\end{equation}
\end{proposition}

\medskip

For detailed proofs and further properties of this construction, we refer to~\cite{CS}.  

\begin{remark}
Briefly, the resolution extends the 1997 Bardzell resolution for monomial algebras~\cite{B}.  
The maps $\delta$ and $\gamma$ coincide with those in the monomial case.  
Indeed, if the algebra is monomial and $S\subset Q_2$, then $\delta^2=0$.  
Specifically,
\begin{align*}
  \delta&_{n-1}\delta_n(1\otimes w\otimes 1)\\
  &=\mathrm{split}_{n-2}^{\mathrm L}\mathrm{split}_{n-1}^{\mathrm L}(w)
   \pm\mathrm{split}_{n-2}^{\mathrm L}\mathrm{split}_{n-1}^{\mathrm R}(w)
   \mp\mathrm{split}_{n-2}^{\mathrm R}\mathrm{split}_{n-1}^{\mathrm L}(w)
   +\mathrm{split}_{n-2}^{\mathrm R}\mathrm{split}_{n-1}^{\mathrm R}(w)\\
  &=\mathrm{split}_{n-2}^{\mathrm L}\mathrm{split}_{n-1}^{\mathrm L}(w)
   +\mathrm{split}_{n-2}^{\mathrm R}\mathrm{split}_{n-1}^{\mathrm R}(w)
   \\
   &=0,
\end{align*}
since the middle two terms cancel, and both the first and last terms factor through monomials in $S$, which vanish in $A$.  
For non-monomial algebras, however, $\delta^2\ne0$, and the differentials must be corrected to $\partial_n$ as above so that $\partial^2=0$.  
\end{remark}

\mathcenter

\section{Projective resolution of the diagonal bimodule}
\subsection{Category $\mathcal{C}$ as a quiver with relations}
We can consider our KLRW category $\mathcal{C}$ of $T_i$ branes as a quiver with relations, where the vertices are the idempotents $e_i$, and the arrows are taken to be $s_i$, $p_i$ and $q_i$ defined in Definition~\ref{def:spq}:

\begin{figure}[H]
  \centering
	\begin{tikzcd}[column sep=large]
		\underset{e_0}{\bullet}\ar[r,bend left=10, "p_1"]  \ar[looseness=8, out=120, in=60,"s_0"near end] 
		& \underset{e_1}{\bullet} \ar[l,bend left=10,"q_0"] \ar[r,bend left=10, "p_2"]\ar[looseness=8, out=120, in=60,"s_1"near end]
		&\underset{e_2}{\bullet} \ar[l,bend left=10,"q_1"] \ar[r,bend left=10, "p_3"] \ar[looseness=8, out=120, in=60,"s_2"near end]
		&\underset{e_3}{\bullet} \ar[l,bend left=10,"q_2"] \ar[r,bend left=10, "p_4"] \ar[looseness=8, out=120, in=60,"s_3"near end]
		&\underset{e_4}{\bullet} \ar[l,bend left=10,"q_3"] \ar[looseness=8, out=120, in=60,"s_4"near end]
	\end{tikzcd}  
  \caption{Category $\mathcal{C}$ as a quiver $Q$}
\end{figure}
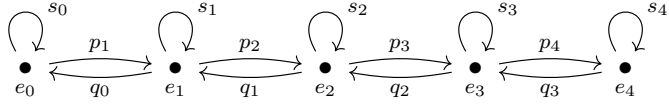
\begin{align}
	Q_0&=\left\{ e_i \right\} \\
	Q_1&=\left\{ p_i, q_i, s_i \right\} 
\end{align}

The reduction system corresponding to the KLRW algebra comes from the relations in Fig.~\ref{fig:KLRW}, where we need to specify one diagram in each relation to be irreducible:
\begin{equation}
  R=\left\{
    \left( q_{i}p_{i+1},s_i \right),\;
    \left( p_{i}q_{i-1},s_i \right),\;
    \left( s_ip_i,p_is_{i-1} \right),\;
    \left( s_iq_i,q_is_{i+1} \right)
  \right\}.
\end{equation}

We will refer to the first two relations as \emph{type I}, and the others as \emph{type II}.
Here we have adopted the convention that a path is considered \emph{irreducible} if its corresponding strand diagram
has all dots placed at the bottom and each black strand is monotone in the horizontal direction,
that is, it never turns back after crossing a red strand. The set $S$ is therefore
\begin{equation}
	S=\left\{ 
    q_{i}p_{i+1},\; p_{i}q_{i-1},\; s_ip_i,\; s_iq_i
	 \right\} 
\end{equation}

Since $S \subset Q_2$, any $n$-ambiguity $u \in S_{n+2}$ must arise from consecutive overlaps of minimal reducible subpaths without intermediate gaps. 
In other words, such a path has the form $u = u_{n+1}\dots u_1 u_0$ with $u_i \in Q_1$ and $ u_{i+1}u_i \in S$, hence $S_n \subset Q_n$. 
Moreover, note that the dots $s_i$ can only terminate a path in $S$. 
Consequently, a dot can appear only at the end of an $n$-ambiguity. 
There are therefore exactly two distinct families of $n$-ambiguities:

\begin{itemize}
  \item \textbf{Type~I ambiguities}: paths oscillating back and forth around a fixed red strand,
  \begin{itemize}
    \item which are in the form of $p_{i+1}q_{i}p_{i+1}q_{i}p_{i+1}q_{i}\dots$, or $q_{i}p_{i+1}q_{i}p_{i+1}q_{i}p_{i+1}\dots$;
  \end{itemize}

  \item \textbf{Type~II ambiguities}: sequences ending at a dot after oscillating around a fixed red strand,
  \begin{itemize}
    \item which are in the form of $s_{i+1}p_{i+1}q_ip_{i+1}q_ip_{i+1}\dots$, or $s_iq_{i}p_{i+1}q_{i}p_{i+1}q_{i}\dots$.
  \end{itemize}
\end{itemize}

\begin{figure}[H]
  \centering
  \begin{subfigure}[b]{0.45\textwidth}
  \centering
	\begin{tikzpicture}[scale=.7]
  \draw[M,line width=1pt] (0.5,-1.5) -- (0.5,6.7);

  \draw[black,line width=1pt]
    (1,-1.5) -- (1,-1.3) 
    to[out=90,in=-90] (0,0)
    to[out=90,in=-90] (1,1.3)
    to[out=90,in=-90] (0,2.6)
    to[out=90,in=-90] (1,3.9)
    to[out=90,in=-90] (0,5.2)
    to[out=90,in=-90] (1,6.5) -- (1,6.7);
  \end{tikzpicture}
  \caption{Type I ambiguity}
  \end{subfigure}
  \begin{subfigure}[b]{0.45\textwidth}
  \centering
	\begin{tikzpicture}[scale=.7]
  \draw[M,line width=1pt] (0.5,-1.5) -- (0.5,6.7);
  \draw[black,line width=1pt]
    (1,-1.5) -- (1,-1.3) 
    to[out=90,in=-90] (0,0)
    to[out=90,in=-90] (1,1.3)
    to[out=90,in=-90] (0,2.6)
    to[out=90,in=-90] (1,3.9)
    to[out=90,in=-90] (0,5.2) -- (0,6.7);
    \fill (0,5.85) circle (3pt);
  \end{tikzpicture}
  \caption{Type II ambiguity}
  \end{subfigure}
  \caption{Two 4-ambiguities in $S_6$}
  \label{fig:ambigue}
\end{figure}
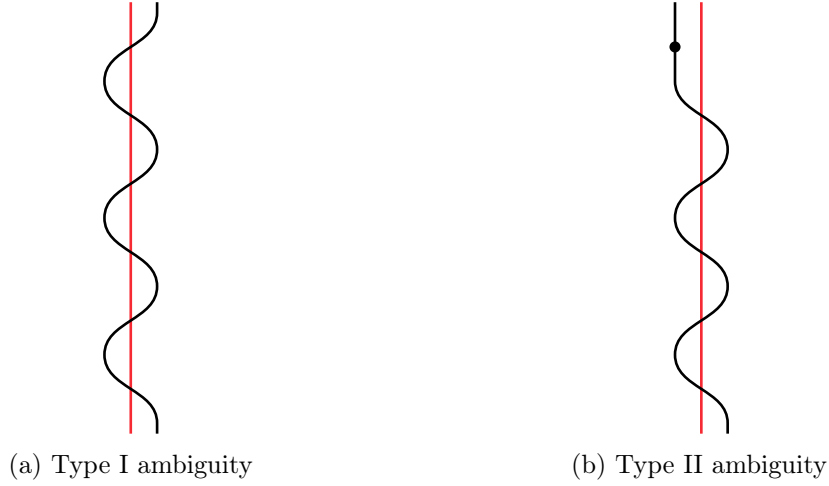

\begin{definition}
  To streamline notation, we denote the ambiguities in $S_n$ as
  \begin{itemize}
    \item Type I ambiguities:\begin{align}
      P^n_{i+1}&\coloneqq p_{i+1}q_{i}p_{i+1}q_{i}p_{i+1}q_{i}\dots\\
      Q^n_{i}&\coloneqq q_{i}p_{i+1}q_{i}p_{i+1}q_{i}p_{i+1}\dots
    \end{align}
    \item Type II ambiguities:\begin{align}
      sP^n_{i+1}&\coloneqq s_{i+1}p_{i+1}q_ip_{i+1}q_ip_{i+1}\dots\\
      sQ^n_{i}&\coloneqq s_iq_{i}p_{i+1}q_{i}p_{i+1}q_{i}\dots
    \end{align}
  \end{itemize}
  This should not cause any confusion with the projective module $P_n$ or the set of paths $Q_n$, as the ambiguities carry two indices.
\end{definition}

We have for the Coulomb branch with $|\mathbf{x}|$ punctures

\begin{equation}
  |S_n| =
  \begin{cases}
    |\mathbf{x}|, & n = 0,\\[4pt]
    3|\mathbf{x}|+1, & n = 1,\\[4pt]
    4|\mathbf{x}|, & n \ge 2.
  \end{cases}
  \label{28}
\end{equation}
For example, in the 4-puncture case, we have $|S_0|=5$, $|S_1|=13$, and $|S_{\ge 2}|=16$.

\begin{remark}
Even before actually starting the computation of the projective resolution, two key advantage of the Chouhy--Solotar resolution already becomes apparent:
\begin{itemize}
  \item The size of each projective bimodule $P_n = A \otimes \Bbbk S_n \otimes A$ remains constant for all $n \ge 2$, 
  whereas in the bar resolution, the size of $P_n$ grows exponentially with $n$. 
  Consequently, all Hochschild cohomology groups can be computed once and for all, since the resolution size does not grow with the cohomological degree $n$.

  \item Moreover, since all ambiguities involve only local paths between two adjacent vertices, 
  the resolution naturally decomposes into smaller, independent components for $n \ge 3$. 
  As a result, computations for arbitrarily many punctures $|\mathbf{x}|$ are uniformly tractable.
\end{itemize}
\end{remark}

\begin{definition}
  To be consistent with the notations in Section~\ref{sect:ajisk}, we denote paths $a_{ji}\in Q_{|i-j|}$
\begin{equation}
  a_{ji}\coloneqq \begin{cases}
    p_{j-1}p_{j-2}\cdots p_{i+1}p_i, & i<j,\\
    e_i, & i=j,\\
    q_{j+1}q_{j+2}\cdots q_{i-1}q_i, & i>j.
  \end{cases}
\end{equation}
\end{definition}

Thus, all irreducible paths have the form of $a_{ji}s^\alpha_i$. The algebra $A=\Bbbk Q/I$, is generated by these irreducible paths, and has multiplication parallel to $\mu^2$ in (\ref{eq:mu2})
\begin{equation}
  a_{kj} s^\beta_j\cdot a_{ji}s^{\alpha}_i = a_{ki} s^{\beta+\alpha +\delta (i,j,k)}_i
\end{equation}
where $\delta(i,j,k)$ is given in (\ref{eq:deltaijk}).

\subsection{Result from the Chouhy--Solotar resolution}

\begin{theorem}
The result of Chouhy--Solotar resolution for our quiver algebra is
	\begin{equation}\label{eq:resolution}
			\cdots \xrightarrow{\partial _{n+1}}P_n\xrightarrow{\partial _n}P_{n-1}\xrightarrow{\partial _{n-1}}\cdots \xrightarrow{\partial _2}P_1\xrightarrow{\partial _1}P_0\xrightarrow{\partial _0}A\rightarrow 0
	\end{equation}
	where $P_n=A\otimes\Bbbk S_n\otimes A$, and $\partial_n$ act by
  \mathleft
  \begin{itemize}
    \item $0\le n\le 1$,
  	\begin{gather}
    \partial _0\left( 1\otimes 1 \right) =1
    \\
    \partial _1\left( 1\otimes w\otimes 1 \right) =1\otimes w-w\otimes 1,\quad \forall w\in Q_1
    \end{gather}
    \item $n\ge 2$ even,
    \begin{gather}
    \partial _{n}\left( 1\otimes Q^{n}_{i}\otimes 1 \right) \label{eq:71}
    \\
    \qquad =1\otimes Q^{n-1}_{i}\otimes p_{i+1}+q_{i}\otimes P^{n-1}_{i+1}\otimes 1-1\otimes sQ^{n-1}_i\otimes 1\notag
    \\
    \partial _{n}( 1\otimes sP^{n}_{i+1}\otimes 1 )
    \\
    \qquad =1\otimes sP^{n-1}_{i+1} \otimes p_{i+1}+s_{i+1}\otimes P_{i+1}^{n-1}\otimes 1 -p_{i+1}\otimes sQ^{n-1}_i\otimes 1-1\otimes P_{i+1}^{n-1}\otimes s_i\notag
    \end{gather}
      \item $n\ge 2$ odd,
    \begin{gather}    
    \partial _{n}\left( 1\otimes Q^{n}_{i}\otimes 1 \right) 
    \\
    \qquad =1\otimes Q^{n-1}_{i}\otimes q_{i}-q_{i}\otimes  P^{n-1}_{i+1}\otimes 1+1\otimes sQ^{n-1}_i\otimes 1\notag
    \\
    \partial _{n}\left( 1\otimes sP^{n}_{i+1}\otimes 1 \right) \label{eq:74}
    \\
    \qquad =1\otimes sP^{n-1}_{i+1}\otimes q_i-s_{i+1}\otimes P^{n-1}_{i+1}\otimes 1 +p_{i+1}\otimes sQ^{n-1}_i\otimes 1+1\otimes P_{i+1}^{n-1}\otimes s_{i+1}\notag
	\end{gather}
  \end{itemize}
  and the equations from \eqref{eq:71}--\eqref{eq:74} by the replacement $P\leftrightarrow Q$, $i\leftrightarrow i+1$.
\end{theorem}

This might seem confusing at first glance, but the structure becomes simple when written diagrammatically. 
For $w\in S_n$, we have $\partial_n(1\otimes w\otimes 1)\in P_{n-1}=A\otimes \Bbbk S_{n-1}\otimes A$. 
Although the detailed construction of the resolution in Section~\ref{sect:CSresolution} is rather intricate, 
the operations involved in defining $\partial_n$ essentially consist of only splitting and reduction operations. 
Formally, we may write
\begin{equation}
  \partial_n(1\otimes w\otimes 1)=\sum_i \lambda_iu_i\otimes  r_i\otimes v_i,
\end{equation}
where $\lambda_i\in\Bbbk$, $u_i,v_i\in A$, and $r_i\in S_{n-1}$. 
Then we have
\begin{equation}\label{eq:prec}
  u_i r_i v_i \preceq w \quad \text{for all } i.
\end{equation}
It is easy to find out that the possible forms of $u_i r_i v_i$ are quite limited—namely, there are only three in the case of type~I ambiguities.

\begin{equation}\label{eq:partialI}
  \begin{tikzpicture}[scale=0.6]
  \draw[M,line width=1pt] (0.5,0) -- (0.5,7);
  \draw[gray] (-0.3,0.5)--(1.3,0.5);
  \draw[gray] (-0.3,6.5)--(1.3,6.5);
  \draw[black,line width=1pt]
    (1,0) -- (1,0.5) 
    to[out=90,in=-90] (0,1.5)
    to[out=90,in=-90] (1,2.5)
    to[out=90,in=-90] (0,3.5)
    to[out=90,in=-90] (1,4.5)
    to[out=90,in=-90] (0,5.5)
    to[out=90,in=-90] (1,6.5) -- (1,7);
  \draw[->,line width=1pt] (2,3.5) -- node[above, midway] {$\partial_n$} (4,3.5);

  \draw[M,line width=1pt] (5.5,0) -- (5.5,7);
  \draw[gray] (4.7,1)--(6.3,1);
  \draw[gray] (4.7,6)--(6.3,6);
  \draw[black,line width=1pt]
    (6,0)
    to[out=90,in=-90] (5,1)
    to[out=90,in=-90] (6,2)
    to[out=90,in=-90] (5,3)
    to[out=90,in=-90] (6,4)
    to[out=90,in=-90] (5,5)
    to[out=90,in=-90] (6,6) -- (6,7);
  \node at (7.5,3.5) {$+(-1)^n$};
  \draw[M,line width=1pt] (9.5,0) -- (9.5,7);
  \draw[gray] (8.7,1)--(10.3,1);
  \draw[gray] (8.7,6)--(10.3,6);
  \draw[black,line width=1pt]
    (10,0)--(10,1)
    to[out=90,in=-90] (9,2)
    to[out=90,in=-90] (10,3)
    to[out=90,in=-90] (9,4)
    to[out=90,in=-90] (10,5)
    to[out=90,in=-90] (9,6)
    to[out=90,in=-90] (10,7);
  \node at (11.5,3.5) {$-(-1)^n$};
  \draw[M,line width=1pt] (13.5,0) -- (13.5,7);
  \draw[gray] (12.7,1)--(14.3,1);
  \draw[gray] (12.7,6)--(14.3,6);
  \draw[black,line width=1pt]
    (14,0)--(14,1)
    to[out=90,in=-90] (13,2)
    to[out=90,in=-90] (14,3)
    to[out=90,in=-90] (13,4)
    to[out=90,in=-90] (14,5)-- (14,7);
  \fill[black] (14,5.5) circle (3pt);
  \node[text=gray] at (-2,0){$A$};
  \node[text=gray] at (-2,0.7){$\otimes$};
  \node[text=gray] at (-2,3.5){$\Bbbk S_n$};
  \node[text=gray] at (-2,6.3){$\otimes$};
  \node[text=gray] at (-2,7){$A$};
  \node[text=gray] at (16,0){$A$};
  \node[text=gray] at (16,1){$\otimes$};
  \node[text=gray] at (16,3.5){$\Bbbk S_{n-1}$};
  \node[text=gray] at (16,6){$\otimes$};
  \node[text=gray] at (16,7){$A$};
  \node[text=gray] at (20,3.5){$\phantom{\Bbbk S_{n-1}}$};
  \end{tikzpicture}
\end{equation}
and four options for type II ambiguities:
\begin{equation}\label{eq:partialII}
  \begin{tikzpicture}[scale=0.6]
  \draw[M,line width=1pt] (0.5,0) -- (0.5,7);
  \draw[gray] (-0.3,0.5)--(1.3,0.5);
  \draw[gray] (-0.3,6.5)--(1.3,6.5);
  \draw[black,line width=1pt]
    (1,0) -- (1,0.5) 
    to[out=90,in=-90] (0,1.5)
    to[out=90,in=-90] (1,2.5)
    to[out=90,in=-90] (0,3.5)
    to[out=90,in=-90] (1,4.5)
    to[out=90,in=-90] (0,5.5) -- (0,7);
  \fill[black] (0,6) circle (3pt);

  \draw[->,line width=1pt] (2,3.5) -- node[above, midway] {$\partial_n$} (4,3.5);

  \draw[M,line width=1pt] (5.5,0) -- (5.5,7);
  \draw[gray] (4.7,1)--(6.3,1);
  \draw[gray] (4.7,6)--(6.3,6);
  \draw[black,line width=1pt]
    (6,0)
    to[out=90,in=-90] (5,1)
    to[out=90,in=-90] (6,2)
    to[out=90,in=-90] (5,3)
    to[out=90,in=-90] (6,4)
    to[out=90,in=-90] (5,5) -- (5,7);
  \fill[black] (5,5.5) circle (3pt);

  \node at (7.5,3.5) {$+(-1)^n$};
  \draw[M,line width=1pt] (9.5,0) -- (9.5,7);
  \draw[gray] (8.7,1)--(10.3,1);
  \draw[gray] (8.7,6)--(10.3,6);
  \draw[black,line width=1pt]
    (10,0)--(10,1)
    to[out=90,in=-90] (9,2)
    to[out=90,in=-90] (10,3)
    to[out=90,in=-90] (9,4)
    to[out=90,in=-90] (10,5)
    to[out=90,in=-90] (9,6)--(9,7);
  \fill[black] (9,6.5) circle (3pt);

  \node at (11.5,3.5) {$-(-1)^n$};
  \draw[M,line width=1pt] (13.5,0) -- (13.5,7);
  \draw[gray] (12.7,1)--(14.3,1);
  \draw[gray] (12.7,6)--(14.3,6);
  \draw[black,line width=1pt]
    (14,0)--(14,1)
    to[out=90,in=-90] (13,2)
    to[out=90,in=-90] (14,3)
    to[out=90,in=-90] (13,4)
    to[out=90,in=-90] (14,5)-- (14,6)
    to[out=90,in=-90] (13,7);
  \fill[black] (14,5.5) circle (3pt);

  \node at (15.5,3.5) {$-(-1)^n$};
  \draw[M,line width=1pt] (17.5,0) -- (17.5,7);
  \draw[gray] (16.7,1)--(18.3,1);
  \draw[gray] (16.7,6)--(18.3,6);
  \draw[black,line width=1pt]
    (18,0)--(18,1)
    to[out=90,in=-90] (17,2)
    to[out=90,in=-90] (18,3)
    to[out=90,in=-90] (17,4)
    to[out=90,in=-90] (18,5)
    to[out=90,in=-90] (17,6)--(17,7);
  \fill[black] (18,0.5) circle (3pt);

  \node[text=gray] at (-2,0){$A$};
  \node[text=gray] at (-2,0.7){$\otimes$};
  \node[text=gray] at (-2,3.5){$\Bbbk S_n$};
  \node[text=gray] at (-2,6.3){$\otimes$};
  \node[text=gray] at (-2,7){$A$};
  \node[text=gray] at (20,0){$A$};
  \node[text=gray] at (20,1){$\otimes$};
  \node[text=gray] at (20,3.5){$\Bbbk S_{n-1}$};
  \node[text=gray] at (20,6){$\otimes$};
  \node[text=gray] at (20,7){$A$};
  \end{tikzpicture}
\end{equation}

The first two terms in both (\ref{eq:partialI}) and (\ref{eq:partialII}) are the na\"ive terms, where $u_iv_ir_i=w$. The others are nontrivial, where $u_iv_ir_i\prec w$. Writing (\ref{eq:partialI}) and (\ref{eq:partialII}) in formulae, we get precisely (\ref{eq:71}) -- (\ref{eq:74}).

\begin{proof}
According to Theorem~\ref{thm:4.1}, since our resolution satisfies condition~(\ref{eq:prec}), the second assumption is satisfied. The only remaining point to check is that $\partial^2=0$, which can be verified by a straightforward computation.  
\end{proof}

\section{Hochschild cohomology of the braiding functors}

We now proceed to derive the Hochschild cohomology of braiding functors. For a general braiding functor $\beta$, its Hochschild cohomology is isomorphic to the complex $(\Hom^{\bullet -d}(P_d,B_\beta),\d_{\bullet})$, where $B_\beta$ is the corresponding bimodule and $\d_\bullet$ is the induced differential.

\subsection{$H^\bullet\Hom(\mathrm{id},\mathrm{id})$}

For the identity functor, on $\mathcal{C}$ we can only have degree $0$ morphisms, thus the cochain complex is given by
\begin{equation}
  \cdots\to\Hom (P_{n-1},\Delta)\xrightarrow{\d_{n-1}}\Hom (P_n,\Delta)\xrightarrow{\d_n}\Hom (P_{n+1},\Delta)\to\cdots
\end{equation}
where $\d_n\phi=\phi\comp\partial_{n+1}$.

\subsubsection{The cochain complex}

Since $P_n=A\otimes \Bbbk S_n\otimes A$ is a free bimodule, to define a homomorphism $\phi \colon P_n\to \Delta$, it suffices to specify the images of the generators $S_n$. For a generator $g\in S$ whose target $T_i$ and source $T_j=T_i,T_{i\pm 1}$, the element $\phi(g)$ may be an arbitrary morphism in $\Hom (T_j, T_i)$.
Using the basis $\left\{a_{ij}s^{\ell}\right\}_{\ell\ge 0}$ of this morphism space, we may write
\begin{equation}
  \phi(g)=\sum_{\ell=0}^\infty \phi_{\ell}(g)a_{ij}s^{\ell}
\end{equation}
where each $\phi_{\ell}(g)\in \Bbbk$. Thus the data of a bimodule homomorphism \[\phi\colon P_n\to \Delta\] is equivalent to a family of maps
\begin{equation}
   \phi_\ell\colon S_n\to \Bbbk, \quad \ell\ge 0.
\end{equation}

\begin{proposition}
  \begin{align}
    \Hom (P_n,\Delta)&\cong \Hom_{\mathsf{Set}}(\Z_{\ge 0}, \Hom_{\mathsf{Set}}( S_n,\Bbbk))
    \\
    &\cong \Hom_{\mathsf{Set}}(\Z_{\ge 0}\times S_n,\Bbbk) = \Bbbk^{\Z_{\ge 0}\times S_n}
  \end{align}
\end{proposition}

\begin{theorem}\label{thm:dn}
\mathleft
Let $\phi\in\Hom(P_n,\Delta)$ have coefficient maps 
$\{\phi_\ell\}_{\ell\ge 0}$.  
The coefficient maps of $d_n\phi$ are given as follows:

\begin{itemize}
	\item $n=0$,
\begin{gather}
( \mathrm{d}_0\phi ) _{\ell}( p_{i+1} ) =\phi _{\ell}( e_{i+1} ) -\phi _{\ell}( e_i ) 
\\
( \mathrm{d}_0\phi ) _{\ell}( q_i ) =\phi _{\ell}( e_i ) -\phi _{\ell}( e_{i+1} ) 
\\
( \mathrm{d}_0\phi ) _{\ell}( s_i ) =0
\end{gather}
\item $n=1$,
\begin{gather}
( \mathrm{d}_1\phi ) _{\ell}( p_{i+1}q_i ) =\phi _{\ell-1}( p_{i+1} ) +\phi _{\ell-1}( q_i ) -\phi _{\ell}( s_{i+1} ) 
\\
( \mathrm{d}_1\phi ) _{\ell}( q_ip_{i+1} ) =\phi _{\ell-1}( q_i ) +\phi _{\ell-1}( p_{i+1} ) -\phi _{\ell}( s_i ) 
\\
( \mathrm{d}_1\phi ) _{\ell}( s_{i+1}p_{i+1} ) =\phi _{\ell}( s_{i+1} ) -\phi _{\ell}( s_i ) 
\\
( \mathrm{d}_1\phi ) _{\ell}( s_iq_i ) =\phi _{\ell}( s_i ) -\phi _{\ell}( s_{i+1} ) 
\end{gather}
\item $n\ge 2$ even,
\begin{gather}
( \mathrm{d}_n\phi ) _{\ell}( P^{n+1}_{i+1} ) 
=\phi _{\ell}( P_{i+1}^n ) -\phi _{\ell}( Q_i^n ) +\phi _{\ell}( sP_{i+1}^n ) 
\\
( \mathrm{d}_n\phi ) _{\ell}( Q_i^{n+1} ) 
=\phi _{\ell}( Q_i^n ) -\phi _{\ell}( P_{i+1}^n ) +\phi _{\ell}( sQ_i^n )  
\\
( \mathrm{d}_n\phi ) _{\ell}( sP_{i+1}^{n+1} ) 
=\phi _{\ell-1}( sP_{i+1}^n ) +\phi _{\ell-1}( sQ_i^n )
\\
( \mathrm{d}_n\phi ) _{\ell}( sQ_i^{n+1} ) 
=\phi _{\ell-1}( sQ_i^n ) +\phi _{\ell-1}( sP_{i+1}^n ) 
\end{gather}
\item $n\ge 2$ odd,
\begin{gather}
( \mathrm{d}_n\phi ) _{\ell}( P^{n+1}_{i+1} ) 
=\phi _{\ell-1}( P_{i+1}^n ) +\phi _{\ell-1}( Q_i^n ) -\phi _{\ell}( sP_{i+1}^n ) 
\\
( \mathrm{d}_n\phi ) _{\ell}( Q_i^{n+1} ) 
=\phi _{\ell-1}( Q_i^n ) +\phi _{\ell-1}( P_{i+1}^n ) -\phi _{\ell}( sQ_i^n ) 
\\
( \mathrm{d}_n\phi ) _{\ell}( sP_{i+1}^{n+1} )
 =\phi _{\ell}( sP_{i+1}^n ) -\phi _{\ell}( sQ_i^n )
\\
( \mathrm{d}_n\phi ) _{\ell}( sQ_i^{n+1} ) 
=\phi _{\ell}( sQ_i^n ) -\phi _{\ell}( sP_{i+1}^n )  
\end{gather}
where we take all $\phi_{-1}\coloneqq 0$.
\end{itemize}
\end{theorem}
\begin{proof}
  Direct computation.
\end{proof}

\subsubsection{The cohomology}
\label{sect:HHid}
\begin{theorem}
  \begin{align}
    H^n\Hom (\mathrm{id},\mathrm{id})=\mathrm{HH}^n(\Delta)=\begin{cases}
      \Bbbk^{\mathbb{Z}_{\ge 0}}, & 0\le x\le 1,\\
      \Bbbk^{|\mathbf{x}|-1}, & n=2,\\
      0, & n\ge 3.
    \end{cases}
  \end{align}
\end{theorem}
\begin{proof}
For $n\ge 3$, the resolution of the quiver algebra 
$A=\Bbbk Q/I$ decomposes into blocks of the shape
\begin{equation}
  Q^{(i)}=\begin{tikzcd}[column sep=large]
    \underset{e_i}{\bullet}\ar[r,bend left=10,"p_{i+1}"]
    \ar[looseness=8,out=120,in=60,"s_i"near end]
    & \underset{e_{i+1}}{\bullet}\ar[l,bend left=10,"q_i"]
    \ar[looseness=8,out=120,in=60,"s_{i+1}"near end]
  \end{tikzcd}
\end{equation}
with only four ambiguities in each block $Q^{(i)}$. Written explicitly, for $n\ge 2$ we have $|Q^{(i)}_n\cap S_n|=4$ and
\begin{align}
  &Q^{(i)}_n\cap S_n=\{P^n_{i+1},\,Q^n_i,\,sP^n_{i+1},\,sQ^n_i\},\\
  &(Q^{(i)}_n\cap S_n)\cap(Q^{(j)}_n\cap S_n)=\varnothing,\quad i\ne j.\label{eq:blocks}
\end{align}
Hence the matrix for $\d_n$ is block diagonal, we may treat all $Q^{(i)}$ separately, 
representing the Hochschild differential 
as a family of $4\times4$ matrices 
parametrized by $0\le i<|\mathbf{x}|$.

\smallskip
\noindent
\textbf{Case $n\ge 3$ even.}

\begin{itemize}
\item If $\phi\in \ker \d_{n}$, i.e. $\mathrm{d}_n\phi=0$, then by Theorem~\ref{thm:dn} for all $\ell\ge 0$ and $0\le i<|\mathbf{x}|$ we have
\begin{equation}
  0=\left[ \begin{array}{l}
(\mathrm{d}_n\phi )_{\ell}(P_{i+1}^{n+1})\\
(\mathrm{d}_n\phi )_{\ell}(Q_{i}^{n+1})\\
(\mathrm{d}_n\phi )_{\ell +1}(sP_{i+1}^{n+1})\\
(\mathrm{d}_n\phi )_{\ell +1}(sQ_{i}^{n+1})\\
\end{array} \right] =\left[ \begin{matrix}
1&		-1&		1&		0\\
-1&		1&		0&		1\\
0&		0&		1&		1\\
0&		0&		1&		1\\
\end{matrix} \right] \left[ \begin{array}{l}
\phi _{\ell}(P_{i+1}^{n})\\
\phi _{\ell}(Q_{i}^{n})\\
\phi _{\ell}(sP_{i+1}^{n})\\
\phi _{\ell}(sQ_{i}^{n})\\
\end{array} \right] .\label{eq:kereven}
\end{equation}
This matrix has rank~$2$, eigenvalues $\{2,2,0,0\}$, and its null space is spanned by the zero-eigenvalue eigenvectors $V^{n,\ell}_{i,P}\coloneqq [1,0,-1,1]^{\mathsf{T}}$ and $V^{n,\ell}_{i,Q}\coloneqq [0,1,1,-1]^{\mathsf{T}}$.
\item If $\phi\in \mathrm{im}\,\d_{n-1}$, i.e. $\phi=\d_{n-1}\psi$, then for all $\ell\ge 0$ and $0\le i<|\mathbf{x}|$ we have
\begin{equation}
  \left[ \begin{array}{l}
	\phi _{\ell}(P_{i+1}^{n})\\
	\phi _{\ell}(Q_{i}^{n})\\
	\phi _{\ell}(sP_{i+1}^{n})\\
	\phi _{\ell}(sQ_{i}^{n})\\
\end{array} \right] =\left[ \begin{array}{l}
	(\mathrm{d}_{n-1}\psi )_{\ell}(P_{i+1}^{n-1})\\
	(\mathrm{d}_{n-1}\psi )_{\ell}(Q_{i}^{n-1})\\
	(\mathrm{d}_{n-1}\psi )_{\ell}(sP_{i+1}^{n-1})\\
	(\mathrm{d}_{n-1}\psi )_{\ell}(sQ_{i}^{n-1})\\
\end{array} \right] =\left[ \begin{matrix}
	1&		1&		-1&		0\\
	1&		1&		0&		-1\\
	0&		0&		1&		-1\\
	0&		0&		-1&		1\\
\end{matrix} \right] \left[ \begin{array}{l}
	\psi _{\ell -1}(P_{i+1}^{n-1})\\
	\psi _{\ell -1}(Q_{i}^{n-1})\\
	\psi _{\ell}(sP_{i+1}^{n-1})\\
	\psi _{\ell}(sQ_{i}^{n-1})\\
\end{array} \right]. \label{eq:imeven}
\end{equation} 
\begin{itemize}
  \item For $\ell \ge 1$, all the components of the $\psi$ vector on the RHS of \eqref{eq:imeven} can be chosen freely. The matrix in \eqref{eq:imeven} has rank $2$, eigenvalues $\{2,2,0,0\}$, with the image spanned by the nonzero-eigenvalue eigenvectors, which are also $V^{n,\ell}_{i,P}\coloneqq [1,0,-1,1]^{\mathsf{T}}$ and $V^{n,\ell}_{i,Q}\coloneqq [0,1,1,-1]^{\mathsf{T}}$;
  \item For $\ell =0$, the components $\psi _{ -1}(P_{i+1}^{n-1})$ and	$\psi _{ -1}(Q_{i}^{n-1})$ have to be fix to zero, since there are no diagrams with $-1$ dot. Thus \eqref{eq:imeven} is reduced to
  \begin{equation}
  \left[ \begin{array}{l}
	\phi _{0}(P_{i+1}^{n})\\
	\phi _{0}(Q_{i}^{n})\\
	\phi _{0}(sP_{i+1}^{n})\\
	\phi _{0}(sQ_{i}^{n})\\
\end{array} \right] =\left[ \begin{array}{l}
	(\mathrm{d}_{n-1}\psi )_{0}(P_{i+1}^{n-1})\\
	(\mathrm{d}_{n-1}\psi )_{0}(Q_{i}^{n-1})\\
	(\mathrm{d}_{n-1}\psi )_{0}(sP_{i+1}^{n-1})\\
	(\mathrm{d}_{n-1}\psi )_{0}(sQ_{i}^{n-1})\\
\end{array} \right] =\left[ \begin{matrix}
	-1&		0\\
	0&		-1\\
	1&		-1\\
	-1&		1\\
\end{matrix} \right] \left[ \begin{array}{l}
	\psi _{0}(sP_{i+1}^{n-1})\\
	\psi _{0}(sQ_{i}^{n-1})\\
\end{array} \right].
\end{equation} 
Still, this matrix has rank $2$, with the image spanned by the colomn vectors $[-1,0,1,-1]^{\mathsf{T}}$, $[0,-1,-1,1]^{\mathsf{T}}$, which is the same subspace as $\ell \ge 1$.
\end{itemize}
Thus for all $\ell$, we have $\ker \d_n=\mathrm{im}\, \d_{n-1}$.
\begin{equation}
  \mathrm{HH}^n(\Delta)=0,\quad n\ge 3\;\text{even}.
\end{equation}
\end{itemize}

\smallskip
\noindent
\textbf{Case $n\ge 3$ odd.}

\begin{itemize}
\item If $\phi\in \ker \d_{n}$, i.e. $\mathrm{d}_n\phi=0$, then for all $\ell\ge 0$ and $0\le i<|\mathbf{x}|$ we have
\begin{equation}
  0=\left[ \begin{array}{l}
(\mathrm{d}_n\phi )_{\ell}(P_{i+1}^{n+1})\\
(\mathrm{d}_n\phi )_{\ell}(Q_{i}^{n+1})\\
(\mathrm{d}_n\phi )_{\ell}(sP_{i+1}^{n+1})\\
(\mathrm{d}_n\phi )_{\ell}(sQ_{i}^{n+1})\\
\end{array} \right] =\left[ \begin{matrix}
1&		1&		-1&		0\\
1&		1&		0&		-1\\
0&		0&		1&		-1\\
0&		0&		-1&		1\\
\end{matrix} \right] \left[ \begin{array}{l}
\phi _{\ell-1}(P_{i+1}^{n})\\
\phi _{\ell-1}(Q_{i}^{n})\\
\phi _{\ell}(sP_{i+1}^{n})\\
\phi _{\ell}(sQ_{i}^{n})\\
\end{array} \right] .\label{eq:kerodd}
\end{equation}
\item If $\phi\in \mathrm{im}\,\d_{n-1}$, i.e. $\phi=\d_{n-1}\psi$, then for all $\ell\ge 0$ and $0\le i<|\mathbf{x}|$ we have
\begin{equation}
  \left[ \begin{array}{l}
	\phi _{\ell-1}(P_{i+1}^{n})\\
	\phi _{\ell-1}(Q_{i}^{n})\\
	\phi _{\ell}(sP_{i+1}^{n})\\
	\phi _{\ell}(sQ_{i}^{n})\\
\end{array} \right] =\left[ \begin{array}{l}
	(\mathrm{d}_{n-1}\psi )_{\ell-1}(P_{i+1}^{n-1})\\
	(\mathrm{d}_{n-1}\psi )_{\ell-1}(Q_{i}^{n-1})\\
	(\mathrm{d}_{n-1}\psi )_{\ell}(sP_{i+1}^{n-1})\\
	(\mathrm{d}_{n-1}\psi )_{\ell}(sQ_{i}^{n-1})\\
\end{array} \right] =\left[ \begin{matrix}
	1&		-1&		1&		0\\
	-1&		1&		0&		1\\
	0&		0&		1&		1\\
	0&		0&		1&		1\\
\end{matrix} \right] \left[ \begin{array}{l}
	\psi _{\ell -1}(P_{i+1}^{n-1})\\
	\psi _{\ell -1}(Q_{i}^{n-1})\\
	\psi _{\ell -1}(sP_{i+1}^{n-1})\\
	\psi _{\ell -1}(sQ_{i}^{n-1})\\
\end{array} \right]. \label{eq:imodd}
\end{equation} 
\begin{itemize}
  \item For $\ell \ge 1$, all the components of the vectors on the RHS of both \eqref{eq:kerodd} and \eqref{eq:imodd} can be chosen freely. The two matrices are the same as the even case with the order reversed. Thus, the kernel of the matrix in \eqref{eq:kerodd} and the image of the matrix in \eqref{eq:imodd} are both spanned by the other eigenvectors in the $n\ge 3$ even case, which are both $\mathrm{Span}\,\{V^{n,\ell-1}_{i,P}\coloneqq (1,0,1,1)^\mathsf{T},\,V^{n,\ell-1}_{i,Q}\coloneqq (0,1,1,1)^\mathsf{T}\}$.
  \item For $\ell =0$, \eqref{eq:kerodd} and \eqref{eq:imodd} are reduced to
  \begin{align}\label{eq:kerodd0}
  0=\left[ \begin{array}{l}
(\mathrm{d}_n\phi )_{0}(P_{i+1}^{n+1})\\
(\mathrm{d}_n\phi )_{0}(Q_{i}^{n+1})\\
(\mathrm{d}_n\phi )_{0}(sP_{i+1}^{n+1})\\
(\mathrm{d}_n\phi )_{0}(sQ_{i}^{n+1})\\
\end{array} \right] &=\left[ \begin{matrix}
-1&		0\\
0&		-1\\
1&		-1\\
-1&		1\\
\end{matrix} \right] \left[ \begin{array}{l}
\phi _{0}(sP_{i+1}^{n})\\
\phi _{0}(sQ_{i}^{n})\\
\end{array} \right] \\
\left[ \begin{array}{l}
	\phi _{0}(sP_{i+1}^{n})\\
	\phi _{0}(sQ_{i}^{n})\\
\end{array} \right] =\left[ \begin{array}{l}
	(\mathrm{d}_{n-1}\psi )_{0}(sP_{i+1}^{n-1})\\
	(\mathrm{d}_{n-1}\psi )_{0}(sQ_{i}^{n-1})\\
\end{array} \right] &=0
\end{align} 

The only solution to \eqref{eq:kerodd0} is  $\left[ 
\phi _{0}(sP_{i+1}^{n}),\;
\phi _{0}(sQ_{i}^{n}) \right]^\mathsf{T}=\left[ 0,\;0 \right]^\mathsf{T}$, thus the kernel and the image are both $0$.
\end{itemize}
Thus for all $\ell$, we have $\ker \d_n=\mathrm{im}\, \d_{n-1}$.
\begin{equation}
  \mathrm{HH}^n(\Delta)=0,\quad n\ge 3\;\text{odd}.
\end{equation}
\end{itemize}

\smallskip
\noindent
\textbf{Case $n=2$.}
This is where things become different, since
\begin{equation}
  (Q^{(i)}_1\cap S_1)\cap (Q^{(i-1)}_1\cap S_1)=\{s_{i}\}\ne \varnothing,
\end{equation}
\eqref{eq:blocks} is not true for $n=1$. Thus, $\mathrm{im}\,\d_1$ begins to relate different blocks $Q^{(i)}$ to each other

\begin{itemize}
\item If $\phi\in \ker \d_{2}$, i.e. $\mathrm{d}_2\phi=0$, the situation is the same as \eqref{eq:kereven}. Here, we change the order of components as

\begin{equation}\label{eq:kerd2}
0=\left[ \begin{array}{l}
	(\mathrm{d}_2\phi )_{\ell +1}(sQ_{i}^{3})\\
	(\mathrm{d}_2\phi )_{\ell}(Q_{i}^{3})\\
	(\mathrm{d}_2\phi )_{\ell}(P_{i+1}^{3})\\
	(\mathrm{d}_2\phi )_{\ell +1}(sP_{i+1}^{3})\\
\end{array} \right] =\left[ \begin{matrix}
	1&		0&		0&		1\\
	1&		1&		-1&		0\\
	0&		-1&		1&		1\\
	1&		0&		0&		1\\
\end{matrix} \right] \left[ \begin{array}{l}
	\phi _{\ell}(s_iq_i)\\
	\phi _{\ell}(q_ip_{i+1})\\
	\phi _{\ell}(p_{i+1}q_i)\\
	\phi _{\ell}(s_{i+1}p_{i+1})\\
\end{array} \right] .
\end{equation}

The kernel is the same, spanned by $V^{2,\ell}_{i,Q}\coloneqq [-1,1,0,1]^{\mathsf{T}}$ and $V^{2,\ell}_{i,P}\coloneqq [1,0,1,-1]^{\mathsf{T}}$, where $0\le i<|\mathbf{x}|$.
\item If $\phi\in \mathrm{im}\,\d_{n-1}$, i.e. $\phi=\d_{n-1}\psi$, then for all $\ell\ge 0$ and $0\le i<|\mathbf{x}|$ we have on block $Q^{(i)}$:

\begin{equation}
  \left[ \begin{array}{l}
	\phi _{\ell}(s_iq_i)\\
	\phi _{\ell}(q_ip_{i+1})\\
	\phi _{\ell}(p_{i+1}q_i)\\
	\phi _{\ell}(s_{i+1}p_{i+1})\\
\end{array} \right] =\left[ \begin{array}{l}
	(\mathrm{d}_1\psi )_{\ell}(s_iq_i)\\
	(\mathrm{d}_1\psi )_{\ell}(q_ip_{i+1})\\
	(\mathrm{d}_1\psi )_{\ell}(p_{i+1}q_i)\\
	(\mathrm{d}_1\psi )_{\ell}(s_{i+1}p_{i+1})\\
\end{array} \right] =\left[ \begin{matrix}
	1&		0&		0&		-1\\
	-1&		1&		1&		0\\
	0&		1&		1&		-1\\
	-1&		0&		0&		1\\
\end{matrix} \right] \left[ \begin{array}{l}
	\psi _{\ell}(s_i)\\
	\psi _{\ell -1}(q_i)\\
	\psi _{\ell -1}(p_{i+1})\\
	\psi _{\ell}(s_{i+1})\\
\end{array} \right] 
\end{equation}

However, now $s_i$ belongs to both blocks $Q^{(i)}$ and $Q^{(i-1)}$, the total matrix is no longer block diagonal:

\begin{align}
&\left[ \begin{array}{l}
	\vdots\\
	\phi _{\ell}(s_{i-1}q_{i-1})\\
	\phi _{\ell}(q_{i-1}p_i)\\
	\phi _{\ell}(p_iq_{i-1})\\
	\phi _{\ell}(s_ip_i)\\
	\phi _{\ell}(s_iq_i)\\
	\phi _{\ell}(q_ip_{i+1})\\
	\phi _{\ell}(p_{i+1}q_i)\\
	\phi _{\ell}(s_{i+1}p_{i+1})\\
	\vdots\\
\end{array} \right] =\left[ \begin{matrix}
	\ddots&		&		&		&		&		&		&		&		\\
	&		1&		0&		0&		-1&		&		&		&		\\
	&		-1&		1&		1&		0&		&		&		&		\\
	&		0&		1&		1&		-1&		&		&		&		\\
	&		-1&		0&		0&		1&		&		&		&		\\
	&		&		&		&		1&		0&		0&		-1&		\\
	&		&		&		&		-1&		1&		1&		0&		\\
	&		&		&		&		0&		1&		1&		-1&		\\
	&		&		&		&		-1&		0&		0&		1&		\\
	&		&		&		&		&		&		&		&		\ddots\\
\end{matrix} \right] \left[ \begin{array}{l}
	\vdots\\
	\psi _{\ell}(s_{i-1})\\
	\psi _{\ell -1}(q_{i-1})\\
	\psi _{\ell -1}(p_i)\\
	\psi _{\ell}(s_i)\\
	\psi _{\ell -1}(q_i)\\
	\psi _{\ell -1}(p_{i+1})\\
	\psi _{\ell}(s_{i+1})\\
	\vdots\\
\end{array} \right]
\label{eq:imd1}
\\
&=:\left[ 
	\cdots	J_{\ell}(s_{i-1})
	,\,J _{\ell -1}(q_{i-1})
	,\,J _{\ell -1}(p_i)
	,\,J _{\ell}(s_i)
	,\,J _{\ell -1}(q_i)
	,\,J _{\ell -1}(p_{i+1})
	,\,J _{\ell}(s_{i+1}) 	\cdots \right] \left[ \begin{array}{l}
	\vdots\\
	\psi _{\ell}(s_{i-1})\\
	\psi _{\ell -1}(q_{i-1})\\
	\psi _{\ell -1}(p_i)\\
	\psi _{\ell}(s_i)\\
	\psi _{\ell -1}(q_i)\\
	\psi _{\ell -1}(p_{i+1})\\
	\psi _{\ell}(s_{i+1})\\
	\vdots\\
\end{array} \right]
\end{align}
where for $w\in Q_1$ we denote by $J_\ell (w)$ the colomn vector in the large matrix corresponding to $\psi_\ell (w)$.
\begin{itemize}
  \item For $\ell \ge 1$, all the components of the $\psi$ vector on the RHS of \eqref{eq:imd1} can be chosen freely. The image of $\d_1$ is given by the linear span of all column vector of the matrix. We can see that
  \begin{align}
    V^{2,\ell}_{0,Q}&=-J_{\ell}(s_0)\\
    V^{2,\ell}_{0,P}&=-V^{2,\ell}_{0,Q}+J_{\ell-1}(q_0)=J_{\ell}(s_0)+J_{\ell-1}(q_0)\\
    V^{2,\ell}_{1,Q}&=-V^{2,\ell}_{0,P}-J_{\ell}(s_1)=-J_{\ell}(s_0)-J_{\ell-1}(q_0)-J_{\ell}(s_1)\\
    V^{2,\ell}_{1,P}&=-V^{2,\ell}_{1,Q}+J_{\ell-1}(q_1)=J_{\ell}(s_0)+J_{\ell-1}(q_0)+J_{\ell}(s_1)+J_{\ell-1}(q_1)\\
    \cdots&\cdots\notag\\
    V^{2,\ell}_{i,Q}&=-V^{2,\ell}_{i-1,P}-J_{\ell}(s_i)\label{eq:148}\\
    V^{2,\ell}_{i,P}&=-V^{2,\ell}_{i,Q}+J_{\ell-1}(q_i)\label{eq:144}\\
    \cdots&\cdots\notag
  \end{align}
  Thus, all the null eigenvectors of $\d_2$ can be represented by a linear combination of column vectors of $\d_1$. Thus the $\ell\ge 1$ part is exact.
  \item For $\ell =0$, the components $\psi _{ -1}(q_j)$ and $\psi _{ -1}(p_{j+1})$ must be set to zero. This means we can no longer use $J_{-1}(q_j)$ and $J_{-1}(p_{j+1})$ in the linear combination:
  \begin{align}
\left[ \begin{array}{l}
	\vdots\\
	\phi _{0}(s_{i-1}q_{i-1})\\
	\phi _{0}(q_{i-1}p_i)\\
	\phi _{0}(p_iq_{i-1})\\
	\phi _{0}(s_ip_i)\\
	\phi _{0}(s_iq_i)\\
	\phi _{0}(q_ip_{i+1})\\
	\phi _{0}(p_{i+1}q_i)\\
	\phi _{0}(s_{i+1}p_{i+1})\\
	\vdots\\
\end{array} \right] &=\left[ \begin{matrix}
	\ddots&		&		&		&		\\
	&		1&				-1&			&		\\
	&		-1&			0&		&		\\
	&		0&			-1&		&		\\
	&		-1&			1&		&		\\
	&			&		1&		-1&		\\
	&			&		-1&	0&		\\
	&			&		0&	-1&		\\
	&			&		-1&			1&		\\
	&		&		&					&		\ddots\\
\end{matrix} \right] \left[ \begin{array}{l}
	\vdots\\
	\psi _{0}(s_{i-1})\\
	\psi _{0}(s_i)\\
	\psi _{0}(s_{i+1})\\
	\vdots\\
\end{array} \right]\label{eq:145}
\\
&=:\left[ 
	\cdots	J_{0}(s_{i-1})
	,\,J _{0}(s_i)
	,\,J _{0}(s_{i+1}) 	\cdots \right] \left[ \begin{array}{l}
	\vdots\\
	\psi _{0}(s_{i-1})\\
	\psi _{0}(s_i)\\
	\psi _{0}(s_{i+1})\\
	\vdots\\
\end{array} \right]
\end{align}

In this case, the term $J_{\ell-1}(q_i)$ in \eqref{eq:144} is no longer available. Therefore, the $\ell=0$ part is no longer exact.

Each $J_0(s_i)$ is linearly independent, since the entry corresponding to 
$\phi_0(q_{i}p_{i+1})$ is nonzero only for $J_0(s_i)$. 
Consequently, the matrix in \eqref{eq:145} has rank $|\mathbf{x}|+1$, which is the dimension of the $\ell=0$ part of $\operatorname{im}\d_1$.

As shown earlier, the $\ell=0$ component of $\ker d_2$ 
is spanned by $2|\mathbf{x}|$ linearly independent vectors 
$\{V^{2,0}_{i,P}, V^{2,0}_{i,Q}\}_{0\le i<|\mathbf{x}|}$.
Hence,
\begin{equation}
\dim^{\ell=0}  \ker d_2= 2|\mathbf{x}|,
\qquad
\dim^{\ell=0} \operatorname{im} d_1 = |\mathbf{x}|+1,
\end{equation}

and
\begin{equation}
\dim \mathrm{HH}^2(\Delta)
= \dim^{\ell=0}\ker d_2 - \dim^{\ell=0}\operatorname{im} d_1
= |\mathbf{x}| - 1.
\end{equation}
\end{itemize}
\end{itemize}
Therefore,
\begin{equation}
  \mathrm{HH}^2(\Delta)=\Bbbk^{|\mathbf{x}|-1}.
\end{equation}

The vectors $J_\ell$ span the image and the vectors $V^{2,\ell}$ span the kernel.
Thus, we can take the representatives of the Hochschild cohomology as
\begin{equation}
  \mathfrak{V}_i\coloneqq V^{2,\ell}_{i-1,P}-V^{2,\ell}_{i,Q}=[0,0,\dots,1,0,1,-1,1,-1,0,-1,\dots,0,0]^\mathsf{T},\quad 0< i < |\mathbf{x}|.
\end{equation}

\smallskip
\noindent
\textbf{Case $n=1$.}

\begin{itemize}
  \item The condition for $\phi\in \ker \d_{1}$ follows from \eqref{eq:imd1}
  \begin{equation}
  0=\left[ \begin{array}{l}
	\vdots\\
	(\d_1\phi) _{\ell}(s_{i-1}q_{i-1})\\
	(\d_1\phi) _{\ell}(q_{i-1}p_i)\\
	(\d_1\phi) _{\ell}(p_iq_{i-1})\\
	(\d_1\phi) _{\ell}(s_ip_i)\\
	(\d_1\phi) _{\ell}(s_iq_i)\\
	(\d_1\phi) _{\ell}(q_ip_{i+1})\\
	(\d_1\phi) _{\ell}(p_{i+1}q_i)\\
	(\d_1\phi) _{\ell}(s_{i+1}p_{i+1})\\
	\vdots\\
\end{array} \right] =\left[ \begin{matrix}
	\ddots&		&		&		&		&		&		&		&		\\
	&	\color{JO}	1&	\color{JO}	0&	\color{JO}	0&		\color{JO}-1&		&		&		&		\\
	&		-1&		1&		1&		0&		&		&		&		\\
	&		0&		1&		1&		-1&		&		&		&		\\
	&	\color{JO}	-1&	\color{JO}	0&	\color{JO}	0&	\color{JO}	1&		&		&		&		\\
	&		&		&		&		\color{JO}1&	\color{JO}	0&	\color{JO}	0&	\color{JO}	-1&		\\
	&		&		&		&		-1&		1&		1&		0&		\\
	&		&		&		&		0&		1&		1&		-1&		\\
	&		&		&		&	\color{JO}	-1&	\color{JO}	0&	\color{JO}	0&	\color{JO}	1&		\\
	&		&		&		&		&		&		&		&		\ddots\\
\end{matrix} \right] \left[ \begin{array}{l}
	\vdots\\
	\phi _{\ell}(s_{i-1})\\
	\phi _{\ell -1}(q_{i-1})\\
	\phi _{\ell -1}(p_i)\\
	\phi _{\ell}(s_i)\\
	\phi _{\ell -1}(q_i)\\
	\phi _{\ell -1}(p_{i+1})\\
	\phi _{\ell}(s_{i+1})\\
	\vdots\\
\end{array} \right]\label{eq:kerd1}
\end{equation}
This can be summarized concisely as
\begin{equation}
\begin{cases}
\phi_0(s_{i})=0, &\\
\phi _{\ell}(s_{i})=:\sigma_{\ell-1}, & \ell> 0,\; \forall \, 0\le i\le |\mathbf{x}|,\\
\phi _{\ell}(q_{i}) + \phi _{\ell}(p_{i+1})=\sigma _{\ell},& \ell\ge 0,\;\forall\, 0\le i <|\mathbf{x}|.\\
\end{cases}
\end{equation}
where the middle line comes from the blue rows of \eqref{eq:kerd1}, while the first and third lines come from the black rows.

\item The condition for $\phi=\d_0\psi\in \operatorname{im} \d_{0}$ is, by Theorem~\ref{thm:dn}
\begin{align}
\phi_{\ell}( p_{i+1} ) &=\psi _{\ell}( e_{i+1} ) -\psi _{\ell}( e_i ) 
\\
\phi_{\ell}( q_i ) &=\psi _{\ell}( e_i ) -\psi _{\ell}( e_{i+1} ) 
\\
\phi_{\ell}( s_i ) &=0
\end{align}
which is equivalent to
\begin{equation}
\begin{cases}
\phi _{\ell}(s_{i})=0, & \forall \, \ell\ge 0,\; 0\le i\le |\mathbf{x}|,\\
\phi _{\ell}(q_{i}) + \phi _{\ell}(p_{i+1})=0,& \forall\, \ell\ge 0,\;0\le i <|\mathbf{x}|.\\
\end{cases}
\end{equation}
\end{itemize}

Therefore, $\{\sigma_\ell\}_{\ell\ge 0}$ are precisely the parameters of the Hochschild cohomology:
\begin{equation}
  \mathrm{HH}^1(\Delta)=\Bbbk^{\Z_{\ge 0}}.
\end{equation}

\smallskip
\noindent
\textbf{Case $n=0$.}

\begin{itemize}
  \item The condition for $\phi\in \ker \d_{0}$ is simply
  \begin{equation}
    \epsilon_\ell=\phi _{\ell}( e_i ),\quad \ell\ge 0,\;\forall\, 0\le i\le |\mathbf{x}|.
  \end{equation}
\end{itemize}
Therefore, $\{\epsilon_\ell\}_{\ell\ge 0}$ parametrize the Hochschild cohomology:
\begin{equation}
  \mathrm{HH}^0(\Delta)=\Bbbk^{\Z_{\ge 0}}.
\end{equation}
\end{proof}

\subsection{$H^\bullet\Hom(\mathrm{id},\bim)$}

\begin{theorem}
  \begin{align}
    H^n\Hom (\mathrm{id},\bim)=\mathrm{HH}^n(\mathfrak{B}_{i^-})=\begin{cases}
      \Bbbk^{\mathbb{Z}_{\ge 1}}, & n=0\\
      \Bbbk^{\mathbb{Z}_{\ge 0}}, & n=1\\
      \Bbbk^{|\mathbf{x}|-2}, & n=2\\
      0, & n\ge 3.
    \end{cases}
  \end{align}
\end{theorem}
\begin{proof}
From the short exact sequence \eqref{eq:coker}, we obtain the long exact sequence of Hochschild cohomology groups:
\begin{equation}\label{eq:LES}
	\begin{tikzcd}[column sep=small]
		\cdots\ar[r]
		&\mathrm{HH}^n(\mathfrak{B}_{i^-})\ar[r]
		&\mathrm{HH}^n(\Delta)\ar[r]
		&\mathrm{HH}^n(\mathfrak{S}_i)\ar[r]
		&\mathrm{HH}^{n+1}(\mathfrak{B}_{i^-})\ar[r]
		&\cdots
	\end{tikzcd}
\end{equation}

The Hochschild cohomology of $\mathfrak{S}_i$ can be computed directly. 
Indeed, since $\mathfrak{S}_i$ is generated by a single diagram $\textit{\k{e}}_i$, which goes from $T_i$ to $T_i$, to define a morphism 
$\phi \in \Hom(P_n, \mathfrak{S}_i)$, it suffices to assign a scalar in $\Bbbk$ to each element in $S_n$ 
whose source and target are both $e_i$. 
For $n=0,1$, there is only one such element, namely $e_i \in S_0$ and $s_i \in S_1$; 
for $n \ge 2$, there are two such elements, which are $P^n_i$ and $Q^n_i$ for even $n$, and $sP^n_i$ and $sQ^n_i$ for odd $n$. 
Therefore,

\begin{equation}
  \Hom (P_n,\mathfrak{S}_i)=\begin{cases}
    \Bbbk, & 0\le n\le 1,\\
    \Bbbk ^2, & n\ge 2.
  \end{cases}
\end{equation}

Using the resolution in Theorem~\ref{thm:dn}, the corresponding complex is easily seen to be
\begin{equation}
	\begin{tikzcd}
		\Bbbk\ar[r,"0"]
		&\Bbbk\ar[r,"\lbracket -1\comma -1\rbracket"]
		&\Bbbk^2\ar[r,"0"]
		&\Bbbk^2\ar[r,"-\mathrm{id}"]
		&\Bbbk^2\ar[r,"0"]
		&\Bbbk^2\ar[r,"-\mathrm{id}"]
		&\Bbbk^2\ar[r]
		&\cdots
	\end{tikzcd}
\end{equation}

Hence
\begin{equation}
	\mathrm{HH}^\bullet(\mathfrak{S}_i)=H^\bullet\mathrm{Hom}(\mathrm{id},\mathfrak{S}_i)=\left\{\Bbbk,0,\Bbbk,0,0,0,\dots \right\}
\end{equation}

Therefore \eqref{eq:LES} is for $n\ge 4$,
\begin{equation}
	\begin{tikzcd}
		0\ar[r]
		&\mathrm{HH}^n(\mathfrak{B}_{i^-})\ar[r,"\cong"]
		&\mathrm{HH}^n(\Delta)\ar[r]
		&0
	\end{tikzcd}
\end{equation}
\begin{equation}
	\mathrm{HH}^n(\mathfrak{B}_{i^-})\cong \mathrm{HH}^n(\Delta)
  = 0,\quad n\ge 4.
\end{equation}

For the lower part of \eqref{eq:LES}, we have

\begin{equation}
	\begin{tikzcd}
		0\ar[r]
		&\mathrm{HH}^0(\mathfrak{B}_{i^-})\ar[r]
		&\mathrm{HH}^0(\Delta)\ar[r]
		&\Bbbk\ar[r]
		&\mathrm{HH}^1(\mathfrak{B}_{i^-})\ar[r]
		&\mathrm{HH}^1(\Delta)\ar[r]&0
	\end{tikzcd}
\end{equation}

\begin{equation}
	\begin{tikzcd}
		0\ar[r]
		&\mathrm{HH}^2(\mathfrak{B}_{i^-})\ar[r]
		&\mathrm{HH}^2(\Delta)\ar[r]
		&\Bbbk\ar[r]
		&\mathrm{HH}^3(\mathfrak{B}_{i^-})\ar[r]
		&0
	\end{tikzcd}
\end{equation}

It is easy to show that the maps $\mathrm{HH}^0(\Delta)\to\mathrm{HH}^0(\mathfrak{S}_i)$ and $\mathrm{HH}^2(\Delta)\to\mathrm{HH}^2(\mathfrak{S}_i)$ are surjections. Therefore,

\begin{align}
  &\mathrm{HH}^0(\mathfrak{B}_{i^-})\cong \mathrm{HH}^0(\Delta)/\Bbbk\cong \Bbbk^{\Z_{\ge 1}}\\
  &\mathrm{HH}^1(\mathfrak{B}_{i^-})\cong \mathrm{HH}^1(\Delta)\cong \Bbbk^{\Z_{\ge 0}}\\
  &\mathrm{HH}^2(\mathfrak{B}_{i^-})\cong \mathrm{HH}^2(\Delta)/\Bbbk\cong \Bbbk^{|\mathbf{x}|-2}\\
  &\mathrm{HH}^3(\mathfrak{B}_{i^-})=0
\end{align}
\end{proof}
\begin{remark}
  One can also repeat the procedure in Section~\ref{sect:HHid} to find this. More concretely, $\mathrm{HH}^0(\mathfrak{B}_{i^-})$ demands $\epsilon_0=0$ and $\mathrm{HH}^2(\mathfrak{B}_{i^-})$ demands $\vartheta_i=0$.
\end{remark}

\section{Homotopy deformation retract from bar resolution}

To obtain explicit expressions for the natural transformations, we translate our resolution $P_\bullet$ back to the bar resolution.  
In our $A_\infty$-category $\mathcal{C}$, the only nontrivial structure map is $\mu^2$.  
Moreover, defining $\bar{A} \coloneqq  A / (\Bbbk Q_0 \cdot 1_A)$, the bar resolution from Definition~\ref{def:Bar} can, in the language of quiver algebras, be reduced to the \emph{normalized bar resolution}:

\begin{definition}[normalized bar resolution]
\begin{align}
  &\mathrm{Bar}_n=A\otimes \bar{A}^{\otimes n}\otimes A  \\
  &\bar{\partial}_n\colon \mathrm{Bar}_n\to\mathrm{Bar}_{n-1}\\
  &\begin{aligned}
    \bar{\partial}_n(a_{n+1}\otimes [a_n|\cdots| a_1]\otimes a_0)=&a_{n+1}\otimes[a_n|\cdots|a_2]\otimes a_1a_0\\
    &+\sum_{i=1}^{n-1}(-1)^i a_{n+1}\otimes[a_n|\cdots|a_{i+1}a_{i}|a_{i-1}|\cdots|a_1]\otimes a_0\\
    &+(-1)^n a_{n+1}a_n\otimes[a_{n-1}|\cdots|a_1]\otimes a_0
  \end{aligned}\label{eq:Barn}
\end{align}  
\end{definition}

The convention in the RHS of \eqref{eq:Barn} is consistent with \cite{CS, BW}, which differs by an overall $\pm$ sign with $\mu^1_{\Fun (\mathcal{A},\mathcal{B})}$ in \cite{Seidel-book}.

A homotopy deformation retract from the normalized bar resolution onto Chouhy--Solotar resolution is constructed in \cite{BW}.
\begin{theorem}
  There is a special homotopy deformation retract
  \begin{equation}
	\begin{tikzcd}
		\left(P_\bullet,\partial_\bullet\right)\ar[rr,yshift=3,"F_\bullet"]&&\left(\mathrm{Bar}_\bullet,\bar{\partial}_\bullet\right)\ar[ll,yshift=-3,"G_\bullet"]\ar[looseness=5, out=30, in=-30, "h_\bullet"]
	\end{tikzcd}
\end{equation}
such that
\begin{gather}
	G_nF_n=\mathrm{id},\qquad F_nG_n=\mathrm{id}+h_{n-1}\bar{\partial}_n+\bar{\partial}_{n+1}h_n
	\\
	h_nF_n=0,\qquad G_{n+1}h_n=0,\qquad h_{n+1}h_n=0
\end{gather}

\end{theorem}

The map that we need is $G_\bullet$, which is given by
\begin{proposition}
  \begin{align*}
  &G_{-1}=\mathrm{id}_A
  \\
  &G_n\left( a\otimes y\otimes b \right) =a\cdot\rho _{n-1}G_{n-1}\bar{\partial}_n\left( 1\otimes y\otimes b \right)
  \end{align*}
  where $y=[y_n|\cdots|y_1]$, and $\bar{\partial}_0\colon A\otimes A\to A$ is given by the multiplication in $A$.
\end{proposition}

Although we have proved that our cohomologically nontrivial natural transformations all concentrate in degree $\le 2$, for completeness and potential application in the future, we here present the result for arbitrary $n$.
\begin{theorem}\label{thm:Gn}
The result of $G_\bullet$ in our quiver is
  \begin{equation}
    G_n(a\otimes [y_n|\cdots|y_1]\otimes b)=a\cdot\sum_{w\preceq y_n\dots y_1}\mathrm{split}_n(w) \cdot b,
  \end{equation}
  where in the sum, different $w$ with the same result of $\mathrm{split}_n(w)$ are counted only once.
\end{theorem}

This description can be further simplified into a straightforward rule. So before proceeding to the proof, we first define a notion and consider an example.
\begin{definition}
  For a path $w \in Q_n$, we define the \emph{turning number} $\wp(w)$ of $w$ as the number of subpaths in $w$ that travel from a vertex $e_i$ 
  to its adjacent vertex $e_{i\pm 1}$ and then return to $e_i$.
\end{definition}
For example, we have $\wp(a_{ji}s_i^\alpha)=0$, $\wp(q_{i+1}p_{i+2}p_{i+1}s_i^\alpha q_i)=2$, $\wp(sQ^n_i)=n-2$, $\wp(Q^n_i)=n-1$.

\begin{proposition}
  If $v \preceq u$, then $\wp(v) \leq \wp(u)$.
\end{proposition}

\begin{proof}
  Under a type~I reduction $pq \mapsto s$, the value of $\wp$ may decrease by $1$, 
  whereas under a type~II reduction $sq \mapsto qs$, it always remains unchanged. 
  Therefore, under any sequence of reductions, $\wp$ can only decrease monotonically.
\end{proof}

\begin{proposition}
  For $y=y_n\dots y_1$, $y_i\in \bar{A}$, $\wp(y)\le n-1$.
  For $w\in S_n$, $\wp(w)=n-1$ (type I) or $n-2$ (type II).
\end{proposition}

\begin{example}\label{example}
  Take \[
	G_4\left( 1\otimes [a_{30}s_0|a_{04}|a_{41}s_1^2|a_{14}s_4]\otimes 1\right)\]
  as an example.
  Since all ambiguities in $S_n$ only bounces between two adjacent vertices, to obtain a path $w$ with $s\subset w\preceq y_n\dots y_1$ where $s\in S_n$, we first need to restrict the middle of the path $a_{30}s_0a_{04}a_{41}s_1^2a_{14}s_4$ between some two vertices, i.e. around some red strand. In this example, this leads to 3 cases:
\begin{equation}
  \begin{tikzpicture}[scale=0.5]
  \node[text=gray] at (-2,0){$A$};
  \node[text=gray] at (-2,0.7){$\otimes$};
  \node[text=gray] at (-2,1.5){$\bar{A}$};
  \node[text=gray] at (-2,2.5){$\otimes$};
  \node[text=gray] at (-2,3.5){$\bar{A}$};
  \node[text=gray] at (-2,4.5){$\otimes$};
  \node[text=gray] at (-2,5.5){$\bar{A}$};
  \node[text=gray] at (-2,6.5){$\otimes$};
  \node[text=gray] at (-2,7.5){$\bar{A}$};
  \node[text=gray] at (-2,8.3){$\otimes$};
  \node[text=gray] at (-2,9){$A$};
  \begin{scope}[xshift=-0.5cm]
    \foreach \x in {1,2,3,4} {
    \draw[M,line width=1pt] (\x,0) node[below] {$x_{\x}$} -- (\x,9);
  }
  \end{scope}
  \foreach \x in {0.5,2.5,4.5,6.5,8.5} {
    \draw[gray] (-0.3,\x)--(4.3,\x);
  }
  \draw[black,line width=1pt]
    (4,0) -- (4,0.5) 
    -- (4,1) to[out=90,in=-90] (1,2.5)
    -- (1,3) to[out=90,in=-90]
       (4,4.5) to[out=90,in=-90] (0,6.5)
    -- (0,7) to[out=90,in=-90] (3,8.5)
    -- (3,9);
  \fill[black] (4,0.75) circle (3pt);
  \fill[black] (1,2.65) circle (3pt);
  \fill[black] (1,2.95) circle (3pt);
  \fill[black] (0,6.75) circle (3pt);
  \draw[->,line width=1pt] (5,4.5) -- (7,4.5);
  \begin{scope}[xshift=8cm]
  \foreach \x in {0.5,1.5,2.5,3.5} {
    \draw[M,line width=1pt] (\x,0) -- (\x,9);
  }
  \node[below,text=M] at (1.5,0) {$x_2$};
  \draw[black,line width=1pt]
    (4,0) -- (4,0.5) 
    to[out=90,in=-90,looseness=0.7] (1,1.5) -- (1,3)
    to[out=90,in=-90] (2,4) -- (2,5.5)
    to[out=90,in=-90] (1,6.5) -- (1,8)
    to[out=90,in=-90,looseness=0.8] (3,9);
  \fill[black] (4,0.25) circle (3pt);
  \fill[black] (1,2.7) circle (3pt);
  \fill[black] (1,2.4) circle (3pt);
  \fill[black] (2,5.2) circle (3pt);
  \fill[black] (2,4.9) circle (3pt);
  \fill[black] (1,7.7) circle (3pt);
  \fill[black] (1,7.4) circle (3pt);
  \end{scope}
  \begin{scope}[xshift=14cm]
  \foreach \x in {0.5,1.5,2.5,3.5} {
    \draw[M,line width=1pt] (\x,0) -- (\x,9);
  }
  \node[below,text=M] at (2.5,0) {$x_3$};
  \draw[black,line width=1pt]
    (4,0) -- (4,0.5) 
    to[out=90,in=-90,looseness=0.8] (2,1.5) -- (2,3)
    to[out=90,in=-90] (3,4) -- (3,5.5)
    to[out=90,in=-90] (2,6.5) -- (2,8)
    to[out=90,in=-90] (3,9);
  \fill[black] (4,0.25) circle (3pt);
  \fill[black] (2,2.7) circle (3pt);
  \fill[black] (2,2.4) circle (3pt);
  \fill[black] (2,2.1) circle (3pt);
  \fill[black] (3,5.2) circle (3pt);
  \fill[black] (2,7.7) circle (3pt);
  \fill[black] (2,7.4) circle (3pt);
  \fill[black] (2,7.1) circle (3pt);
  \end{scope}
  \begin{scope}[xshift=20cm]
  \foreach \x in {0.5,1.5,2.5,3.5} {
    \draw[M,line width=1pt] (\x,0) -- (\x,9);
  }
  \node[below,text=M] at (3.5,0) {$x_4$};
  \draw[black,line width=1pt]
    (4,0) -- (4,0.5) 
    to[out=90,in=-90] (3,1.5) -- (3,3)
    to[out=90,in=-90] (4,4) -- (4,5.5)
    to[out=90,in=-90] (3,6.5) -- (3,8)
    -- (3,9);
  \fill[black] (4,0.25) circle (3pt);
  \fill[black] (3,2.7) circle (3pt);
  \fill[black] (3,2.4) circle (3pt);
  \fill[black] (3,2.1) circle (3pt);
  \fill[black] (3,1.8) circle (3pt);
  \fill[black] (3,7.7) circle (3pt);
  \fill[black] (3,7.4) circle (3pt);
  \fill[black] (3,7.1) circle (3pt);
  \fill[black] (3,6.8) circle (3pt);
  \end{scope}
  \end{tikzpicture}\label{eq:181}
\end{equation}

The next step is to perform type~II reductions to split off a subpath in $S_n$. 
Since reductions only move dots downward, and ambiguities in $S_n$ 
cannot contain a dot in the middle, all dots except those at the top 
(i.e. at the position corresponding to the source of $y_n$) must be moved to the bottom. 
Each time a top dot is dropped, a new term appears in the resulting expression. 
Thus, the first diagram on the RHS of~\eqref{eq:181} contributes three terms:

\begin{equation}\label{eq:182}
  \begin{tikzpicture}[scale=0.5]
  \foreach \x in {0.5,1.5,2.5,3.5} {
    \draw[M,line width=1pt] (\x,0) -- (\x,9);
  }
  \node[below,text=M] at (1.5,0) {$x_2$};
  \draw[black,line width=1pt]
    (4,0) -- (4,0.5) 
    to[out=90,in=-90,looseness=0.7] (1,1.5) -- (1,3)
    to[out=90,in=-90] (2,4) -- (2,5.5)
    to[out=90,in=-90] (1,6.5) -- (1,8)
    to[out=90,in=-90,looseness=0.8] (3,9);
  \fill[black] (4,0.25) circle (3pt);
  \fill[black] (1,2.7) circle (3pt);
  \fill[black] (1,2.4) circle (3pt);
  \fill[black] (2,5.2) circle (3pt);
  \fill[black] (2,4.9) circle (3pt);
  \fill[black] (1,7.7) circle (3pt);
  \fill[black] (1,7.4) circle (3pt);
  \draw[->,line width=1pt] (4.5,4.5) -- (6.5,4.5);
  \begin{scope}[xshift=7cm,yshift=0.5cm]
  \foreach \x in {0.5,1.5,2.5,3.5} {
    \draw[M,line width=1pt] (\x,-0.5) -- (\x,8.5);
  }
  \draw[gray] (-0.3,3)--(4.3,3);
  \draw[gray] (-0.3,7)--(4.3,7);
  \node[below,text=M] at (1.5,-0.5) {$x_2$};
  \draw[black,line width=1pt]
    (4,-0.5) -- (4,1.7) 
    to[out=90,in=-90,looseness=0.7] (2,3) 
    to[out=90,in=-90] (1,4) 
    to[out=90,in=-90] (2,5)
    to[out=90,in=-90] (1,6) -- (1,7)-- (1,7.5)
    to[out=90,in=-90] (3,8.5);
  \fill[black] (4,0.3) circle (3pt);
  \fill[black] (4,0.6) circle (3pt);
  \fill[black] (4,0.9) circle (3pt);
  \fill[black] (4,-0.3) circle (3pt);
  \fill[black] (4,0) circle (3pt);
  \fill[black] (1,7.5) circle (3pt);
  \fill[black] (1,6.5) circle (3pt);
  \end{scope}
  \node at (12,4.5) {$+$};
  \begin{scope}[xshift=13cm,yshift=0.5cm]
  \foreach \x in {0.5,1.5,2.5,3.5} {
    \draw[M,line width=1pt] (\x,-0.5) -- (\x,8.5);
  }
  \draw[gray] (-0.3,3)--(4.3,3);
  \draw[gray] (-0.3,7)--(4.3,7);
  \node[below,text=M] at (1.5,-0.5) {$x_2$};
  \draw[black,line width=1pt]
    (4,-0.5) -- (4,1.7) 
    to[out=90,in=-90,looseness=0.7] (2,3) 
    to[out=90,in=-90] (1,4) 
    to[out=90,in=-90] (2,5)
    to[out=90,in=-90] (1,6) -- (1,7)-- (1,7.5)
    to[out=90,in=-90] (3,8.5);
  \fill[black] (4,-.3) circle (3pt);
  \fill[black] (4,0) circle (3pt);
  \fill[black] (4,0.3) circle (3pt);
  \fill[black] (4,0.6) circle (3pt);
  \fill[black] (4,0.9) circle (3pt);
  \fill[black] (4,1.2) circle (3pt);
  \fill[black] (1,6.5) circle (3pt);
  \end{scope}
  \node at (18,4.5) {$+$};
  \begin{scope}[xshift=19cm,yshift=0.5cm]
  \foreach \x in {0.5,1.5,2.5,3.5} {
    \draw[M,line width=1pt] (\x,-0.5) -- (\x,8.5);
  }
  \draw[gray] (-0.3,3)--(4.3,3);
  \draw[gray] (-0.3,7)--(4.3,7);
  \node[below,text=M] at (1.5,-0.5) {$x_2$};
  \draw[black,line width=1pt]
    (4,-0.5) -- (4,1.7) 
    to[out=90,in=-90,looseness=0.7] (2,3) 
    to[out=90,in=-90] (1,4) 
    to[out=90,in=-90] (2,5)
    to[out=90,in=-90] (1,6)
    to[out=90,in=-90] (2,7)--(2,7.5)
    to[out=90,in=-90] (3,8.5);
  \fill[black] (4,-0.3) circle (3pt);
  \fill[black] (4,0) circle (3pt);
  \fill[black] (4,0.3) circle (3pt);
  \fill[black] (4,0.6) circle (3pt);
  \fill[black] (4,0.9) circle (3pt);
  \fill[black] (4,1.2) circle (3pt);
  \fill[black] (4,1.5) circle (3pt);
  \node[text=gray] at (6,1.3){$A$};
  \node[text=gray] at (6,3){$\otimes$};
  \node[text=gray] at (6,5){$S_4$};
  \node[text=gray] at (6,7){$\otimes$};
  \node[text=gray] at (6,8){$A$};
  \end{scope}
  \node at (31,3){$\phantom{\otimes}$};
  \node at (31,5){$\phantom{S_4}$};
  \end{tikzpicture}
\end{equation}
while the second and the third diagrams each produce four terms:
\begin{equation}
  \begin{tikzpicture}[scale=0.5]
  \foreach \x in {0.5,1.5,2.5,3.5} {
    \draw[M,line width=1pt] (\x,0) -- (\x,9);
  }
  \node[below,text=M] at (2.5,0) {$x_3$};
  \draw[black,line width=1pt]
    (4,0) -- (4,0.5) 
    to[out=90,in=-90,looseness=0.8] (2,1.5) -- (2,3)
    to[out=90,in=-90] (3,4) -- (3,5.5)
    to[out=90,in=-90] (2,6.5) -- (2,8)
    to[out=90,in=-90] (3,9);
  \fill[black] (4,0.25) circle (3pt);
  \fill[black] (2,2.7) circle (3pt);
  \fill[black] (2,2.4) circle (3pt);
  \fill[black] (2,2.1) circle (3pt);
  \fill[black] (3,5.2) circle (3pt);
  \fill[black] (2,7.7) circle (3pt);
  \fill[black] (2,7.4) circle (3pt);
  \fill[black] (2,7.1) circle (3pt);
  \draw[->,line width=1pt] (4.5,4.5) -- (6.5,4.5);
  \begin{scope}[xshift=7cm,yshift=0.5cm]
  \foreach \x in {0.5,1.5,2.5,3.5} {
    \draw[M,line width=1pt] (\x,-0.5) -- (\x,8.5);
  }
  \draw[gray] (-0.3,3)--(4.3,3);
  \draw[gray] (-0.3,7)--(4.3,7);
  \node[below,text=M] at (2.5,-0.5) {$x_3$};
  \draw[black,line width=1pt]
    (4,-0.5) -- (4,2) 
    to[out=90,in=-90] (3,3) 
    to[out=90,in=-90] (2,4) 
    to[out=90,in=-90] (3,5)
    to[out=90,in=-90] (2,6) -- (2,7)-- (2,7.5)
    to[out=90,in=-90] (3,8.5);
  \fill[black] (4,-.3) circle (3pt);
  \fill[black] (4,0) circle (3pt);
  \fill[black] (4,0.3) circle (3pt);
  \fill[black] (4,0.6) circle (3pt);
  \fill[black] (4,0.9) circle (3pt);
  \fill[black] (2,6.5) circle (3pt);
  \fill[black] (2,7.2) circle (3pt);
  \fill[black] (2,7.5) circle (3pt);
  \end{scope}
  \node at (12,4.5) {$+$};
  \begin{scope}[xshift=13cm,yshift=0.5cm]
  \foreach \x in {0.5,1.5,2.5,3.5} {
    \draw[M,line width=1pt] (\x,-0.5) -- (\x,8.5);
  }
  \draw[gray] (-0.3,3)--(4.3,3);
  \draw[gray] (-0.3,7)--(4.3,7);
  \node[below,text=M] at (2.5,-0.5) {$x_3$};
  \draw[black,line width=1pt]
    (4,-0.5) -- (4,2) 
    to[out=90,in=-90,looseness=0.7] (3,3) 
    to[out=90,in=-90] (2,4) 
    to[out=90,in=-90] (3,5)
    to[out=90,in=-90] (2,6) -- (2,7)-- (2,7.5)
    to[out=90,in=-90] (3,8.5);
  \fill[black] (4,-.3) circle (3pt);
  \fill[black] (4,0) circle (3pt);
  \fill[black] (4,0.3) circle (3pt);
  \fill[black] (4,0.6) circle (3pt);
  \fill[black] (4,0.9) circle (3pt);
  \fill[black] (4,1.2) circle (3pt);
  \fill[black] (2,6.5) circle (3pt);
  \fill[black] (2,7.5) circle (3pt);
  \end{scope}
  \node at (18,4.5) {$+$};
  \begin{scope}[xshift=19cm,yshift=0.5cm]
  \foreach \x in {0.5,1.5,2.5,3.5} {
    \draw[M,line width=1pt] (\x,-0.5) -- (\x,8.5);
  }
  \draw[gray] (-0.3,3)--(4.3,3);
  \draw[gray] (-0.3,7)--(4.3,7);
  \node[below,text=M] at (2.5,-0.5) {$x_3$};
  \draw[black,line width=1pt]
    (4,-0.5) -- (4,2) 
    to[out=90,in=-90,looseness=0.7] (3,3) 
    to[out=90,in=-90] (2,4) 
    to[out=90,in=-90] (3,5)
    to[out=90,in=-90] (2,6) -- (2,7)-- (2,7.5)
    to[out=90,in=-90] (3,8.5);
  \fill[black] (4,-.3) circle (3pt);
  \fill[black] (4,0) circle (3pt);
  \fill[black] (4,0.3) circle (3pt);
  \fill[black] (4,0.6) circle (3pt);
  \fill[black] (4,0.9) circle (3pt);
  \fill[black] (4,1.2) circle (3pt);
  \fill[black] (4,1.5) circle (3pt);
  \fill[black] (2,6.5) circle (3pt);
  \end{scope}
  \node at (24,4.5) {$+$};
  \begin{scope}[xshift=25cm,yshift=0.5cm]
  \foreach \x in {0.5,1.5,2.5,3.5} {
    \draw[M,line width=1pt] (\x,-0.5) -- (\x,8.5);
  }
  \draw[gray] (-0.3,3)--(4.3,3);
  \draw[gray] (-0.3,7)--(4.3,7);
  \node[below,text=M] at (2.5,-0.5) {$x_3$};
  \draw[black,line width=1pt]
    (4,-0.5) -- (4,2) 
    to[out=90,in=-90] (3,3) 
    to[out=90,in=-90] (2,4) 
    to[out=90,in=-90] (3,5)
    to[out=90,in=-90] (2,6)
    to[out=90,in=-90] (3,7)--(3,8.5);
  \fill[black] (4,-0.3) circle (3pt);
  \fill[black] (4,0) circle (3pt);
  \fill[black] (4,0.3) circle (3pt);
  \fill[black] (4,0.6) circle (3pt);
  \fill[black] (4,0.9) circle (3pt);
  \fill[black] (4,1.2) circle (3pt);
  \fill[black] (4,1.5) circle (3pt);
  \fill[black] (4,1.8) circle (3pt);
  \node[text=gray] at (6,1.3){$A$};
  \node[text=gray] at (6,3){$\otimes$};
  \node[text=gray] at (6,5){$S_4$};
  \node[text=gray] at (6,7){$\otimes$};
  \node[text=gray] at (6,8){$A$};
  \end{scope}
  \end{tikzpicture}
\end{equation}

\begin{equation}
  \begin{tikzpicture}[scale=0.5]
  \foreach \x in {0.5,1.5,2.5,3.5} {
    \draw[M,line width=1pt] (\x,0) -- (\x,9);
  }
  \node[below,text=M] at (3.5,0) {$x_4$};
  \draw[black,line width=1pt]
    (4,0) -- (4,0.5) 
    to[out=90,in=-90] (3,1.5) -- (3,3)
    to[out=90,in=-90] (4,4) -- (4,5.5)
    to[out=90,in=-90] (3,6.5) -- (3,8)
    -- (3,9);
  \fill[black] (4,0.25) circle (3pt);
  \fill[black] (3,2.7) circle (3pt);
  \fill[black] (3,2.4) circle (3pt);
  \fill[black] (3,2.1) circle (3pt);
  \fill[black] (3,1.8) circle (3pt);
  \fill[black] (3,7.7) circle (3pt);
  \fill[black] (3,7.4) circle (3pt);
  \fill[black] (3,7.1) circle (3pt);
  \fill[black] (3,6.8) circle (3pt);
  \draw[->,line width=1pt] (4.5,4.5) -- (6.5,4.5);
  \begin{scope}[xshift=7cm,yshift=0.5cm]
  \foreach \x in {0.5,1.5,2.5,3.5} {
    \draw[M,line width=1pt] (\x,-0.5) -- (\x,8.5);
  }
  \draw[gray] (-0.3,3)--(4.3,3);
  \draw[gray] (-0.3,7)--(4.3,7);
  \node[below,text=M] at (3.5,-0.5) {$x_4$};
  \draw[black,line width=1pt]
    (4,-0.5) -- (4,3) 
    to[out=90,in=-90] (3,4) 
    to[out=90,in=-90] (4,5)
    to[out=90,in=-90] (3,6) -- (3,8.5);
  \fill[black] (4,-.3) circle (3pt);
  \fill[black] (4,0) circle (3pt);
  \fill[black] (4,0.3) circle (3pt);
  \fill[black] (4,0.6) circle (3pt);
  \fill[black] (4,0.9) circle (3pt);
  \fill[black] (3,6.5) circle (3pt);
  \fill[black] (3,7.2) circle (3pt);
  \fill[black] (3,7.5) circle (3pt);
  \fill[black] (3,7.8) circle (3pt);
  \end{scope}
  \node at (12,4.5) {$+$};
  \begin{scope}[xshift=13cm,yshift=0.5cm]
  \foreach \x in {0.5,1.5,2.5,3.5} {
    \draw[M,line width=1pt] (\x,-0.5) -- (\x,8.5);
  }
  \draw[gray] (-0.3,3)--(4.3,3);
  \draw[gray] (-0.3,7)--(4.3,7);
  \node[below,text=M] at (3.5,-0.5) {$x_4$};
  \draw[black,line width=1pt]
    (4,-0.5) -- (4,3) 
    to[out=90,in=-90] (3,4) 
    to[out=90,in=-90] (4,5)
    to[out=90,in=-90] (3,6) -- (3,8.5);
  \fill[black] (4,-.3) circle (3pt);
  \fill[black] (4,0) circle (3pt);
  \fill[black] (4,0.3) circle (3pt);
  \fill[black] (4,0.6) circle (3pt);
  \fill[black] (4,0.9) circle (3pt);
  \fill[black] (4,1.2) circle (3pt);
  \fill[black] (3,6.5) circle (3pt);
  \fill[black] (3,7.5) circle (3pt);
  \fill[black] (3,7.8) circle (3pt);
  \end{scope}
  \node at (18,4.5) {$+$};
  \begin{scope}[xshift=19cm,yshift=0.5cm]
  \foreach \x in {0.5,1.5,2.5,3.5} {
    \draw[M,line width=1pt] (\x,-0.5) -- (\x,8.5);
  }
  \draw[gray] (-0.3,3)--(4.3,3);
  \draw[gray] (-0.3,7)--(4.3,7);
  \node[below,text=M] at (3.5,-0.5) {$x_4$};
  \draw[black,line width=1pt]
    (4,-0.5) -- (4,3) 
    to[out=90,in=-90] (3,4) 
    to[out=90,in=-90] (4,5)
    to[out=90,in=-90] (3,6) -- (3,8.5);
  \fill[black] (4,-.3) circle (3pt);
  \fill[black] (4,0) circle (3pt);
  \fill[black] (4,0.3) circle (3pt);
  \fill[black] (4,0.6) circle (3pt);
  \fill[black] (4,0.9) circle (3pt);
  \fill[black] (4,1.2) circle (3pt);
  \fill[black] (4,1.5) circle (3pt);
  \fill[black] (3,6.5) circle (3pt);
  \fill[black] (3,7.8) circle (3pt);
  \end{scope}
  \node at (24,4.5) {$+$};
  \begin{scope}[xshift=25cm,yshift=0.5cm]
  \foreach \x in {0.5,1.5,2.5,3.5} {
    \draw[M,line width=1pt] (\x,-0.5) -- (\x,8.5);
  }
  \draw[gray] (-0.3,3)--(4.3,3);
  \draw[gray] (-0.3,7)--(4.3,7);
  \node[below,text=M] at (3.5,-0.5) {$x_4$};
  \draw[black,line width=1pt]
    (4,-0.5) -- (4,3) 
    to[out=90,in=-90] (3,4) 
    to[out=90,in=-90] (4,5)
    to[out=90,in=-90] (3,6) -- (3,8.5);
  \fill[black] (4,-.3) circle (3pt);
  \fill[black] (4,0) circle (3pt);
  \fill[black] (4,0.3) circle (3pt);
  \fill[black] (4,0.6) circle (3pt);
  \fill[black] (4,0.9) circle (3pt);
  \fill[black] (4,1.2) circle (3pt);
  \fill[black] (4,1.5) circle (3pt);
  \fill[black] (4,1.8) circle (3pt);
  \fill[black] (3,6.5) circle (3pt);
  \node[text=gray] at (6,1.3){$A$};
  \node[text=gray] at (6,3){$\otimes$};
  \node[text=gray] at (6,5){$S_4$};
  \node[text=gray] at (6,7){$\otimes$};
  \node[text=gray] at (6,8){$A$};
  \end{scope}
  \end{tikzpicture}\label{eq:184}
\end{equation}

Therefore, the result contains 11 terms:
\begin{align}
G_4&\left( 1\otimes [a_{30}s_0|a_{04}|a_{41}s_{1}^{2}|a_{14}s_4]\otimes 1 \right) 
\\
&=a_{31}s_1\otimes sQ_{1}^{4}\otimes a_{24}s_{4}^{5}+a_{31}\otimes sQ_{1}^{4}\otimes a_{24}s_{4}^{6}+a_{31}\otimes P_{2}^{4}\otimes a_{24}s_{4}^{7}\notag
\\
&\;\;+a_{32}s_{2}^{2}\otimes sQ_{2}^{4}\otimes a_{34}s_{4}^{5}+a_{32}s_2\otimes sQ_{2}^{4}\otimes a_{34}s_{4}^{6}+a_{32}\otimes sQ_{2}^{4}\otimes a_{34}s_{4}^{7}+1\otimes P_{3}^{4}\otimes a_{34}s_{4}^{8}\notag
\\
&\;\;+s_{3}^{3}\otimes sQ_{3}^{4}\otimes s_{4}^{5}+s_{3}^{2}\otimes sQ_{3}^{4}\otimes s_{4}^{6}+s_3\otimes sQ_{3}^{4}\otimes s_{4}^{7}+1\otimes sQ_{3}^{4}\otimes s_{4}^{8}\notag
\end{align}

This example should give a comprehensive recipe for evaluating $G_n$ on arbitrary input.
\end{example}
Now we write down the proof for Theorem~\ref{thm:Gn}

\begin{proof}
  \mathleft
	By induction, the base case is clear 
	\begin{gather}
		G_0\left( a\otimes b \right) =a\cdot \rho _{-1}G_{-1}d_0\left( 1\otimes b \right)=a\otimes b
    \\
    \Longrightarrow G_0=\mathrm{id}_{A\otimes A}
	\end{gather}
which is expected. If this is too abstract, we can go a step further,
\begin{gather}
G_1\left( a\otimes y_1\otimes b \right) =a\cdot \rho _0G_0\bar{\partial}_1\left( 1\otimes y_1\otimes b \right)
\\
=a\cdot\rho _0\left( 1\otimes\pi \left( y_1b \right) -y_1\otimes b \right)
\\
=a\cdot \mathrm{split}_1\left( y_1 \right) \cdot b
\end{gather}
Again, this is expected.
	
Now, suppose the result holds for $G_{n-1}$, then
\begin{gather}
G_n\left( a\otimes [y_n| \cdots | y_1]\otimes b \right) =a\cdot\rho _{n-1}G_{n-1}\bar{\partial}_n\left( 1\otimes [y_n| \cdots| y_1]\otimes b \right) \cdot b
\\
=a\cdot \rho _{n-1}\left( \begin{array}{l}
	\sum_{x\preceq y_n\dots y_2}{\mathrm{split}_{n-1}\left( x \right)}\cdot \pi(y_1b)\\
	+\sum_{i=1}^{n-1}{\sum_{x\preceq y_n\cdots \pi \left( y_{i+1}y_{i} \right) \cdots y_1}{\left( -1 \right) ^i \mathrm{split}_{n-1}\left( x \right)\cdot b}}\\
	+\left( -1 \right) ^n\sum_{x\preceq y_{n-1}\cdots y_{1}}{y_n\cdot \mathrm{split}_{n-1}\left( x \right)}\cdot y_n\\
\end{array} \right)
\\
= a\cdot\left( 1+\sum_{i\geqslant 1}{\left( \gamma _{n-1}\left( \delta _n-\partial _n \right) \right) ^i} \right)\label{eq:193}\\
\qquad\qquad(-1)^n\left( \begin{array}{l}
	\sum_{\begin{smallmatrix}
	qwr\preceq y_n\dots y_2\\
	q\otimes w\otimes r\in P_{n-1}\\
\end{smallmatrix}}{\mathrm{split}_n(qw)\cdot \pi (ry_1b)}\\
	+\sum_{i=1}^{n-1}{\sum_{\begin{smallmatrix}
	qwr\preceq y_n\dots \pi \left( y_{i+1}y_{i} \right) \dots y_1\\
	q\otimes w\otimes r\in P_{n-1}\\
\end{smallmatrix}}{\left( -1 \right) ^i\mathrm{split}_n( qw)\cdot \pi(rb)}}\\
	+\left( -1 \right) ^n\sum_{\begin{smallmatrix}
	qwr\preceq y_{n-1}\dots y_{1}\\
	q\otimes w\otimes r\in P_{n-1}\\
\end{smallmatrix}}{\mathrm{split}_n(\pi(y_nq)w)\cdot \pi \left( rb \right)}\\
\end{array} \right) \notag
\end{gather}

All but the last term in the large parenthesis in~\eqref{eq:193} vanish. 
The reason is as follows. For any other term to be nonvanishing, it must contain a subpath in $S_n$ inside $qw$ with $q \in A$ and $w \in S_{n-1}$. 
Since ambiguities can contain a dot only at the end of a path, $w$ must be of type~I. 
Moreover, because $qwr$ is reduced from a path obtained by concatenating $n-1$ irreducible paths, 
we have $\wp(qwr) \leq n-2$. 
However, since $w \in S_{n-1}$ is a type I ambiguity, $\wp(w) = n-2$, 
which means that the endpoint of $w$ must already lie in the middle of $y_n$. 
Consequently, the beginning of $r$ can neither turn nor carry a dot, 
and therefore no $S_n$ element can be split off.

So we have for now,
\begin{gather}
G_n\left( a\otimes [y_n|\cdots | y_1]\otimes b \right) =\cdots 
\\
= a\cdot\left( 1+\sum_{i\geqslant 1}{\left( \gamma _{n-1}\left( \delta _n-\partial _n \right) \right) ^i} \right) 
\sum_{\begin{smallmatrix}
	qwr\preceq y_{n-1}\dots y_{1}\\
	q\otimes w\otimes r\in P_{n-1}\\
\end{smallmatrix}}{\mathrm{split}_n(\pi(y_nq)w)\cdot \pi \left( rb \right)}
\\
= a\cdot \sum_{\begin{smallmatrix}
	qwr\preceq y_{n-1}\dots y_{1}\\
	q\otimes w\otimes r\in P_{n-1}\\
\end{smallmatrix}}{(\mathrm{split}_n(\pi(y_nq)w)\cdot r)}\cdot b\label{eq:196}
\\
\qquad
+a\cdot\left( 1+\sum_{i\geqslant 1}{\left( \gamma _{n-1}\left( \delta _n-\partial _n \right) \right) ^i} \right) \gamma _{n-1}\left( \delta _n-\partial _n \right)
\sum_{\begin{smallmatrix}
	qwr\preceq y_{n-1}\dots y_{1}\\
	q\otimes w\otimes r\in P_{n-1}\\
\end{smallmatrix}}{\mathrm{split}_n(\pi(y_nq)w)\cdot \pi \left( rb \right)}\notag
\end{gather}

Now we have obtained a specific term in the result:
\[
a \cdot 
\sum_{\begin{smallmatrix}
	qwr \preceq y_{n-1}\dots y_{1}\\
	q\otimes w\otimes r\in P_{n-1}
\end{smallmatrix}}
(\mathrm{split}_n(\pi(y_nq)w)\cdot r)
\cdot b.
\]
Let us analyze the diagrammatic meaning of this term.  
In this case, $w$ must again be a type~I ambiguity.  
The expression $\mathrm{split}_n(\pi(y_nq)w)\cdot r$ means that the path $y_{n-1}\dots y_2y_1$ is reduced to $qwr$ such that $w\in S_{n-1}$ is a type~I ambiguity and $q,r\in A$ are irreducible and then fully reduce $y_nq$ so that $w$, together with the first arrow of $\pi(y_nq)$, forms an ambiguity in $S_n$, which is subsequently split off.  

In terms of Example~\ref{example}, this means that the dots at the top cannot drop down (as they are created after applying $\pi(y_nq)$), whereas all other dots originally at lower positions can freely move to the bottom (as they are produced during the $G_{n-1}$ step). Hence, this term corresponds to the first diagrams among the resulting terms in~\eqref{eq:182}--\eqref{eq:184}.

We are now left with the second term in~\eqref{eq:196}. 
To evaluate this term, we first apply $\delta_n - \partial_n$ to $\sum_{\begin{smallmatrix}
	qwr\preceq y_{n-1}\dots y_{1}\\
	q\otimes w\otimes r\in P_{n-1}\\
\end{smallmatrix}}{\mathrm{split}_n(\pi(y_nq)w)\cdot \pi \left( rb \right)}$, and then apply $\gamma_{n-1}$. Recall that $\delta_n-\partial _n$ acts on the ambiguities by

\begin{gather}
		\left(\delta_n-\partial _n\right)\left( 1\otimes Q^n_i\otimes 1 \right)  =\left( -1 \right) ^n 1\otimes sQ^{n-1}_i\otimes 1\notag
		\\
		\left(\delta_n-\partial _n\right)\left( 1\otimes sQ^n_i \otimes 1 \right) =\left( -1 \right) ^n q_i\otimes sP_{i+1}^{n-1}\otimes 1+\left( -1 \right) ^n 1\otimes Q_i^{n-1}\otimes s_{i+\varsigma}
\end{gather}
where $i+\varsigma$ marks the position of the source of $Q^{n-1}_i$, which is \[\varsigma=\begin{cases}
      1,& n\text{ even},\\
      0,& n\text{ odd}.
    \end{cases}\] 
If we interchange $P\leftrightarrow Q$, $i\leftrightarrow i+1$, $\varsigma$ should be replaced by $-\varsigma$.

Note that $\delta_n - \partial_n$ maps a type~I ambiguity to a type~II ambiguity, 
and a type~II ambiguity to a sum of one type~I term and one type~II term. 
The input \[\sum_{\begin{smallmatrix}
	qwr\preceq y_{n-1}\dots y_{1}\\
	q\otimes w\otimes r\in P_{n-1}\\
\end{smallmatrix}}{\mathrm{split}_n(\pi(y_nq)w)\cdot \pi \left( rb \right)}\]
may contain both type~I and type~II terms.  
However, as observed earlier, only type~I terms survive after applying $\gamma_{n-1}$.  
Since only type~II terms produce type~I terms under the action of $\delta_n - \partial_n$, 
we conclude that only the type~II components in the input are relevant.  
In the diagrammatic language of Example~\ref{example}, 
these correspond precisely to the terms that contain at least one dot in the top position.  

In general, a type~II term takes the form
\[
a_{kj}s_j^\beta \otimes sQ^n_j \otimes a_{j+\varsigma,i}s_i^\alpha,
\]
and applying $\gamma_{n-1}(\delta_n - \partial_n)$ to it yields

\begin{gather}
\gamma _{n-1}( \delta _n-\partial _n ) ( a_{kj}s_{j}^{\beta}\otimes sQ_{j}^{n}\otimes a_{j+\varsigma ,i}s_{i}^{\alpha} ) 
\\
=( -1 ) ^n\gamma _{n-1}( \pi ( a_{kj}s_{j}^{\beta}q_j ) \otimes sP_{j+1}^{n-1}\otimes a_{j+\varsigma ,i}s_{i}^{\alpha}+a_{kj}s_{j}^{\beta}\otimes Q_{j}^{n-1}\otimes a_{j+\varsigma ,i}s_{i}^{\alpha +1} ) 
\\
=( -1 ) ^{2n}( \mathrm{split}_n( \pi ( a_{kj}s_{j}^{\beta}q_j ) \otimes sP_{j+1}^{n-1} ) \cdot a_{j+\varsigma ,i}s_{i}^{\alpha}+\mathrm{split}_n( a_{kj}s_{j}^{\beta}Q_{j}^{n-1} ) \cdot a_{j+\varsigma ,i}s_{i}^{\alpha +1} ) 
\\
=\mathrm{split}_n( a_{kj}s_{j}^{\beta}Q_{j}^{n-1} ) \cdot a_{j+\varsigma ,i}s_{i}^{\alpha +1}
\\
=a_{kj}s_{j}^{\beta-1}\otimes sQ_{j}^{n}\otimes a_{j+\varsigma ,i}s_{i}^{\alpha +1}
\end{gather}
which is, precisely dropping one dot at the top position to the bottom. This gives all the second diagrams in \eqref{eq:182}--\eqref{eq:184}.

The infinite sum $1+\sum_{i\geqslant 1}{\left( \gamma _{n-1}\left( \delta _n-\partial _n \right) \right) ^i}$ 
iteratively carries out this process, each time dropping one dot from the top to the bottom. 
The procedure terminates once no dots remain in the top position.  
At this stage, if there is a turning, we split out a type~I ambiguity, applying $\gamma_{n-1}(\delta_n-\partial_n)$ then yields zero; if there is not, the result vanishes already.

Hence, this construction produces exactly all the possible terms described in Theorem~\ref{thm:Gn} and Example~\ref{example}.

\end{proof}

\section{Components of the $A_\infty$-natural transformations}

Now we are ready to write down the explicit components of the natural transformations. 
Although one could in principle write out all cocycles, for simplicity we present only one representative from each cohomology class.

\subsection{$\mathrm{Nat}(\mathrm{id},\mathrm{id})$}
For a degree $g$ natural transformation $\eta$, $\eta^d$ has a degree shift $[g-d]$. Since in the subcategory $\mathcal{C}$, all morphisms have degree $0$, thus for a degree $g$ natural transformation $\eta\colon\mathrm{id}\Rightarrow \mathrm{id}$, only the $\eta^g$ component is nonzero on morphisms between generators $T_i$.

\subsubsection{Degree 0}
Recall from Section~\ref{sect:HHid} that the cocycles $\phi\in \Hom (P_0,\Delta)$ are parametrized by $\{\epsilon_\ell\}_{\ell\ge 0}\in\Bbbk^{\Z_{\ge 0}}$,
\begin{equation}
  \epsilon_\ell=\phi _{\ell}( e_i ),\quad \ell\ge 0,\;\forall\, 0\le i\le |\mathbf{x}|.
\end{equation}
Since $G_0=id_{A\otimes A}$, the result is:
\begin{theorem}\label{thm:Nat0}
  The degree 0 natural transformations \[\eta_{\{\epsilon_\ell\}}\colon \mathrm{id}\Rightarrow \mathrm{id},
\qquad \{\epsilon_\ell\}_{\ell\ge 0}\in\Bbbk^{\Z_{\ge 0}}\cong\mathrm{HH}^0(\Delta)\] are given by
\begin{gather}
  \eta^0(T_i)=\sum_{\ell=0}^{\infty}\epsilon_{\ell}\, s_i^\ell\\
  \eta^1=\eta^2=\eta^3=\cdots=0
\end{gather}
\end{theorem}

These are the natural transformations in the traditional sense.

\subsubsection{Degree 1}

\begin{theorem}
  The degree 1 natural transformations 
  \[\eta_{\{\sigma_\ell\}}\colon \mathrm{id}\Rightarrow \mathrm{id},
\qquad \{\sigma_\ell\}_{\ell\ge 0}\in\Bbbk ^{\Z_{\ge 0}}\cong \mathrm{HH}^1(\Delta)\] are given by
\begin{gather}
  \eta^1(a_{ji}s^\alpha)=\mathfrak{q}(a_{ji}s^\alpha)\sum_{\ell=0}^{\infty}\frac{1}{2}\sigma_{\ell} \, a_{ji}s^{\alpha+\ell}\\
  \eta^0=\eta^2=\eta^3=\cdots=0
\end{gather}
where
\begin{equation}
  \mathfrak{q}(a_{ji}s^\alpha)\coloneqq 2\alpha+|i-j|.
\end{equation}
\end{theorem}
\begin{remark}
  Here, $\mathfrak{q}(a)$ is the $\mathfrak{q}$-grading of the morphism $a$, which is crucial in Aganagic's categorification of Khovanov cohomology. We have discovered it naturally.
\end{remark}
\begin{proof}

The degree 1 cocycles are
\begin{equation}
\begin{cases}
\phi_0(s_{i})=0, &\\
\phi _{\ell}(s_{i})=:\sigma_{\ell-1}, & \ell> 0,\; \forall \, 0\le i\le |\mathbf{x}|,\\
\phi _{\ell}(q_{i}) + \phi _{\ell}(p_{i+1})=\sigma _{\ell},& \ell\ge 0,\;\forall\, 0\le i <|\mathbf{x}|.\\
\end{cases}
\end{equation}

where $\{\sigma_\ell\}_{\ell\ge 0}\in\Bbbk^{\Z_{\ge 0}}$ parametrize cohomology. For a representative, we can take $\phi _{\ell}(q_{i}) = \phi _{\ell}(p_{i+1})=\sigma _{\ell+1}/2$.

Precomposing with $G_1$, one gets
\begin{align}
  \eta _{\left\{ \sigma _{\ell} \right\}}^{1}\left( a_{ji}s^{\alpha} \right) &=\phi _{\left\{ \sigma _{\ell} \right\}}\comp G_1\left( a_{ji}s^{\alpha} \right) 
\\
&=\phi _{\left\{ \sigma _{\ell} \right\}}\left( \sum_{\beta =0}^{\alpha -1}{a_{ji}s^{\alpha -\beta -1}\otimes s_i\otimes s_{i}^{\beta}}+\sum_{k=i+\mathrm{sgn}(i-j)}^{j}{a_{jk}\otimes \begin{smallmatrix}
	q_k\\
	p_k\\
\end{smallmatrix}\otimes a_{k+\mathrm{sgn}(i-j),i}s_{i}^{\alpha}} \right) \label{eq:211}
\\
&=\sum_{\beta =0}^{\alpha -1}{a_{ji}s^{\alpha -\beta -1}\cdot \sum_{\ell =0}^{\infty}{\sigma _{\ell}}s_{i}^{\ell +1}\cdot s_{i}^{\beta}}+\sum_{k=i\pm 1}^j{a_{jk}\cdot \sum_{\ell =0}^{\infty}{\frac{1}{2}\sigma _{\ell}}\begin{smallmatrix}
	q_k\\
	p_k\\
\end{smallmatrix}s_{k+\mathrm{sgn}(i-j)}^{\ell}\cdot a_{k+\mathrm{sgn}(i-j),i}s_{i}^{\alpha}}
\\
&=\left( 2\alpha +|i-j| \right) \sum_{\ell =0}^{\infty}{\frac{1}{2}\sigma _{\ell}}a_{ji}s_{i}^{\alpha +\ell}
\end{align}
where in \eqref{eq:211}, we choose $q_k$ if $i>j$, or $p_k$ if $i<j$.
\end{proof}

\subsubsection{Degree 2}

\begin{theorem}\label{thm:Nat2}
The degree 2 natural transformations
\[
\eta_{\{\vartheta_i\}}\colon \mathrm{id}\Rightarrow \mathrm{id},
\qquad 
\{\vartheta_1,\dots,\vartheta_{|\mathbf{x}|-1}\}\in \Bbbk^{|\mathbf{x}|-1}
\cong \mathrm{HH}^2(\Delta),
\]
are given by
\begin{gather}
  \eta^2(a_{kj}s^{\beta},a_{ji}s^{\alpha})=\mathfrak{c}(a_{kj}s^{\beta},a_{ji}s^{\alpha})\, a_{ki}s^{\beta+\alpha+\delta(i,j,k)-1}\label{eq:Nat2}\\
  \eta^0=\eta^1=\eta^3=\cdots=0
\end{gather}
where
\begin{equation}\label{eq:coe}
\mathfrak{c} (a_{kj}s^{\beta},a_{ji}s^{\alpha})\coloneqq \begin{cases}
	-\beta \vartheta _j-\sum_{t=1}^{j-k}{2(\beta +t)\vartheta _{j-t}}-2(\beta +j-k)\sum_{t=i}^{k-1}{\vartheta _t}+\left( \beta +j-k \right) \vartheta _i,&		i<k\le j,\\
	\beta \vartheta _j+\sum_{t=1}^{k-j}{2(\beta +t)\vartheta _{j+t}}+2(\beta +k-j)\sum_{t=k+1}^i{\vartheta _t}-\left( \beta +k-j \right) \vartheta _i,&		j\le k<i,\\
	-\beta \vartheta _j-\sum_{t=1}^{j-i}{2(\beta +t)\vartheta _{j-t}}+(\beta +j-i)\vartheta _i,&		k\le i<j,\\
	\beta \vartheta _j+\sum_{t=1}^{i-j}{2(\beta +t)\vartheta _{j+t}}-(\beta +i-j)\vartheta _i,&		j<i\le k,\\
	0,&		\mathrm{else}.\\
\end{cases}
\end{equation}
Here we set $\vartheta_0=\vartheta_{|\mathbf{x}|}=0$.
\end{theorem}

\begin{remark}\label{remark}
  The coefficient $\mathfrak{c}(a_2,a_1)$ is best summarized diagrammatically.
  For the strand diagram representing the path $a_2a_1$, slide the middle of the black strand across one red strand at a time. If there is a subpath of the form $p_i s_i^{\alpha}$ in the middle, it contributes $\alpha(\vartheta_i+\vartheta_{i-1})$ to $\mathfrak{c}$. Moreover, if there is an additional $q_{i-1}$ above $p_i s_i^{\alpha}$, add another $\vartheta_{i-1}$ to the total.
  Summing all contributions, and if $a_2a_1$ bends to the right (i.e. $j>i$), there is an overall minus sign.
\end{remark}
For an illustration, consider $\mathfrak{c}(a_{24}s^2,a_{40}s)$. Following Fig.~\ref{fig:coe}, we get
\begin{equation}
  \mathfrak{c}(a_{24}s^2,a_{40}s)=-4\vartheta_0-8\vartheta_1-8\vartheta_2-6\vartheta_3-2\vartheta_4.
\end{equation}
If there are only 4 punctures, then set $\vartheta_0=\vartheta_4=0$.

\begin{figure}[H]
  \centering
  \begin{tikzpicture}[scale=0.5]
  \node[text=gray] at (-7,1){$A$};
  \node[text=gray] at (-7,2){$\otimes$};
  \node[text=gray] at (-7,3){$A$};
    \foreach \x in {-5.5,-2.5,-3.5,-4.5} {
    \draw[M,line width=1pt] (\x,0) -- (\x,4);
    \draw[M,line width=1pt] (\x,0) -- (\x,4);
    \draw[M,line width=1pt] (\x,0) -- (\x,4);
  }
    \draw[gray] (-6.3,2)--(-1.7,2);
  \draw[black,line width=1pt]
    (-6,0) -- (-6,0.25) 
    to[out=90,in=-90] (-2,2)
    -- (-2,2.5) to[out=90,in=-90] (-4,4);
  \fill[black] (-6,0.2) circle (3pt);
  \fill[black] (-2,2.2) circle (3pt);
  \fill[black] (-2,2.5) circle (3pt);
  \draw[gray] (-1,5)--(9,5);
  \node at (0,6) {$\vartheta_0$}; 
  \node at (2,6) {$\vartheta_1$}; 
  \node at (4,6) {$\vartheta_2$}; 
  \node at (6,6) {$\vartheta_3$}; 
  \node at (8,6) {$\vartheta_4$}; 
  \node at (8,3) {$2$}; 
  \node at (6,3) {$2$}; 
  \node at (6,1) {$1$}; 
  \begin{scope}[yshift=-5cm]
  \foreach \x in {-5.5,-2.5,-3.5,-4.5} {
    \draw[M,line width=1pt] (\x,0) -- (\x,4);
    \draw[M,line width=1pt] (\x,0) -- (\x,4);
    \draw[M,line width=1pt] (\x,0) -- (\x,4);
  }
  \draw[black,line width=1pt]
    (-6,0) -- (-6,0.4) 
    to[out=90,in=-90] (-3,2)
    -- (-3,3) to[out=90,in=-90] (-4,4);
  \fill[black] (-6,0.25) circle (3pt);
  \fill[black] (-3,2.2) circle (3pt);
  \fill[black] (-3,2.5) circle (3pt);
  \fill[black] (-3,2.8) circle (3pt);
  \draw[gray] (-1,4.5)--(9,4.5);
  \node at (6,3) {$3$}; 
  \node at (4,3) {$3$}; 
  \node at (4,1) {$1$}; 
  \end{scope}
  \begin{scope}[yshift=-10cm]
  \foreach \x in {-5.5,-2.5,-3.5,-4.5} {
    \draw[M,line width=1pt] (\x,0) -- (\x,4);
    \draw[M,line width=1pt] (\x,0) -- (\x,4);
    \draw[M,line width=1pt] (\x,0) -- (\x,4);
  }
  \draw[black,line width=1pt]
    (-6,0) -- (-6,0.5) 
    to[out=90,in=-90] (-4,2)-- (-4,4);
  \fill[black] (-6,0.25) circle (3pt);
  \fill[black] (-4,2.2) circle (3pt);
  \fill[black] (-4,2.5) circle (3pt);
  \fill[black] (-4,2.8) circle (3pt);
  \fill[black] (-4,3.1) circle (3pt);
  \draw[gray] (-1,4.5)--(9,4.5);
  \node at (4,3) {$4$}; 
  \node at (2,3) {$4$}; 
  \end{scope}
  \begin{scope}[yshift=-15cm]
  \foreach \x in {-5.5,-2.5,-3.5,-4.5} {
    \draw[M,line width=1pt] (\x,0) -- (\x,4);
    \draw[M,line width=1pt] (\x,0) -- (\x,4);
    \draw[M,line width=1pt] (\x,0) -- (\x,4);
  }
  \draw[black,line width=1pt]
    (-6,0) -- (-6,0.5) 
    to[out=90,in=-90] (-5,1.5) -- (-5,3) to[out=90,in=-90] (-4,4);
  \fill[black] (-6,0.25) circle (3pt);
  \fill[black] (-5,2.2) circle (3pt);
  \fill[black] (-5,2.5) circle (3pt);
  \fill[black] (-5,2.8) circle (3pt);
  \fill[black] (-5,1.9) circle (3pt);
  \draw[gray] (-1,4.5)--(9,4.5);
  \node at (0,3) {$4$}; 
  \node at (2,3) {$4$}; 
  \draw[gray] (-6,-1)--(9,-1);
  \node at (0,-2.5) {$4$}; 
  \node at (2,-2.5) {$8$}; 
  \node at (4,-2.5) {$8$}; 
  \node at (6,-2.5) {$6$}; 
  \node at (8,-2.5) {$2$}; 
  \end{scope}
  \end{tikzpicture}
  \caption{Evaluation of coefficient $\mathfrak{c}$.}
  \label{fig:coe}
\end{figure}

Below is the proof of Theorem~\ref{thm:Nat2}.
\begin{proof}
  From \eqref{eq:148},
  \begin{align}
    &-J_{\ell}(s_0)=V^{2,\ell}_{0,Q}\\
    &-J_{\ell}(s_i)=V^{2,\ell}_{i-1,P}+V^{2,\ell}_{i,Q},\quad 0<i<|\mathbf{x}|\\
    &-J_{\ell}(s_{|\mathbf{x}|})=V^{2,\ell}_{|\mathbf{x}|-1,P}
  \end{align}
  where the vectors $J_\ell$ span the image and the vectors $V^{2,\ell}$ span the kernel.

  Thus, we can take the representatives of the Hochschild cohomology as
  \begin{equation}
    \mathfrak{V}_i\coloneqq V^{2,\ell}_{i-1,P}-V^{2,\ell}_{i,Q}=[0,0,\dots,1,0,1,-1,1,-1,0,-1,\dots,0,0]^\mathsf{T},\quad 1\le i \le |\mathbf{x}|-1.
  \end{equation}

  A general representative of the cohomology class is given by a linear combination of $\mathfrak{V}_i$,
  \begin{equation}
    \mathfrak{V}=\sum_{i=1}^{|\mathbf{x}|-1}\vartheta_i\mathfrak{V}_i
  \end{equation}

  More concretely, we have
  \begin{align}
\left[ \begin{array}{l}
	\vdots\\
	\phi _0(s_{i-1}q_{i-1})\\
	\phi _0(q_{i-1}p_i)\\
	\phi _0(p_iq_{i-1})\\
	\phi _0(s_ip_i)\\
	\phi _0(s_iq_i)\\
	\phi _0(q_ip_{i+1})\\
	\phi _0(p_{i+1}q_i)\\
	\phi _0(s_{i+1}p_{i+1})\\
	\vdots\\
\end{array} \right] &=\left[ \begin{matrix}
	\cdots&		\mathfrak{V} _{i-1}&		\mathfrak{V} _i&		\mathfrak{V} _{i+1}&		\cdots\\
\end{matrix} \right] \left[ \begin{array}{c}
	\vdots\\
	\vartheta _{i-1}\\
	\vartheta _i\\
	\vartheta _{i+1}\\
	\vdots\\
\end{array} \right] 
\\
&=\left[ \begin{matrix}
	\ddots&		&		&		&		\\
	&		1&		1&		&		\\
	&		-1&		0&		&		\\
	&		0&		1&		&		\\
	&		-1&		-1&		&		\\
	&		&		1&		1&		\\
	&		&		-1&		0&		\\
	&		&		0&		1&		\\
	&		&		-1&		-1&		\\
	&		&		&		&		\ddots\\
\end{matrix} \right] \left[ \begin{array}{c}
	\vdots\\
	\vartheta _{i-1}\\
	\vartheta _i\\
	\vartheta _{i+1}\\
	\vdots\\
\end{array} \right] 
  \end{align}
which are
\begin{align}
&\phi _0(q_ip_{i+1})=-\vartheta _i
\\
&\phi _0(p_{i+1}q_i)=\vartheta _{i+1}
\\
&\phi _0(s_iq_i)=\vartheta _i+\vartheta _{i+1}
\\
&\phi _0(s_{i+1}p_{i+1})=-\vartheta _i-\vartheta _{i+1}
\end{align}

Note that all the cohomologically nontrivial coefficients are $\phi_0\colon S_2\to \Bbbk$. Each element in $S_2$, after reduction, corresponds to a strand with one dot. Thus, $\eta$ removes one dot, yielding the $a_{ki}s^{\beta+\alpha+\delta(i,j,k)-1}$ in \eqref{eq:Nat2}. Combining this with the map $G_2$, one obtains the diagrammatic description of the coefficients given precisely in~\eqref{eq:coe} and Remark~\ref{remark}.
\end{proof}

\begin{remark}
  One can directly verify the $A_\infty$-naturality condition~\eqref{eq:mu1-natural} on the results of Theorems~\ref{thm:Nat0}--\ref{thm:Nat2}.  
  As an example, consider a general degree $2$ natural transformation $\eta$ acting on $a_{13}s\otimes a_{30}\otimes a_{02}$.  
  In this case, equation~\eqref{eq:mu1-natural} reduces to
  One can directly verify the condition for $A_\infty$-natural transformations \eqref{eq:mu1-natural} on the results Theorems~\ref{thm:Nat0}--\ref{thm:Nat2}. For example, consider a general degree 2 natural transformation $\eta$ acting on $a_{13}s\otimes a_{30}\otimes a_{02}$. Here \eqref{eq:mu1-natural} reduces to
  \begin{align}
    &\left( -\mu ^2\left( \eta ^2\otimes \mathrm{id} \right) +\mu ^2\left( \mathrm{id}\otimes \eta ^2 \right) -\eta ^2\left( \mu ^2\otimes \mathrm{id} \right) +\eta ^2\left( \mathrm{id}\otimes \mu ^2 \right) \right) \left( a_{13}s\otimes a_{30}\otimes a_{02} \right) 
    \\
    &=-\eta ^2\left( a_{13}s,a_{30} \right) \cdot a_{02}+a_{13}s\cdot \eta ^2\left( a_{30},a_{02} \right) -\eta ^2\left( a_{10}s^3,a_{02} \right) +\eta ^2\left( a_{13}s,a_{32}s^2 \right) 
    \\
    &=\left( 3\vartheta _0+6\vartheta _1+4\vartheta _2+\vartheta _3 \right) a_{10}s^2\cdot a_{02}+a_{13}s\cdot \left( 2\vartheta _1+2\vartheta _2 \right) a_{32}s\\
    &\qquad -\left( 3\vartheta _0+8\vartheta _1+4\vartheta _2 \right) a_{12}s^3-\left( 2\vartheta _2+\vartheta _3 \right) a_{12}s^3\notag
    \\
    &=\left( 
      3\vartheta _0+6\vartheta _1+4\vartheta _2+\vartheta _3+2\vartheta _1+2\vartheta _2-3\vartheta _0-8\vartheta _1-4\vartheta _2-2\vartheta _2-\vartheta _3\right) a_{12}s^3
    \\
    &=0
  \end{align}
\end{remark}

\subsection{$\mathrm{Nat}(\mathrm{id},\bim)$}

Away from the punctures, the natural transformations to $\bim$ looks much similar to the ones to $\mathrm{id}$.
By degree reasons, for $|\eta|=g$, only $\eta^g$ and $\eta^{g+1}$ components can be nonzero acting on chains of morphisms between geenrators. We will not give all the proofs for conciseness.

In all expressions of the natural transformation $\mathrm{id} \Rightarrow \bim$, for a similar reason as in Remark~\ref{rmk:beta}, the choice of the $[\tfrac12,\tfrac12]$-morphism into $\bim T_i$ in $\eta^g$ is merely conventional. One may choose a different partition; accordingly, the component $\eta^{g+1}$ will change in a compatible way.

\subsubsection{Degree 0}

\begin{theorem}
  The degree 0 natural transformations
  \[\eta_{\{\epsilon_\ell\}}\colon \mathrm{id}\Rightarrow \bim,
  \qquad \{\epsilon_\ell\}_{\ell\ge 1}\in\Bbbk^{\Z_{\ge 1}}\cong\mathrm{HH}^0(\mathfrak{B}_{i^-})\] are given by
\begin{align}
  &\eta^0(T_j)=\begin{cases}
    \sum_{\ell=1}^{\infty}\epsilon_{\ell}\, s_j^\ell,& j\ne i,\\
    \sum_{\ell=1}^{\infty}\epsilon_{\ell}\,\begin{bmatrix}
    \frac12 q_{i-1}s^{\ell-1}, &\frac12 p_{i+1}s^{\ell-1}
  \end{bmatrix}, & j=i,\\
  \end{cases}\\
  &\eta^1(a_{kj}s^{\alpha})=\begin{cases}
    0, & k\ne i,\\
    \mathrm{sgn}(i-j)
    \sum_{\ell=1}^{\infty}{\frac12 \epsilon_{\ell}\, a_{ij}s^{\alpha+\ell-1}}, & k=i,\\
  \end{cases}\\
  &\eta^2=\eta^3=\eta^4=\cdots=0
\end{align}
\end{theorem}
For example, we can check the naturality condition for $\eta\colon \mathrm{id}\Rightarrow \beta_{2^-}$ with only $\epsilon_2=2$ on morphism $a\coloneqq a_{24}s\in \Hom(T_4,T_2)$,

\begin{figure}[H]
  \centering
  \begin{tikzpicture}[scale=1.3]
  \node at (0,0) {$T_4$};
  \node at (0,4) {$T_2$};
  \node at (4,0) {$T_4$};
  \node at (4,2.8) {$T_3$};
  \node at (4,4.2) {$T_1$};
  \node at (2.8,3.5) {$T_2$};
  \node at (4,3.5) {$\oplus$};
  \draw[gray] (4.3,2.5)--(4.5,2.5)--(4.5,4.5)--(4.3,4.5);
  \draw[gray] (2.6,2.5)--(2.4,2.5)--(2.4,4.5)--(2.6,4.5);
  \node[gray] at (5.5,4.2) {\scriptsize $\beta_{2^-}T_2$};
  \node[gray] at (5.5,0) {\scriptsize $\beta_{2^-}T_4$};
  \draw[->,shorten <=8pt, shorten >=8pt, line width=0.7pt] (2.8,3.5) --  (4,4.2);
  \draw[->,shorten <=8pt, shorten >=8pt, line width=0.7pt] (2.8,3.5) --  (4,2.8);
  \draw[->,shorten <=8pt, shorten >=8pt, line width=0.7pt] (0,0) --node[right,midway] {\scriptsize $a$}  (0,4);
  \node[left] at (0,2) {\tikz[scale=1.5, line width=0.7pt]{
  \foreach \x in {0.05,0.15,0.25,0.35} {
      \draw[M] (\x,-0.13) -- (\x,0.13);
  }
  \fill[black] (0.4,-0.08) circle (0.8pt);
  \draw[black]
      (0.4,-0.13) -- (0.4,-0.07)
      to[out=90,in=-90](0.2,0.13);}};

  \node[JO] at (1,1.3) {\scriptsize $-$\tikz[scale=1.5, line width=0.7pt]{
  \foreach \x in {0.05,0.15,0.25,0.35} {
      \draw[M] (\x,-0.13) -- (\x,0.13);
  }
  \fill[black] (0.4,-0.09) circle (0.8pt);
  \fill[black] (0.4,-0.01) circle (0.8pt);
  \draw[black]
      (0.4,-0.13) -- (0.4,-0)
      to[out=90,in=-90](0.2,0.13);}};
  
  \draw[->, JK,shorten <=8pt, shorten >=8pt, line width=0.7pt] (0,0) --node[below,midway]{\scriptsize$\eta^0(T_4)$}  (4,0);
  \node[JK, above] at (2,0) {\scriptsize $2$ \tikz[scale=1.5, line width=0.7pt]{
  \foreach \x in {0.05,0.15,0.25,0.35} {
      \draw[M] (\x,-0.13) -- (\x,0.13);
  }
  \fill[black] (0.4,-0.04) circle (0.8pt);
  \fill[black] (0.4,0.04) circle (0.8pt);
  \draw[black]
      (0.4,-0.13) -- (0.4,0.13);}};

  \draw[->, JK,shorten <=8pt, shorten >=8pt, line width=0.7pt] (0,4) to[out=30,in=150] node[above,midway] {\scriptsize$\eta^0(T_2)$} (4,4.2);
  \draw[->, JK,shorten <=8pt, shorten >=8pt, line width=0.7pt] (0,4) to[out=-60,in=-150] node[above,midway] {\scriptsize$\eta^0(T_2)$} (4,2.8);

  \draw[->, shorten <=8pt, shorten >=8pt, line width=0.7pt] (4,0) --node[right,midway]{\scriptsize $\beta^1_{2^-}(a)$} (4,2.8);
  \node[left] at (4,1.5) {\tikz[scale=1.5, line width=0.7pt]{
  \foreach \x in {0.05,0.15,0.25,0.35} {
      \draw[M] (\x,-0.13) -- (\x,0.13);
  }
  \fill[black] (0.4,-0.08) circle (0.8pt);
  \draw[black]
      (0.4,-0.13) -- (0.4,-0.05)
      to[out=90,in=-90](0.3,0.11)--(0.3,0.13);}};

  \draw[->, JO, shorten <=8pt, shorten >=8pt, line width=0.7pt] (0,0) to[out=70, in=150] node[left,midway] {\scriptsize$\eta^1(a)$} (2.8,3.5);

  \draw[->, shorten <=8pt, shorten >=8pt, line width=0.7pt] (4,0) -- (4,2.8);
  \node[JK] at (1.8,4.25) {\tikz[scale=1.5, line width=0.7pt]{
  \foreach \x in {0.05,0.15,0.25,0.35} {
      \draw[M] (\x,-0.13) -- (\x,0.13);
  }
  \fill[black] (0.2,-0.08) circle (0.8pt);
  \draw[black]
      (0.2,-0.13) -- (0.2,-0.05)
      to[out=90,in=-90](0.1,0.11)--(0.1,0.13);}};

    \draw[->, shorten <=8pt, shorten >=8pt, line width=0.7pt] (4,0) -- (4,2.8);
  \node[JK] at (2,2) { \tikz[scale=1.5, line width=0.7pt]{
  \foreach \x in {0.05,0.15,0.25,0.35} {
      \draw[M] (\x,-0.13) -- (\x,0.13);
  }
  \fill[black] (0.2,-0.08) circle (0.8pt);
  \draw[black]
      (0.2,-0.13) -- (0.2,-0.05)
      to[out=90,in=-90](0.3,0.11)--(0.3,0.13);}};

  \node[above, xshift=-3pt] at (3.25,3.75) {\tikz[scale=1.5, line width=0.7pt]{
  \foreach \x in {0.05,0.15,0.25,0.35} {
      \draw[M] (\x,-0.1) -- (\x,0.1);
  }
  \draw[black]
      (0.2,-0.1) -- (0.2,-0.08)
      to[out=90,in=-90] (0.1,0.08)
      -- (0.1,0.1);}};
  \node[below, xshift=-9pt] at (3.25,3.25) {\scriptsize$-$ \tikz[scale=1.5, line width=0.7pt]{
  \foreach \x in {0.05,0.15,0.25,0.35} {
      \draw[M] (\x,-0.1) -- (\x,0.1);
  }
  \draw[black]
      (0.2,-0.1) -- (0.2,-0.08)
      to[out=90,in=-90] (0.3,0.08)
      -- (0.3,0.1);}};
  \end{tikzpicture}
  \caption{Naturality condition for $\eta\in \mathrm{Nat}^0(\mathrm{id},\bim)$}
\end{figure}

We can see that
\begin{equation}
  \mu^2_{\delta}(\beta^1(a),\eta^0(T_4))-\mu^2_{\delta}(\eta^0(T_2),a)+\mu^1_{\delta}(\eta^1(a))=0.
\end{equation}

For the length 2 input, one have for example,

\begin{figure}[H]
  \centering
  \begin{tikzpicture}[scale=1.3]
  \node at (4,-1) {$T_1$};
  \node at (0,-1) {$T_1$};
  \node at (-1.2,1) {$T_4$};
  \node at (0,4) {$T_2$};
  \node at (2.8,1) {$T_4$};
  \node at (4,2.8) {$T_3$};
  \node at (4,4.2) {$T_1$};
  \node at (2.8,3.5) {$T_2$};
  \node at (4,3.5) {$\oplus$};
  \draw[gray] (4.3,2.5)--(4.5,2.5)--(4.5,4.5)--(4.3,4.5);
  \draw[gray] (2.6,2.5)--(2.4,2.5)--(2.4,4.5)--(2.6,4.5);
  \node[gray] at (5.5,4.2) {\scriptsize $\beta_{2^-}T_2$};

  \draw[->, shorten <=8pt, shorten >=8pt, line width=0.7pt] (4,-1) to[out=85,in=-85]node[right,pos=0.25]{\scriptsize $\beta^2_{2^-}(a_2,a_1)$} (2.8,3.5);

  \draw[->,shorten <=8pt, shorten >=8pt, line width=0.7pt] (2.8,3.5) --  (4,4.2);
  \draw[->,shorten <=8pt, shorten >=8pt, line width=0.7pt] (2.8,3.5) -- node[below,midway]{\scriptsize$-$} (4,2.8);
  \draw[->,shorten <=8pt, shorten >=8pt, line width=0.7pt] (-1.2,1) --node[left,midway] {\scriptsize $a_2$}  (0,4);
  
  \draw[->, shorten <=8pt, shorten >=8pt, line width=0.7pt] (0,-1) --node[left,midway]{\scriptsize$a_1$}  (-1.2,1);
  \draw[->,dashed, shorten <=8pt, shorten >=8pt, line width=0.7pt] (0,-1) --node[right,pos=0.8]{\scriptsize$a_2 a_1$}  (0,4);
  \draw[->, shorten <=8pt, shorten >=8pt, line width=0.7pt] (0,-1) --node[left,midway]{\scriptsize$a_1$}  (-1.2,1);
  \draw[->, JK,shorten <=8pt, shorten >=8pt, line width=0.7pt] (0,-1) --node[below,midway]{\scriptsize$\eta^0(T_1)$}  (4,-1);

  \draw[->, shorten <=8pt, shorten >=8pt, line width=0.7pt] (2.8,1) --node[right,midway]{\scriptsize $\beta^1_{2^-}(a_2)$} (4,2.8);

  \draw[->, gray!50, shorten <=8pt, shorten >=8pt, line width=0.7pt] (0,-1) -- node[below,pos=0.6, xshift=10pt] {\scriptsize$\eta^1(a_1)=0$} (2.8,1);
  \draw[->, JO, shorten <=8pt, shorten >=8pt, line width=0.7pt] (-1.2,1) -- node[left,pos=0.6] {\scriptsize$\eta^1(a_2)$} (2.8,3.5);
  \draw[->, JO,dashed, shorten <=8pt, shorten >=8pt, line width=0.7pt] (0,-1) -- node[left,midway] {\scriptsize$\eta^1(a_2 a_1)$} (2.8,3.5);

  \end{tikzpicture}
  \caption{Naturality condition for $\eta\in \mathrm{Nat}^0(\mathrm{id},\bim)$}
\end{figure}
One can check directly that
\begin{equation}
  -\mu^2_{\delta}(\eta^1(a_2), a_1) -\mu^2_{\delta}(\bim^1(a_2),\eta^1(a_1)) +\mu^2_{\delta}(\bim^2(a_2,a_1), \eta^0(X_0))+\eta^1(\mu^2(a_2,a_1))=0
\end{equation}

\begin{remark}
Even when considering degree 0 natural transformations, we cannot restrict ourselves to the classical notion in ordinary categories, because nontrivial higher components (e.g. $\eta^1$) arise in the $\mathrm{id}\Rightarrow\bim$ case.
\end{remark}

\subsubsection{Degree 1}
\begin{theorem}
  The degree 1 natural transformations
  \[\eta_{\{\sigma_\ell\}}\colon \mathrm{id}\Rightarrow \bim,
  \qquad \{\sigma_\ell\}_{\ell\ge 0}\in\Bbbk^{\Z_{\ge 0}}\cong\mathrm{HH}^1(\mathfrak{B}_{i^-})\] are given by
\begin{align}
  &  \eta^1(a_{kj}s^\alpha)=\begin{cases}
    \mathfrak{q}(a_{kj}s^\alpha)\sum_{\ell=0}^{\infty}\frac{1}{2}\sigma_{\ell} \, a_{kj}s^{\alpha+\ell},& k\ne i,
    \\
    \mathfrak{q}(a_{ij}s^\alpha)\sum_{\ell=0}^{\infty}\frac{1}{2}\sigma_{\ell} \, a_{i+\mathrm{sgn}(j-i),j}s^{\alpha+\ell},& k=i\ne j
    \\
    \mathfrak{q}(s^\alpha)\sum_{\ell=0}^{\infty}\frac{1}{2}\sigma_{\ell} \, \begin{bmatrix}
    \frac12 q_{i-1}s^{\alpha+\ell-1}, &\frac12 p_{i+1}s^{\alpha+\ell-1}
  \end{bmatrix},& k=j=i.
  \end{cases}
  \\
  &\eta^2(a_{lk}s^\beta,a_{kj}s^\alpha)=\begin{cases}
    0, & l\ne i,\\
    \mathfrak{Q}(a_{ik}s^\beta,a_{kj}s^\alpha)\sum_{\ell=0}^{\infty}\frac12 \sigma_\ell a_{ij}s^{\beta+\alpha+\delta(j,k,i)+\ell-1}, & l=i.
  \end{cases}\\
  &\eta^0=\eta^3=\eta^4=\cdots=0
\end{align}
where
\begin{align}
\mathfrak{Q}(a_2,a_1)&=\mathfrak{Q}(a_{ik}s^\beta,a_{kj}s^\alpha)\\
&=\begin{cases}
  \mathrm{sgn}(i-k)\mathfrak{q}(a_2a_1),& k<i<j\text{ or }j<i<k,
  \\
  \mathrm{sgn}(j-i)\frac{1}{2} \mathfrak{q}(a_2),& k=i,
  \\
  \mathrm{sgn}(i-k)\frac{1}{2} \mathfrak{q}(a_2a_1),& j=i,
  \\
  0,&\text{else}.
\end{cases}\label{eq:242}
\end{align}
one can check if $\beta+\alpha+\delta(j,k,i)=0$, then $\mathfrak{Q}=0$, thus the sum in \eqref{eq:242} will not contain the undefined term $s^{-1}$.
\end{theorem}

One can verify the naturality conditions on inputs of length $2$ and $3$ by writing down the corresponding commutative diagrams. For example,

\begin{figure}[H]
  \centering
  \begin{tikzpicture}[scale=1.3]
  \node at (0,-1) {$T_4$};
  \node at (-1.2,1) {$T_1$};
  \node at (0,4) {$T_2$};
  \node at (2.8,-0.2) {$T_1$};
  \node at (4,2.8) {$T_3$};
  \node at (4,4.2) {$T_1$};
  \node at (2.8,3.5) {$T_2$};
  \node at (4,3.5) {$\oplus$};
  \draw[gray] (4.3,2.5)--(4.5,2.5)--(4.5,4.5)--(4.3,4.5);
  \draw[gray] (2.6,2.5)--(2.4,2.5)--(2.4,4.5)--(2.6,4.5);
  \node[gray] at (5.5,4.2) {\scriptsize $\beta_{2^-}T_2$};
  \node[gray] at (3.3,-0.7) {\scriptsize $\beta_{2^-}T_1$};
  \draw[->, JK, shorten <=8pt, shorten >=8pt, line width=0.7pt] (-1.2,1) to[out=30,in=160] node[left,pos=0.6,yshift=5pt] {\scriptsize$\begin{smallmatrix}
    \eta^1(a_2)\\
    \sim\mathfrak{q}(a_2)
  \end{smallmatrix}$}
  (4,4.2);
  \draw[->, JK, dashed, shorten <=8pt, shorten >=8pt, line width=0.7pt] (0,-1) -- node[right,pos=0.6] {\scriptsize$\begin{smallmatrix}
    \eta^1(a_2a_1)\\
    \sim\mathfrak{q}(a_2a_1)
  \end{smallmatrix}$}
  (4,2.8);
  \draw[->,shorten <=8pt, shorten >=8pt, line width=0.7pt] (2.8,3.5) --  (4,4.2);
  \draw[->,shorten <=8pt, shorten >=8pt, line width=0.7pt] (2.8,3.5) -- node[below,midway]{\scriptsize$-$} (4,2.8);
  \draw[->,shorten <=8pt, shorten >=8pt, line width=0.7pt] (-1.2,1) --node[left,midway] {\scriptsize $a_2$}  (0,4);
  \draw[->, shorten <=8pt, shorten >=8pt, line width=0.7pt] (0,-1) --node[left,midway]{\scriptsize$a_1$}  (-1.2,1);
  \draw[->, dashed, shorten <=8pt, shorten >=8pt, line width=0.7pt] (0,-1) --node[right,pos=0.8]{\scriptsize$a_2 a_1$}  (0,4);
  \draw[->, shorten <=8pt, shorten >=8pt, line width=0.7pt] (0,-1) --node[left,midway]{\scriptsize$a_1$}  (-1.2,1);
  \draw[->, shorten <=8pt, shorten >=8pt, line width=0.7pt] (2.8,-0.2) to[out=15.9,in=-20]node[right,midway]{\scriptsize $\beta^1_{2^-}(a_2)$} (4,4.2);
  \draw[->, JK, shorten <=8pt, shorten >=8pt, line width=0.7pt] (0,-1) -- node[below,midway] {\scriptsize $\begin{smallmatrix}
    \eta^1(a_1)\\
    \sim \mathfrak{q}(a_1)
  \end{smallmatrix}$}
  (2.8,-0.2);
  \draw[->, JO, shorten <=8pt, shorten >=8pt, line width=0.7pt] (0,-1) -- node[left,midway] {\scriptsize$\begin{smallmatrix}
    \eta^2(a_2, a_1)\\
    \sim\mathfrak{q}(a_2a_1)
  \end{smallmatrix}$}
   (2.8,3.5);

  \end{tikzpicture}
  \caption{Naturality condition for $\eta\in \mathrm{Nat}^1(\mathrm{id},\bim)$}
  \label{fig:Nat1}
\end{figure}

In Fig.~\ref{fig:Nat1}, we only write the coefficient before each diagram. We can check the naturality condition \eqref{eq:mu1-natural}, which is
\begin{gather}
  \mu^2_{\delta}(\eta^1(a_2),a_1)+\mu^2_{\delta}(\bim^1(a_2),\eta^1(a_1))+\mu^1_{\delta}(\eta^2(a_2,a_1))-\eta^1(\mu^2(a_2,a_1))=0,
  \\
  \implies \eta^1(a_2)\cdot a_1+\bim^1(a_2)\cdot \eta^1(a_1)=\delta\eta^2(a_2,a_1)+\eta^2(a_2,a_1)\delta+\eta^1(a_2\cdot a_1).
\end{gather}

The length~$3$ case is straightforward to verify, so we omit it here.

\subsubsection{Degree 2}

\begin{theorem}
The degree 2 natural transformations
\[
\eta_{\{\vartheta_i\}}\colon \mathrm{id}\Rightarrow \bim,
\qquad 
\{\vartheta_1,\dots,\vartheta_{i-1},\vartheta_{i+1},\dots,\vartheta_{|\mathbf{x}|-1}\}\in \Bbbk^{|\mathbf{x}|-2}
\cong \mathrm{HH}^2(\mathfrak{B}_{i^-}),
\]
are given by
\begin{align}
  &\eta^2(a_{lk}s^{\beta},a_{kj}s^{\alpha})
  =\begin{cases}
    \mathfrak{c}(a_{lk}s^{\beta},a_{kj}s^{\alpha})\, a_{lj}s^{\beta+\alpha+\delta(j,k,l)-1},& l\ne i,\\
    \mathfrak{c}(a_{ik}s^{\beta},a_{kj}s^{\alpha})\, a_{i+\mathrm{sgn}(j-i),j}s^{\beta+\alpha+\delta(j,k,i)-1},& l=i\ne j\\
    \mathfrak{c}(a_{ik}s^{\beta},a_{ki}s^{\alpha})\, \begin{bmatrix}
    \frac12 q_{i-1}s^{\beta+\alpha+|i-k|-2}, &\frac12 p_{i+1}s^{\beta+\alpha+|i-k|-2}
  \end{bmatrix},& l=j=i.
  \end{cases}\\
  &\eta ^3\left( a_{ml}s^{\gamma},a_{lk}s^{\beta},a_{kj}s^{\alpha} \right) =\begin{cases}
	0,&		m\ne i,\\
	\mathfrak{C} \left( a_{il}s^{\gamma},a_{lk}s^{\beta},a_{kj}s^{\alpha} \right) a_{ij}s^{\gamma +\beta +\alpha +\delta \left( j,k,l,i \right) -2},&		m=i.\\
\end{cases}\\
  &\eta^0=\eta^1=\eta^4=\cdots=0
\end{align}
where we set $\vartheta_0=\vartheta_i=\vartheta_{|\mathbf{x}|}=0$ in the expression \eqref{eq:coe} for $\mathfrak{c}(a_2,a_1)$, and 
\begin{align}
\delta \left( i,j,k,l \right)& \coloneqq \delta \left( i,j,l \right) +\delta \left( j,k,l \right) =\delta \left( i,k,l \right) +\delta \left( i,j,k \right) 
\\
&=\frac{1}{2}\left( |i-j|+|j-k|+|k-l|-|i-l| \right) 
\end{align}
stands for the number of dots in $a_{lk}a_{kj}a_{ji}$, and the coefficient $\mathfrak{C}$ is given by
\begin{align}
  \mathfrak{C} &\left( a_3,a_2,a_1 \right) =\mathfrak{C} \left( a_{il}s^{\gamma},a_{lk}s^{\beta},a_{kj}s^{\alpha} \right) 
\\
&=\begin{cases}
	\frac{1}{2}\left( \left( -\mathrm{sgn} \left( j-i \right) +\mathrm{sgn} \left( k-i \right) \right) \mathfrak{c} \left( a_3,a_2 \right) +\left( \mathrm{sgn} \left( j-i \right) -\mathrm{sgn} \left( l-i \right) \right) \mathfrak{c} \left( a_2,a_1 \right) \right) ,&		l\ne i,\\
	\frac{1}{2}\left( \left( -\mathrm{sgn} \left( j-i \right) +\mathrm{sgn} \left( k-i \right) \right) \mathfrak{c} \left( a_3,a_2 \right) \right) ,&		l=i.\\
\end{cases}
\end{align}

\end{theorem}

\bibliographystyle{plain}
\bibliography{refs}

\end{document}